\documentclass[10pt,oneside,reqno]{amsart}
\usepackage{hyperref}
\usepackage{amsmath}
\usepackage{amssymb}
\usepackage{amsxtra}
\usepackage{amsthm}

\allowdisplaybreaks[4]
\usepackage{stmaryrd}
\usepackage{bm}
\usepackage{bbm}
\usepackage{upgreek}
\usepackage{thmtools}
\usepackage{mathtools}
\usepackage{geometry}
\usepackage{enumitem}
\usepackage{comment}
\usepackage{microtype}

\usepackage{mathpazo}
\usepackage{domitian}
\usepackage[T1]{fontenc}

\AtBeginDocument{
 \DeclareSymbolFont{AMSb}{U}{msb}{m}{n}
 \DeclareSymbolFontAlphabet{\mathbb}{AMSb}
 }
 
\let\SSec\S

\newcommand\lopen{\mathopen{]}}
\newcommand\ropen{\mathclose{[}}

\newcommand\rmd{\mathrm{d}}
\newcommand\rme{\mathrm{e}}
\newcommand\rmi{\mathrm{i}}

\newcommand\cI{\mathcal{I}}

\newcommand\cK{\mathcal{K}}

\newcommand\cO{\mathcal{O}}

\newcommand\cS{\mathcal{S}}

\newcommand\cX{\mathcal{X}}

\newcommand\bI{\mathbf{I}}

\newcommand\bff{\mathbf{f}}

\renewcommand\AA{\mathbb{A}}

\newcommand\CC{\mathbb{C}}

\newcommand\QQ{\mathbb{Q}}
\newcommand\RR{\mathbb{R}}

\newcommand\ZZ{\mathbb{Z}}

\newcommand\mf[1]{\mathfrak{#1}}

\newcommand\A{\mathsf{A}}
\newcommand\B{\mathsf{B}}

\newcommand\G{\mathsf{G}}
\renewcommand\H{\mathsf{H}}

\newcommand\M{\mathsf{M}}
\newcommand\N{\mathsf{N}}

\renewcommand\S{\mathsf{S}}
\newcommand\T{\mathsf{T}}

\newcommand\Z{\mathsf{Z}}

\newcommand\SL{\mathsf{SL}}
\newcommand\GL{\mathsf{GL}}

\newcommand\triv{\mathbbm{1}}
\newcommand\dpii{2\uppi\rmi}
\newcommand\bs{\backslash}
\newcommand\diff{\partial}

\renewcommand\Re{\mathop{\mathrm{Re}}}

\DeclareMathOperator\Tr{Tr}
\DeclareMathOperator\sgn{sgn}
\DeclareMathOperator\supp{supp}

\DeclareMathOperator\res{res}
\DeclareMathOperator\fp{f.p.}

\DeclareMathOperator\ord{ord}
\DeclareMathOperator\rad{rad}
\DeclareMathOperator\diam{diam}
\DeclareMathOperator\Res{Res}

\newcommand\fin{\mathrm{fin}}
\newcommand\id{\mathrm{id}}
\newcommand\el{\mathrm{ell}}
\newcommand\hyp{\mathrm{hyp}}
\newcommand\unip{\mathrm{unip}}
\newcommand\spec{\mathrm{spec}}
\newcommand\cusp{\mathrm{cusp}}
\newcommand\temp{\mathrm{temp}}
\newcommand\spl{\mathrm{spl}}
\newcommand\reg{\mathrm{reg}}
\newcommand\Kl{\mathrm{Kl}}
\DeclareMathOperator\vol{vol}
\DeclareMathOperator\orb{orb}
\renewcommand\leq{\leqslant}
\renewcommand\geq{\geqslant}
\let\oldsslash\sslash
\renewcommand\sslash{{\oldsslash}}

\newcommand\legendresymbol[2]{\genfrac{(}{)}{}{}{#1}{#2}}
\newcommand\sstirling[2]{\genfrac{\{}{\}}{0pt}{}{#1}{#2}}

\makeatletter
\def\subsection{\@startsection{subsection}{2}%
  \z@{3pt\@plus0pt}{-.5em}%
  {\normalfont\bfseries}}
\makeatother

\makeatletter
\def\subsubsection{\@startsection{subsubsection}{2}%
  \z@{0pt\@plus0pt}{-.5em}%
  {\normalfont\itshape}}
\makeatother

\makeatletter
\def\@seccntformat#1{%
  \protect\textup{\protect\@secnumfont
    \ifnum\pdfstrcmp{subsection}{#1}=0 \bfseries\fi% subsection # in \bfseries
    \csname the#1\endcsname
    \protect\@secnumpunct
  }%
}  
\makeatother

\newtheoremstyle{THEOREM}% hnamei
{2.5pt}% hSpace abovei
{2pt}% hSpace belowi
{\itshape}% hBody fonti
{}% hIndent amounti
{\bfseries}% hTheorem head fonti
{.}% hPunctuation after theorem headi
{.5em}% hSpace after theorem headi
{\thmname{#1}\thmnumber{ #2}\thmnote{ (#3)}}% hTheorem head spec (can be left empty, meaning ‘normal’)i

\newtheoremstyle{DEFINITION}% hnamei
{2.5pt}% hSpace abovei
{2pt}% hSpace belowi
{}% hBody fonti
{}% hIndent amounti
{\bfseries}% hTheorem head fonti
{.}% hPunctuation after theorem headi
{.5em}% hSpace after theorem headi
{\thmname{#1}\thmnumber{ #2}\thmnote{ (#3)}}% hTheorem head spec (can be left empty, meaning ‘normal’)i

\newtheoremstyle{EXERCISE}% hnamei
{2pt}% hSpace abovei
{2pt}% hSpace belowi
{}% hBody fonti
{}% hIndent amounti
{\scshape}% hTheorem head fonti
{.}% hPunctuation after theorem headi
{.5em}% hSpace after theorem headi
{\thmname{#1}\thmnumber{ #2}\thmnote{ (#3)}}% hTheorem head spec (can be left empty, meaning ‘normal’)i

\theoremstyle{THEOREM}
\newtheorem{theorem}{Theorem}[section]
\newtheorem{lemma}[theorem]{Lemma}
\newtheorem{proposition}[theorem]{Proposition}
\newtheorem{corollary}[theorem]{Corollary}

\theoremstyle{DEFINITION}

\theoremstyle{EXERCISE}
\newtheorem{remark}[theorem]{Remark}

\numberwithin{equation}{section}

\makeatletter
\renewenvironment{proof}[1][\proofname]{\par
  \vspace{-6pt}
  \pushQED{\qed}
  \normalfont \topsep6\p@\@plus6\p@\relax
  \trivlist
  \item[\hskip\labelsep\rmfamily\bfseries
    #1\@addpunct{:}]\ignorespaces
}{
  \popQED\endtrivlist\@endpefalse
  \vspace{-6pt}
}
\makeatother

\makeatletter
\newenvironment{insertproof}[1][\proofname]{
  \par
  \vspace{-6pt}
  \pushQED{\qed}
  \normalfont \topsep6\p@\@plus6\p@\relax
  \trivlist
  \item[\hskip\labelsep\rmfamily\bfseries
    #1\@addpunct{:}]\ignorespaces
}{
  \popQED\endtrivlist\@endpefalse
  \vspace{-6pt}
}
\makeatother

\begin{document}
\setlength \lineskip{3pt}
\setlength \lineskiplimit{3pt}
\setlength \parskip{1pt}
\setlength \partopsep{0pt}

\title{Beyond endoscopy for $\GL_2$ over $\QQ$ with ramification 3: contribution of the elliptic part}
\author{Yuhao Cheng}
\address{Qiuzhen College, Tsinghua University, 100084, Beijing, China}
\email{chengyuhaomath@gmail.com}
%\email{chengyh21@mails.tsinghua.edu.cn}
\date{3 May, 2026}
\begin{abstract}
We continue to work on \emph{Beyond Endoscopy} for $\GL_2$ over $\mathbb{Q}$ with ramification at $S = \{\infty, q_1, \dots, q_r\}$ (where $2 \in S$), generalizing the final step of Altu\u{g}'s work in the unramified setting. We derive an explicit asymptotic formula for the elliptic part when summing over $n<X$ with arbitrary smooth test functions at places in $S$ for the standard representation. As a consequence, we obtain the desired limit for the simple trace formula which only occurs in the ramified case. Moreover, we prove an asymptotic formula for the traces of Hecke operators on cusp forms with arbitrary level and weight $>2$, directly generalizing Altu\u{g}'s final result. Our approach differs entirely from Altu\u{g}'s: We apply a second Poisson summation with respect to the determinant, obtaining a formula on the Hitchin-Steinberg base $\mathfrak{g}\sslash \G$. By changing variables from $(T, N)$ to $(T, \Delta)$ on $\mathfrak{g}\sslash \G$, we perform analysis in the new coordinates.
\end{abstract}

\maketitle
 
\tableofcontents

\section{Introduction}
\subsection{A new way to detect functoriality}
Let $\G,\H$ be reductive algebraic groups over the field $\QQ$ of rational numbers such that $\G$ is quasi-split. Let $\AA$ denote the ad\`ele ring of $\QQ$. Let $^L\G$ be the $L$-group of $\G$. Roughly speaking we want to consider the following problem:
Conjecturally any homomorphism $\phi\colon ^L\H\to {}^L\G$ gives a functorial lift from the automorphic representations $\sigma$ of $\H$ to the automorphic representations $\phi_*\sigma$ of $\G$.
A natural question is to ask the converse, that is, how to determine whether an automorphic representation $\pi$ of $\G$ is given by $\phi_*\sigma$ for some automorphic representation $\sigma$ of $\H$, and we want to find the smallest $\H$ (which we call \emph{primitive}). 

One way to study this is to consider the order of $L$-functions. Let $\rho$ be a complex representation of $^L\G$. For a finite set $S$ of places of $\QQ$ containing the archimedean place $\infty$, the partial $L$-function $L^{S}(s,\pi,\rho)$ can be defined and expressed as a Dirichlet series
\[
L^{S}(s,\pi,\rho)=\sum_{\substack{n=1\\ \gcd(n,S)=1}}^{+\infty}\frac{a_{\pi,\rho}(n)}{n^s}
\]
for $\Re s$ sufficiently large.

If $\pi=\phi_*\sigma$ for some automorphic representation $\sigma$ of a smaller group, then we expect to have
\[
\ord_{s=1}L^S(s,\pi,\rho)=\ord_{s=1}L^S(s,\sigma,\rho\circ\phi)\geq m(\triv,\rho\circ\phi)
\]
and the equality holds when $\H$ is primitive, where $m(\triv,\rho\circ\phi)$ denotes the multiplicity of the trivial representation in $\rho\circ\phi$ which is a complex representation of $^L\H$.

Unfortunately, we know little about the analytic continuation and the order of pole at $s=1$. However, by using the trace formula, we are able to compute the average 
\begin{equation}\label{eq:beyondendoscopy}
\sum_{\pi}m_{\pi}\ord_{s=1}L^{S}(s,\pi,\rho)\prod_{v\in S}\Tr(\pi_v(f_v))
\end{equation}
in a reasonable way, which is Langlands' \emph{Beyond Endoscopy} proposal \cite{langlands2004}.
Here, $\pi$ runs over all cuspidal representations and $f_v$ are nice functions on $\G(\QQ_v)$ for each $v\in S$. $m_\pi$ denotes the multiplicity of $\pi$ in $L^2_{\mathrm{cusp}}(\G(\QQ)\bs \G(\AA)^1)$.
We have
\[
\sum_{\pi}m_\pi a_{\pi,\rho}(p)\prod_{v\in S}\Tr(\pi_v(f_v))=I_{\mathrm{cusp}}(f^{p,\rho}),
\]
where $f^{p,\rho}=\bigotimes_{v\in S}f_v\otimes\bigotimes_{v\notin S}'f_v^{p,\rho}$, with $f_v^{p,\rho}$ spherical for $v\notin S$, and $\Tr(\pi_p(f_p^{p,\rho}))=a_{\pi,\rho}(p)$ and $\Tr(\pi_\ell(f_\ell^{p,\rho}))=1$ if $\ell\notin S\cup\{p\}$.

If $L^S(s,\pi,\rho)$ admits meromorphic continuation on $\Re s>1-\delta$ with neither zeros nor poles on the vertical line $s=1+\rmi t$ except for at $s=1$, then by the Ikehara theorem, we expect the following asymptotic formula:
\[
\lim_{X\to +\infty}\frac{1}{X}\sum_{\substack{p<X\\ p\notin S}}a_{\pi,\rho}(p)\log p=\ord_{s=1}L^S(s,\pi,\rho),
\]
where $p$ runs over all primes less than $X$.
Thus we expect that 
\[
\lim_{X\to +\infty}\frac{1}{X}\sum_{\substack{p<X\\ p \notin S}}\log pI_{\mathrm{cusp}}(f^{p,\rho})=\sum_{\pi}m_\pi \prod_{v\in S}\Tr(\pi_v(f_v))\ord_{s=1}L^S(s,\pi,\rho).
\]

We can consider a more general setting proposed by Sarnak \cite{sarnak2001}. For $\gcd(n,S)=1$, we have
\[
\sum_{\pi}m_\pi a_{\pi,\rho}(n)\prod_{v\in S}\Tr(\pi_v(f_v))=I_{\mathrm{cusp}}(f^{n,\rho}),
\]
where $f^{n,\rho}=\bigotimes_{v\in S}f_v\otimes\bigotimes_{v\notin S}'f_v^{n,\rho}$, with $f_v^{n,\rho}$ spherical for $v\notin S$, and $\Tr(\pi_p(f_p^{n,\rho}))=a_{\pi,\rho}(p^{v_p(n)})$ for all $p\notin S$. By the Ikehara theorem, we expect to have
\[
\lim_{X\to +\infty}\frac{1}{X}\sum_{\substack{n<X\\ \gcd(n,S)=1}}a_{\pi,\rho}(n)=\res_{s=1}L^S(s,\pi,\rho),
\]
and we expect that
\[
\lim_{X\to +\infty}\frac{1}{X}\sum_{\substack{n<X\\ \gcd(n,S)=1}}I_{\mathrm{cusp}}(f^{n,\rho})=\sum_{\pi}m_\pi \prod_{v\in S}\Tr(\pi_v(f_v))\res_{s=1}L^S(s,\pi,\rho),
\]
where $\pi$ runs over certain representations of $\G(\AA)$ depending on $f^{n,\rho}$. For example, if $\G=\GL_2$ and $\rho$ is the standard representation, then $\res_{s=1}L^S(s,\pi,\rho)=0$ for all cuspidal $\pi$ (cf. \cite[Theorem 11.7.1]{getz2024} for example). Hence the expectation is
\begin{equation}\label{eq:standardrepresentation}
\lim_{X\to +\infty}\frac{1}{X}\sum_{\substack{n<X\\ \gcd(n,S)=1}}I_{\mathrm{cusp}}(f^{n,\rho})=0.
\end{equation}

%For more about Beyond Endoscopy, one may refer \cite{langlands2004,arthur2017,altug2024}.

We now assume that $\G=\GL_2$, $\rho=\mathrm{Std}$ is the standard representation.
For any prime number $p$ and $m\in \ZZ_{\geq 0}$, we define
\[
\cX_p^{m}=\{X\in \M_2(\ZZ_p)\,|\, \mathopen{|}\det X\mathclose{|}_p = p^{-m}\}.
\]
For example, if $m=0$, $\cX_p^{m}$ is just $\cK_p=\GL_2(\ZZ_p)$. By Hecke operator theory, we can choose $f^{n}$ to be $\bigotimes_{v\in \mf{S}}f^{n}_v$, where $\mf{S}$ denotes the set of places of $\QQ$, and
\begin{enumerate}[itemsep=0pt,parsep=0pt,topsep=2pt,leftmargin=0pt,labelsep=3pt,itemindent=9pt,label=\textbullet]
  \item If $v=p$, $f^{n}_p=p^{-n_p/2}\triv_{\cX_p^{n_p}}$, where $n_p=v_p(n)$.
  \item If $v=\infty$,  $f^{n}_\infty=f_\infty\in C^\infty(Z_+\bs \G(\RR))$ such that the orbital integrals are compactly supported modulo $Z_+$, and other than this condition they are arbitrary.
\end{enumerate}
Altu\u{g} \cite{altug2015,altug2017,altug2020} proved \eqref{eq:standardrepresentation} by using the Arthur-Selberg trace formula for $f_{\infty,m}$ that is a matrix coefficient of a weight $m$ discrete series for $m>2$ even. More precisely, $f_{\infty,m}$ is a function supported on
\[
\GL_2^{\!+}(\RR)=\left\{g=\begin{pmatrix}
                        a & b \\
                        c & d 
                      \end{pmatrix}\in \GL_2(\RR)\,\middle|\, \det g>0\right\}
\]
such that
\[
f_\infty\left(\begin{pmatrix}  a & b \\ c & d \end{pmatrix}\right)=\frac{m-1}{4\uppi}\frac{\det(g)^{m/2}(2\rmi)^m}{(-b+c+(a+d)\rmi )^m}
\]
for $g=(\begin{smallmatrix} a & b \\c & d \end{smallmatrix})\in \GL_2^{\!+}(\RR)$. 
Altu\u{g} actually proved that
\[
\sum_{n<X}\Tr(T_m(n))\ll_{m,\varepsilon} X^{\frac{31}{32}+\varepsilon},
\]
where $T_m(n)$ denotes the (normalized) $n^{\mathrm{th}}$ Hecke operator acting on $S_m(\SL_2(\ZZ))$, the space of holomorphic cusp forms of weight $m$ for the modular group $\Gamma=\SL_2(\ZZ)$.

\subsection{The main results}
In this paper, we will approach the formula \eqref{eq:standardrepresentation} in the case over $\QQ$ with ramification and  $n$ such that $\gcd(n,S)=1$ and $\rho=\mathrm{Sym}^1$ is the standard representation.

We consider $S=\{\infty,q_1,\dots,q_r\}$ for primes $q_1,\dots,q_r$ such that $2\in S$, and a corresponding function $f^{n}=\bigotimes_{v\in \mf{S}}'f^n_v$ such that the local components at places in $S$ are arbitrary. 
Specifically, for $\gcd(n,S)=1$, the function $f^n$ is defined as follows:
\begin{enumerate}[itemsep=0pt,parsep=0pt,topsep=2pt,leftmargin=0pt,labelsep=3pt,itemindent=9pt,label=\textbullet]
  \item If $v=p\notin S$, we choose $f^n_v=p^{-n_p/2}\triv_{\cX_p^{n_p}}$, where $n_p=v_p(n)$.
  \item If $v=q_i\in S$ and is a prime, we choose $f^n_v=f_{q_i}$ to be an arbitrary function in $C_c^\infty(\G(\QQ_{q_i}))$.
  \item If $v=\infty$, we choose $f^n_v=f_\infty\in C^\infty(Z_+\bs \G(\RR))$ such that its orbital integrals are compactly supported modulo $Z_+$.
\end{enumerate}
Note that in this case, $f_v^{n}$ is spherical for $v\notin S$, and $\Tr(\pi_p(f_p^{n}))=a_{\pi}(p^{n_p})$ for all $p\notin S$.

To prove \eqref{eq:standardrepresentation}, it suffices to give an asymptotic formula of 
\[
\sum_{\substack{n<X\\ \gcd(n,S)=1}}I_\cusp(f^n)
\]
and show that it is $o(X)$. Using the Arthur-Selberg trace formula we may split $I_\cusp(f^n)$ into various terms. For example, we have the elliptic part
\[
\sum_{\substack{n<X\\ \gcd(n,S)=1}}I_\el(f^n).
\]

Our main goal of this paper is to give such an asymptotic formula by reducing into the hyperbolic parts
\[
I_{\hyp}^{\deg=1}(f^n)\qquad\text{and}\qquad \widehat{J}_{\hyp}^{S}(f^n),
\]
which are defined in \eqref{eq:hyperbolicpartdeg1} and \eqref{eq:hyperbolicpartmodify}, respectively.

The main theorem in this paper is
\begin{theorem}[Contribution of the elliptic part]\label{thm:contributeellipticfinal}
For any $\varepsilon>0$, we have
\begin{align*}
\sum_{\substack{n<X\\\gcd(n,S)=1}}I_\el(f^n) =&\mf{A}X^{\frac32}+\mf{B}(X\log X-X)+(\mf{C}+\mf{D})X\\ -&\upgamma_S\sum_{\substack{n<X\\\gcd(n,S)=1}}I_{\hyp}^{\deg=1}(f^n) -\sum_{\substack{n<X\\\gcd(n,S)=1}}\widehat{J}_{\hyp}^{S}(f^n) +O(X^{\frac78+\varepsilon}),
\end{align*}
where $\mf{A}$ is defined in \autoref{thm:contributenontempered}, $\mf{B},\mf{C},\mf{D}$ are defined in \autoref{thm:contributealpha0}, $\upgamma_S$ is defined in \eqref{eq:defgammas}.
The implied constant only depends on $\varepsilon$, $f_\infty$ and $f_{q_i}$.
\end{theorem}

%Also we present a conjecture that the error term can be improved.
%\begin{conjecture}
%For any $\varepsilon>0$, the error term of the contribution of the elliptic part above can be bounded by $X^{1/2+\varepsilon}$.
%\end{conjecture}

%\autoref{thm:contributeellipticfinal} is a generalization of \cite{altug2020} in two different ways. The first is the generalization to the ramified case. The second one, which is the most important, is that we relax the condition of function $f_\infty$ on the archimedean place, which was assumed to be $f_{\infty,m}$ with $m>2$ even.

Although we have not proved the asymptotic formula for the non-elliptic parts, we can prove \eqref{eq:standardrepresentation} in special cases. For example, one can prove the simple trace formula case. Let $v$ be a place of $\QQ$. We call an absolutely integrable function $f_v$ on $\G(\QQ_v)=\GL_2(\QQ_v)$ \emph{supercuspidal} if for any $g,h\in \G(\QQ_v)$, 
\[
\int_{\N(\QQ_v)}f_v(gnh)\rmd n=0,
\]
where $\N$ denotes the unipotent radical of the standard Borel subgroup $\B$ of $\G$. Suppose that additionally
\begin{enumerate}[itemsep=0pt,parsep=0pt,topsep=0pt, leftmargin=0pt,labelsep=2.5pt,itemindent=15pt,label=\upshape{(\roman*)}]
  \item At one place $v_1\in S$, $f_{v_1}$ is supercuspidal.
\item At one place $v_2\in S$ different from $v_1$, $f_{v_2}$ is supported on the regular elliptic elements.
\end{enumerate}
Then we have the \emph{Deligne-Kazhdan simple trace formula} (see \cite[Theorem 3.1 of Lecture V]{gelbart1996trace} or \cite[Corollary 18.4.1]{getz2024}, for example)
\[
I_\el(f^n)=I_\cusp(f^n).
\]
Also, it is easy to see $\mf{A}=\mf{B}=\mf{C}=\mf{D}=0$ and $I_\hyp^{\deg=1}(f^n)=\widehat{J}_{\hyp}^S(f^n)=0$ in this case once we know the definitions of them. Hence in this case we have
\[
\sum_{\substack{n<X\\ \gcd(n,S)=1}}I_\cusp(f^n)\ll X^{\frac78+\varepsilon}
\]
and thus \eqref{eq:standardrepresentation} holds in this case.
Note that this case cannot occur in the unramified case done by Altu\u{g}.

As another application, we will prove a generalization of Altu\u{g}'s final result \cite{altug2020} in the appendix.

\subsection{An overview of the proof}
In \cite{cheng2025}, we did Poisson summation for $\Sigma^n(\xi)=I_\el(f^n)+\Sigma^n(\square)$ on the semilocal space $\QQ_S$, where $\Sigma^n(\square)$ is defined to be $\Sigma(\square)$ in \cite{cheng2025}. Then, we isolated specific representations in the $\xi=0$ term. Also, in \cite{cheng2025b} we bounded the $\xi\neq 0$ term and got the $1/4$-bound. However, we will not split these two terms for the summing over $n<X$ case as in \cite{altug2020}. We will show the main result by a completely different method. 

The first step is to estimate the elliptic part adding the square term
\[
S(X)=\sum_{\substack{n<X\\ \gcd(n,S)=1}}\Sigma^n(\xi)= \sum_{\substack{n<X\\ \gcd(n,S)=1}}(I_\el(f^n)+\Sigma^n(\square))
\]
for which we did Poisson summation with respect to the trace for a single term.

We do a second Poisson summation with respect to the determinant. We need to consider a smooth variant
\[
S_G(X)=\sum_{n\in \ZZ}G\legendresymbol{n}{X}\Sigma^n(\xi)=\sum_{n\in \ZZ^S}G\legendresymbol{n^{(q)}}{X}\sum_{\xi\in \ZZ^S}\phi^n(\xi).
\]
Thus we obtain a formula with respect to $\{T,N\}$ on the  Hitchin-Steinberg base $\mf{g}\sslash \G\cong\mf{a}\sslash W$:
\[
S_G(X)=\sum_{\xi,\alpha\in \ZZ^S}\widehat{\phi}(\xi,\alpha).
\]

Next, we do a change of variable $(T,N)\mapsto (T,\Delta)$ on $\mf{g}\sslash \G$, where $\Delta$ represents the discriminant. After changing the coordinate, we consider the $\alpha=0$ term and the $\alpha\neq 0$ term separately.

For the $\alpha=0$ term, we use Poisson summation for $\xi$ in the reverse direction. Then we use Mellin inversion formula and contour shift method to compute the contribution from the residues. The contribution is almost the same as considering the $\xi=0$ term directly.

For the $\alpha\neq 0$ term, the $\xi$-sum is a theta-type function
\[
\int_{\QQ_S}f(a,b)\rme\legendresymbol{\sigma(x+ua)^2}{2}\rme_q\legendresymbol{\tau(y+vb)^2}{2}\rmd a\rmd b.
\]
Hence we can do Poisson summation on that. The expression after this process is a double Fourier transform which can be treated directly by technical analysis.

After giving an asymptotic formula for $S_G(X)$, we analyze the difference between $S(X)$ and $S_G(X)$. To do this, we need to consider the tempered part and the nontempered part separately. The asymptotic formula of the nontempered parts of them can both be computed explicitly. For the tempered part, their difference can be controlled due to the $1/4$-bound towards the Ramanujan conjecture by the main result of \cite{cheng2025b}. Hence we obtain an asymptotic formula for $S(X)$.

Finally, we compute directly the square term $\Sigma^n(\square)$ and thus reduce the contribution of the elliptic part to the hyperbolic part. The key observation is that in \cite{cheng2025} when we did the Poisson summation, there is a parameter $\vartheta$ and the above proof actually works when $|\vartheta-1/2|$ is sufficiently small. The full asymptotic formula is a linear function of $\vartheta$. Hence we yield the formulas for the degree $0$ term and the degree $1$ term. For the degree $0$ term, we obtain \autoref{thm:contributeellipticfinal}. For the degree $1$ term we also obtain the following result:

\begin{theorem}[Contribution of degree $1$ term of the elliptic part]\label{thm:contributeellipticdeg1}
For any $\varepsilon>0$, we have
\begin{align*}
    &2\prod_{i=1}^{r}(1-q_i^{-1})\sum_{\substack{n<X\\\gcd(n,S)=1}}\sum_{\pm}\sum_{\nu\in \ZZ^r}\sum_{\substack{a\neq b\in \ZZ^S\\ ab=\pm nq^\nu}}\theta_\infty\begin{pmatrix} a & 0 \\ 0 & b \end{pmatrix}\theta_q\begin{pmatrix}
 a & 0 \\ 0 & b\end{pmatrix}\log|(a-b)^{(q)}|\\
  =&-\mf{B}(X\log X-X)-\mf{D}X +O(X^{\frac78+\varepsilon}),
\end{align*}
where the implied constant only depends on $\varepsilon$, $f_\infty$ and $f_{q_i}$.
\end{theorem}

%\noindent\textbf{Question:} Is there a direct proof of \autoref{thm:contributeellipticdeg1}?

\subsection{Notations}
\begin{enumerate}[itemsep=0pt,parsep=0pt,topsep=0pt,leftmargin=0pt,labelsep=3pt,itemindent=9pt,label=\textbullet]
  \item $\# X$ denotes the number of elements in a set $X$.
  \item For $A\subseteq X$, $\triv_A$ denotes the characteristic function on $X$, defined by $\triv_A(x)=1$ for $x\in A$ and $\triv_A(x)=0$ for $x\notin A$.
  \item $\triv$ also denotes the trivial character or the trivial representation.
  \item For $x\in \RR$, $\lfloor x\rfloor$ denotes the greatest integer that is less than or equal to $x$, and $\lceil x\rceil$ denotes the smallest integer that is greater than or equal to $x$.
  \item We often use the notation $a\equiv b\,(n)$ to denote $a\equiv b\pmod n$.
  \item For $D\equiv 0,1\pmod 4$, $\legendresymbol{D}{\cdot}$ denotes the Kronecker symbol.
  \item If $R$ is a ring (which we \emph{always} assume to be commutative with $1$), $R^\times$ denotes its group of units.
  \item $\mf{S}$ denotes the set of places of $\QQ$.
%  \item $\SS^1$ denotes the group of complex numbers with absolute value $1$. 
%  \item For any locally compact abelian group $G$, $\widehat{G}$ or $G\sphat\ $ denotes its Pontryagin dual, which is the set of continuous homomorphisms from $G$ to $\SS^1$, equipped with the compact-open topology.
  \item For $S=\{\infty, q_1,\dots,q_r\}$, $\ZZ^S$ denotes the ring of $S$-integers in $\QQ$, that is,
    \[
    \ZZ^S=\{\alpha\in \QQ\ |\ v_p(\alpha)\geq 0\ \text{for all}\ p\notin S\}.
    \]
    Additionally, we define
    \[
    \QQ_S=\prod_{v\in S}\QQ_v=\RR\times \QQ_{q_1}\times\dots\times\QQ_{q_r}\quad\text{and}\quad\QQ_{S_\fin}=\prod_{v\in S_\fin}\QQ_v=\QQ_{q_1}\times\dots\times\QQ_{q_r}
    \]
  and
    \[
    \ZZ_{S_\fin}=\prod_{v\in S_\fin}\ZZ_v=\ZZ_{q_1}\times\dots\times\ZZ_{q_r}.
    \]
  \item For $n\in \ZZ$, we write $\gcd(n,S)=1$, or $n\in \ZZ_{(S)}$, if 
  $p\nmid n$ for all $p\in S$. %\emph{Note the difference between of $\ZZ_{(S)}$ and $\ZZ_{S_\fin}^\times=\ZZ_{q_1}^\times\times\dots\times\ZZ_{q_r}^\times$.} 
  We write $n\in \ZZ_{(S)}^{>0}$ if additionally $n>0$.
  \item Let $p$ be a prime. For $a\in\QQ$, we define $a_{(p)}$ to be the $p$-part of $a$, that is, 
\[
a_{(p)}=p^{v_p(a)},
\]
and we define $a^{(p)}=a/a_{(p)}$.
Moreover, we define 
  \[
    a_{(q)}=\prod_{i=1}^{r}q_i^{v_{q_i}(a)}\quad\text{and}\quad a^{(q)}=\frac{a}{a_{(q)}}=\prod_{p\notin S}p^{v_{p}(a)}.
  \]
 %\item If $E$ is a number field or a nonarchimedean local field, $\cO_E$ denotes its ring of integers.
  %\item For any prime $p$, $\cK_p=\G(\ZZ_p)$ denotes the standard maximal compact subgroup of $\G(\QQ_p)$, and
  %\[
  %\cI_p=\left\{\begin{pmatrix}
  %             a & b \\
  %             c & d 
  %           \end{pmatrix}\in \G(\ZZ_p)\,\middle|\, p\mid c\right\}.
  %\]
  %denotes the Iwahori subgroup of $\G(\QQ_p)$.
  \item For $n\in \ZZ_{>0}$, $\bm{d}(n)$ denotes the number of divisors of $n$, $\bm{\sigma}(n)$ denotes the sum of divisors of $n$, $\bm{\phi}(n)$ denotes the Euler totient function. $\bm{\mu}(n)$ denotes the M\"obius function, defined by
  \[
  \bm{\mu}(n)=\begin{cases}
    1, & \text{if $n=1$}, \\
    (-1)^m, & \text{if $n$ is a product of $m$ distinct primes}, \\
    0, & \text{otherwise}.
  \end{cases}
  \]
  $\bm{\Lambda}(n)$ denotes the von Mangoldt function, defined by
  \[
  \bm{\Lambda}(n)=\begin{cases}
    \log p, & \text{if $n=p^m$ with $m\in \ZZ_{>0}$}, \\
    0, & \text{otherwise}.
  \end{cases}
  \]
  \item $\Gamma(s)$ denotes the gamma function, defined by
  \[
  \Gamma(s)=\int_{0}^{+\infty}\rme^{-x}x^s\frac{\rmd x}{x},
  \]
  for $\Re s>0$, analytically continued to $\CC$. $\zeta(s)$ denotes the Riemann zeta function, defined by
  \[
    \zeta(s)=\sum_{n=1}^{+\infty}\frac{1}{n^s}=\prod_{p}\frac{1}{1-p^{-s}}
  \]
  for $\Re s>1$, analytically continued to $\CC$.
  \item For any Dirichlet character $\chi$, we define
  \[
  L(s,\chi)=\sum_{n=1}^{+\infty}\frac{\chi(n)}{n^s}=\prod_{p}\frac{1}{1-\chi(p)p^{-s}}.
  \]
  for $\Re s>1$, analytically continued to $\CC$. Moreover, we define
  \[
  L^S(s,\chi)=\sum_{\substack{n=1\\\gcd(n,S)=1}}^{+\infty}\frac{\chi(n)}{n^s}=\prod_{p\notin S}\frac{1}{1-\chi(p)p^{-s}}.
  \]
  for $\Re s>1$, analytically continued to $\CC$. Finally, we set
  \[
  \zeta^S(s)=L^S(s,\triv)=\sum_{\substack{n=1\\\gcd(n,S)=1}}^{+\infty}\frac{1}{n^s}=\prod_{p\notin S}\frac{1}{1-p^{-s}}=\prod_{i=1}^{r}(1-q_i^{-s})\zeta(s).
  \]
  \item $(\sigma)$ denotes the vertical contour from $\sigma-\rmi\infty$ to $\sigma+\rmi\infty$.
  \item Let $f$ be a meromorphic function near $z=z_0$ and suppose that 
  \[
  f(z)=\sum_{n\in \ZZ}a_n(z-z_0)^n
  \]
  is its Laurent expansion near $z_0$, then we denote $\fp_{z=z_0}f(z)=a_0$.
  \item We define $\cS$ to be the set of smooth functions $\Phi$ on $\lopen 0,+\infty\ropen$ such that for any $k\in \ZZ_{\geq 0}$, $\Phi^{(k)}(x)$ are of rapid decay as $x\to +\infty$.
  \item For $G\in C_c^\infty(\RR)$, $p\in [1,+\infty]$ and $k\in \ZZ_{\geq 0}$, $\|G\|_p$ denotes the $L^p$-norm of $G$, and $\|G\|_{k,p}$ denotes the $W^{k,p}$-Sobolev norm of $G$, namely
    \[  
    \|G\|_{k,p}=\sum_{j=0}^{k}\|G^{(j)}\|_{p}.
    \]
  \item We define $\rme(x)=\rme_\infty(x)=\rme^{\dpii x}$. For a prime $p$ and $x\in \QQ_p$, we define $\langle x\rangle_p$ to be the "fractional part" of $x$. Namely, if $x=\sum_{n\geq -N}a_np^n\in \QQ_p$, then 
      \[
      \langle x\rangle_p=\sum_{n=-N}^{-1}a_np^n\in \QQ.
      \]
      We then define $\rme_p(x)=\rme(-\langle x\rangle_p)$ for $x\in \QQ_p$. \emph{Note the minus sign}.
  \item We use $f(x)=O(g(x))$ or $f(x)\ll g(x)$ to denote that there exists a constant $C$ such that $|f(x)|\leq C|g(x)|$ for all $x$ in a specified set. If the constant depends on other variables, they will be subscripted under $O$ or $\ll$. %We sometimes add $x\to a$ to this notation, indicating that we only consider the behavior of $f(x)$ and $g(x)$ near $a$. Specifically, $x$ varies in a fixed neighborhood of $a$ that does not depend on other variables.
  \item The notation $f(x)\asymp g(x)$ indicates that $f(x)\ll g(x)$ and $g(x)\ll f(x)$. If the constant depends on other variables, they will be subscripted under $\asymp$.
  \item We use $\square$ as the end of the proof and use $\blacksquare$ as the end of the proof of a lemma that is inserted in the middle of another proof.
\end{enumerate}

\section{Preliminaries}\label{sec:preliminaries}
In this section we will mainly recall some useful definitions and properties in previous papers \cite{cheng2025,cheng2025b}. Readers familiar with these two papers may skip this section except for \autoref{subsec:singularities} (since there are some notations changed), \autoref{subsec:global} and \autoref{subsec:mellinnorm}.
%\subsection{The Arthur-Selberg trace formula for $\GL_2$}
%The Arthur-Selberg trace formula for $\GL_2$ and $f\in C_c^\infty(\G(\AA)^1)=C_c^\infty(Z_+\bs\G(\AA))$ (where $Z_+$ denotes the identity component of $\Z(\RR)$, the center of $\G(\RR)$) can be written in the following form:
%\begin{equation}\label{eq:traceformulagl2}
%I_{\mathrm{geom}}(f)=I_{\mathrm{spec}}(f),
%\end{equation}
%where the geometric side is
%\begin{equation}\label{eq:geometric}
%I_{\mathrm{geom}}(f)=I_\id(f)+I_\el(f)+I_\hyp(f)+I_\unip(f)
%\end{equation}
%and the spectral side is
%\begin{equation}\label{eq:spectral}
%I_{\mathrm{spec}}(f)=I_{\mathrm{cusp}}(f)+I_{\mathrm{cont}}(f)+\sum_{\mu}\Tr(\mu(f)) +\frac14\sum_{\mu}\Tr(M(\mu,0)(\xi_0\otimes\mu)(f)).
%\end{equation}
%The sum of $\mu$ is over all $1$-dimensional representations of $\G(\QQ)\bs \G(\AA)^1$. %such that $\mu^2=\omega$. 

%The Arthur-Selberg trace formula of $\GL_2$ was occurred in lots of literature. For example, \cite[Chapter 16]{langlands1970}, \cite[Theorem 6.33]{gelbart1979} and \cite[Theorem 7.14]{knapp1997}. %Recently, \cite{aydougdu2017} gave a clearer formulation. 
%The explicit formula of these terms are contained in these references. 

%In this paper we do not need the explicit formula of the spectral side since we will see $I_\cusp(\bff^n)=I_\spec(\bff^n)$ soon.

\subsection{The elliptic part of the trace formula and measure normalizations}\label{subsec:measure}
Let $\mf{S}$ be the set of places of $\QQ$ and let $\AA$ be the ad\`ele ring of $\QQ$. We want to compute the elliptic part of the trace formula for certain functions of the form $f=\bigotimes_{v\in \mf{S}}' f_v$ on $\G(\AA)$. The elliptic part of the trace formula is
\begin{equation}\label{eq:elliptictrace}
I_\el(f)=\sum_{\gamma\in \G(\QQ)^\#_\el}\vol(\gamma)\prod_{v\in \mathfrak{S}}\orb(f_v;\gamma).
\end{equation}
where $ \G(\QQ)^\#_\el$ denotes the set of elliptic conjugacy classes in $\G(\QQ)$, and
\[
\orb(f_v;\gamma)=\int_{\G_\gamma(\QQ_v)\bs \G(\QQ_v)} f_v(g^{-1}\gamma g)\rmd g,
\qquad
\vol(\gamma)=\int_{Z_+\G_\gamma(\QQ)\bs \G_\gamma(\AA)} \rmd g.
\]
with $Z_+$ the connected component of the identity matrix in the center $\Z(\RR)$ of $\G(\RR)$. The measures of $\G_\gamma(\QQ_v)$ and $\G(\QQ_v)$ are normalized as follows. For $\G(\QQ_v)$, 
\begin{enumerate}[itemsep=0pt,parsep=0pt,topsep=0pt,leftmargin=0pt,labelsep=3pt,itemindent=9pt,label=\textbullet]
  \item If $v=p$ is a prime, we normalize the Haar measure on $\G(\QQ_p)$ such that the volume of $\G(\ZZ_p)$ is $1$.
  \item If $v=\infty$, we choose an arbitrary Haar measure.
\end{enumerate} 

By Proposition 2.1 of \cite{cheng2025}, for elliptic $\gamma$, there exists a quadratic extension $E/\QQ$ such that $\G_\gamma\cong\Res_{E/\QQ}\GL_1$. For hyperbolic $\gamma$, we may take $E=\QQ\times \QQ$ so that such identification still holds. Hence $\G_\gamma(\QQ_v)\cong E_v^\times$.
\begin{enumerate}[itemsep=0pt,parsep=0pt,topsep=0pt,leftmargin=0pt,labelsep=3pt,itemindent=9pt,label=\textbullet]
  \item If $v=p$ is a prime, we know that $E_p$ is either $\QQ_p^2$ or a quadratic extension of $\QQ_p$. In the first case, we let $\G_\gamma(\ZZ_p)=\ZZ_p^\times \times \ZZ_p^\times$. In the second case, we let $\G_\gamma(\ZZ_p)=\cO_{E_p}^\times$.
  In each case, we normalize the Haar measure on $\G_\gamma(\QQ_p)=E_p^\times$ such that the volume of $\G_\gamma(\ZZ_p)$ is $1$.
  \item If $v=\infty$, we know that $E_v$ is either $\RR^2$ or $\CC$. If $E_v=\RR^2$, we define the measure on $E_v^\times=\RR^\times \times \RR^\times$ as $\rmd x\rmd y/|xy|$, where $(x,y)$ is the coordinate of $\RR^2$. The action of $Z_+$ on $E_v$ is $a\cdot(x,y)=(ax,ay)$, and we define the measure on $Z_+\bs E_v$ to be the measure of the quotient of $E_v$ by the measure $\rmd a/a$ on $Z_+$. If $E_v=\CC$, we choose the measure on $E_v^\times=\CC^\times$ to be $2\rmd r\rmd \theta/r$, where we use the polar coordinate $(r,\theta)$ on $\CC^\times$. The action of $Z_+$ on $E_v$ is $a\cdot z=az$, and we define the measure on $Z_+\bs E_v$ to be the measure of the quotient of $E_v$ by the measure $\rmd a/a$ on $Z_+$.
\end{enumerate} 
\begin{remark}
These measure normalizations are taken from \cite{finis2011} which was used in the previous two papers \cite{cheng2025,cheng2025b}, instead of those in \cite{langlands2004}, which was used in Altu\u{g}'s paper.
\end{remark}

\subsection{The modified $p$-adic norm}
For any prime $p$, the \emph{modified norm} $|\cdot|_p'$ is defined as follows: 

For $p\neq 2$ and $y\in \QQ_p$, we define $|y|_p' = p^{-2\lfloor v_p(y)/2\rfloor}$.  
This satisfies $|y|_p' = |y|_p$ if $v_p(y)$ is even, and $|y|_p' = p |y|_p$ if $v_p(y)$ is odd.

For $p=2$ and $y\in \QQ_p$, we define
\[
|y|_p' = \begin{cases}
  p^{-v_p(y)}=|y|_p & \text{if $v_p(y)$ is even, $y_0\equiv 1\,(4)$}, \\
  p^{-v_p(y)+2}=p^2|y|_p & \text{if $v_p(y)$ is even, $y_0\equiv 3\,(4)$}, \\
  p^{-v_p(y)+3}=p^3|y|_p & \text{if $v_p(y)$ is odd},
\end{cases}
\]
where $y_0=yp^{-v_p(y)}$.
Clearly $|a^2y|_p'=|a|_p^2|y|_p'$ for any $a\in \QQ_p$.

Also, for any regular element $\gamma\in\G(\QQ_p)$, we denote $T_\gamma=\Tr\gamma$ and $N_\gamma=\det\gamma$. $k_\gamma$ is defined such that $p^{k_\gamma}=|T_\gamma^2-4N_\gamma|_p'^{-1/2}$. This coincides with the original definition of \cite{cheng2025} (see Proposition 2.6 of loc. cit.).

%We present some properties of modified norms which have been proved in \cite{cheng2025b}.
%\begin{lemma}\label{lem:modifynorm}
%$|x|_p'\asymp_p |x|_p$ for all $x\in \QQ_p$. Moreover, if $a\in 1+p^2\ZZ_p$, then $|ay|_p'=|y|_p'$.
%\end{lemma}

%\begin{lemma}\label{lem:ladiconverge}
%If $\Re s>1$, then the integrals
%\[
%\int_{\ZZ_p} |x|_p^{-s}\rmd x  \qquad\text{and}\qquad \int_{\ZZ_p} |x|_p'^{-s}\rmd x 
%\]
%both converge.
%\end{lemma}

\subsection{Singularities of the orbital integrals}\label{subsec:singularities}
We define
\[
\omega_\infty(x)=\begin{cases}
             0, & x>0, \\
             1, & x<0
           \end{cases}
\]
for $x\in \RR$ with $x\neq 0$ and
\[
\omega_p(y)=\legendresymbol{y|y|_{p}'}{p}
\]
for $p\in S$ and $y\in \QQ_p$ with $y\neq 0$. When $p=q_i$, we often write $\omega_p=\omega_i$.
For $\iota\in \{0,1\}$ we define
\[
X_{\iota}=\{x\in \RR\ |\ \omega_\infty(x)=\iota\}
\]
and for $\epsilon_i\in \{0,\pm 1\}$, we define
\[
Y_{\epsilon_i}=\{y_i\in \QQ_{q_i}\ |\ \omega_i(y_i)=\epsilon_i\}. 
\]
The notation here is different from \cite{cheng2025b}: The variables $x$ and $y$ here represent \emph{discriminants} but not traces in this paper. In loc. cit., the determinant $\pm nq^\nu$ is almost fixed but we will vary it in this paper. So it is better to use discriminant as the domain of $\omega_\infty$ and $\omega_p$. Also, we set $\omega_p(T,N)=\omega_p(T^2-4N)$ so that $\omega_p(y)$ defined in \cite{cheng2025b} is $\omega_p(y,\pm nq^\nu)$.

\subsubsection{Nonarchimedean case}
For any prime $p\in S$, we define
\[
\theta_{p}(\gamma)=\frac{1}{\mathopen{|}\det\gamma\mathclose{|}_p^{1/2}}\left(1-\frac{\chi(p)}{p}\right)^{-1}p^{-{k_\gamma}}\orb(f_{p};\gamma),
\]
where $\chi(p)=\omega_p(\Tr\gamma,\det\gamma)$. By Corollary 2.12 of \cite{cheng2025}, the local behavior of $\theta_{p}$ at $z=aI$ is
\begin{equation}\label{eq:shalikalocal}
\theta_{p}(\gamma)=\lambda_1\left(1-\frac{\chi(p)}{p}\right)^{-1}p^{-{k_\gamma}} \frac{1-\chi(p)}{1-p}+\lambda_2.
\end{equation}

Clearly $\theta_{p}(\gamma)$ is invariant under conjugation. Thus $\theta_{p}(\gamma)$ can be parametrized by $T=\Tr\gamma$ and $N=\det\gamma$, i.e., $\theta_{p}(\gamma)=\theta_p(T,N)$. 
Since $\theta_p(\gamma)$ is smooth away from the center, $\theta_p(T,N)$ is smooth except at $T^2=4N$. Moreover, for any $\epsilon\in \{0,\pm 1\}$, by \eqref{eq:shalikalocal} we have
\[
\theta_p(T,N)=\theta_{p,1}(T,N)+\theta_{p,0}(T,N),
\]
where $\theta_{p,\tau}(T,N)={|T^2-4N|_p'}^{\tau/2}\psi_\tau(T,N)$ for $\tau\in \{0,1\}$, where $\psi_\tau(T,N)$ is a smooth function on $\{(T,N)\in \QQ_p^2\,|\,\omega_p(T^2-4N)=\epsilon\}$ and up to boundary, with bounded support, for any $\epsilon\in \{0,\pm 1\}$.

For the invariance with respect to center we have the following lemma:
\begin{lemma}\label{lem:orbitalintegralsmooth}
Let $p\in S$. Then there exists $L\in \ZZ_{>0}$ such that for any $c\in 1+p^{L}\ZZ_p$, any $T\in \QQ_p$ and $N\in \QQ_p^\times$, we have $\theta_p(cT,c^2N)=\theta_p(T,N)$.
\end{lemma}
\begin{proof}
Since $f_p$ is smooth and compactly supported, there exists $L\geq 2$ such that $f_p(bg)=f_p(g)$ for any $c\in 1+p^{L}\ZZ_p$ and $g\in \G(\QQ_p)$. Hence $\orb(f_p;c\gamma)=\orb(f_p;\gamma)$ for any $c\in 1+p^{L}\ZZ_p$ and regular elements $\gamma$. Clearly $c\gamma$ is split (resp. inert, ramified) if and only if $\gamma$ is. Hence $\theta_p(c\gamma)=\theta_p(\gamma)$ for any $c\in 1+p^{L}\ZZ_p$ and regular elements $\gamma$. Thus for any $T\in \QQ_p$ and $N\in \QQ_p^\times$ with $T^2-4N\neq 0$, we have
\[
\theta_p(cT,c^2N)=\theta_p(T,N).
\]
By continuity, the above equation also holds for $T^2-4N=0$.
\end{proof}

\subsubsection{Archimedean case}
Recall the following theorem  (see Theorem 2.13 in \cite{cheng2025}):
\begin{theorem}\label{thm:archimedeanintegral}
For any $f_\infty\in C^\infty(\G(\RR))$, any maximal torus $\T(\RR)$ in $\G(\RR)$ and any $z$ in the center of $\G(\RR)$, there exists a neighborhood $N$ in $\T(\RR)$ of $z$ and smooth functions $g_1,g_2\in C^\infty(N)$ (depending on $f_\infty$ and $z$) such that
\begin{equation}\label{eq:archimedeanintegral}
\orb(f_\infty;\gamma)=g_1(\gamma)+\frac{|\gamma_1\gamma_2|^{1/2}}{|\gamma_1-\gamma_2|}g_2(\gamma)
\end{equation}
for any $\gamma\in \T(\RR)$, where $\gamma_1$ and $\gamma_2$ are the eigenvalues of $\gamma$.  Moreover, $g_1$ and $g_2$ can be extended smoothly to all split and elliptic elements, remaining invariant under conjugation, with $g_1(\gamma)=0$ if $\T(\RR)$ is split, and $g_2$ can further be extended smoothly to the center. If $f_\infty$ is $Z_+$-invariant, then $g_1$ and $g_2$ are also $Z_+$-invariant.
\end{theorem}

We define
\[
\theta_\infty(\gamma)=\frac{|\gamma_1-\gamma_2|}{|\gamma_1\gamma_2|^{1/2}}\orb(f_\infty;\gamma)= \frac{|\gamma_1-\gamma_2|}{|\gamma_1\gamma_2|^{1/2}}g_1(\gamma)+g_2(\gamma).
\]
Since $g_1,g_2$ and $\theta_\infty$ are invariant under conjugation, we can parametrize them by $T_\gamma$ and $N_\gamma$ as in the nonarchimedean case.

Since $T_{z\gamma}=aT_\gamma$ and  $N_{z\gamma}=a^2N_\gamma$ for $z=aI$ with $a>0$, we have $g_i(T_\gamma,N_\gamma)=g_i(aT_\gamma,a^2N_\gamma)$ and $\theta_\infty(T_\gamma,N_\gamma)=\theta_\infty(aT_\gamma,a^2N_\gamma)$ for $i=1,2$ and any $a>0$. We also set 
\[
\theta_{\infty,0}(T,N)=g_2(T,N)\quad\text{and}\quad\theta_{\infty,1}(T,N)=\left|\frac{T^2-4N}{N}\right| ^{\frac{1}{2}}g_1(T,N).
\]
Hence $\theta_\infty(T,N)=\theta_{\infty,1}(T,N)+\theta_{\infty,0}(T,N)$. Moreover, $\theta_{\infty,\sigma}(T,N)=|T^2-4N|^{\sigma/2}\varphi_\sigma(T,N)$ for $\sigma\in \{0,1\}$, where $\varphi_\sigma(T,N)$ is a smooth function on $\{(T,N)\in \RR^2\,|\,\omega_\infty(T^2-4N)=\sigma\}$ up to boundary, with bounded support.

Also, we set $\theta_\infty^\pm(x)=\theta_\infty(x,\pm 1/4)$, which coincides with the notation in \cite{cheng2025} and \cite{cheng2025b}.

In this paper we will use a variant of $\theta_\infty^{\pm}(x)$. We define
\[
\Theta_\infty^\pm(x)=\theta_\infty\left(\pm 1,\frac{1-x}{4}\right).
\]
and write $\widehat{\Theta}_\infty(x)=\Theta_\infty^+(x)+\Theta_\infty^-(x)$. Also, for any prime $p$, we define
\[
\widehat{\Theta}_p(y)=\int_{\QQ_p^\times}\theta_p\left(z,\frac{z^2(1-y)}{4}\right)\frac{\rmd z}{|z|_p}.
\]

\subsection{Global notations}\label{subsec:global}
For $\nu\in \ZZ^r$, we usually denote $\nu_i$ by its $i^{\mathrm{th}}$ component. We define
\[
q^\nu=q_1^{\nu_1}\dots q_r^{\nu_r}.
\]
%For $\alpha,\beta\in \ZZ^r$, we define $\min\{\alpha,\beta\}\in \ZZ^r$ such that the $i^{\mathrm{th}}$ component of $\min\{\alpha,\beta\}$ is $\min\{\alpha_i,\beta_i\}$.

For any $y=(y_1,\dots,y_r)\in \QQ_{S_\fin}=\QQ_{q_1}\times\dots\times\QQ_{q_r}$ and $N\in \QQ$, we define
\[
\theta_{q}(y,N)=\prod_{i=1}^{r}\theta_{q_i}(y_i,N),\qquad|y|_q=\prod_{i=1}^{r}|y_i|_{q_i},\qquad
|y|_q'=\prod_{i=1}^{r}|y_i|_{q_i}',
\]
\[
\rme_q(y)=\prod_{i=1}^{r}\rme_{q_i}(y_i),\qquad \widehat{\Theta}_q(y)=\prod_{i=1}^{r}\widehat{\Theta}_{q_i}(y).
\]
Also, for $\tau=(\tau_1,\dots,\tau_r)\in \{0,1\}^r$, we define
\[
\theta_{q,\tau}(y,N)=\prod_{i=1}^{r}\theta_{q_i,\tau_i}(y_i,N).
\]
We usually embed $\QQ$ in $\QQ_S$ or $\QQ_{S_\fin}$ diagonally.

\subsection{The functions $F$ and $V$}
Let
\[
F(x)=\frac{1}{2K_0(2)}\int_{x}^{+\infty}\rme^{-t-1/t}\frac{\rmd t}{t},
\]
where $K_s(y)$ is the modified Bessel function of the second kind. The Mellin transform of $F$ is
\begin{equation}\label{eq:mellinf}
\widetilde{F}(s)=\frac{1}{s}\frac{K_s(2)}{K_0(2)}.
\end{equation}
From this one can show that (see Proposition 3.7 in \cite{cheng2025} for example) $\widetilde{F}$ is an odd function, and for $s=\sigma+\rmi t\in \CC$ such that $\sigma_1\leq \sigma\leq \sigma_2$, we have
\begin{equation}\label{eq:frapiddecay}
\widetilde{F}(s)\ll_{\sigma_1,\sigma_2} |s|^{|\sigma|-1}\rme^{-\uppi|t|/2}.
\end{equation}
Also, we define
\begin{equation}\label{eq:defv}
V_{\iota,\epsilon}(x)=V_{\iota,\epsilon,1}(x)=\frac{\uppi^{1/2}}{2\uppi \rmi}\int_{(1)}\widetilde{F}(s)\prod_{i=1}^{r}\frac{1-\epsilon_i q_i^{s-1}}{1-\epsilon_i q_i^{-s}}\frac{\Gamma((\iota+s)/2)}{\Gamma((\iota+1-s)/2)}(\uppi x)^{-s}\rmd s.
\end{equation}
The functions $F$ and $V$ occur in the approximate functional equation in \cite[Section 3]{cheng2025}.

Let $\cS$ be the set of smooth functions $\Phi$ on $\lopen 0,+\infty\ropen$ such that for any $k\in \ZZ_{\geq 0}$, $\Phi^{(k)}(x)$ are of rapid decay as $x\to +\infty$. The following propositions are proved in \cite{cheng2025b}:
\begin{proposition}\label{prop:propertyf}
We have $F\in \cS$. Moreover, $\widetilde{F}(s)$ has a meromorphic continuation to the whole complex plane, holomorphic except for a simple pole at $s=0$, and is of rapid decay in $t$ when $\sigma$ is fixed.
\end{proposition}

\begin{proposition}\label{prop:propertyh}
For any $\iota$ and $\epsilon$, we have $V_{\iota,\epsilon}\in \cS$. Moreover, $\widetilde{V_{\iota,\epsilon}}(s)$ is holomorphic on $\Re s>0$, and is of rapid decay in $t$ when $\sigma>0$ is fixed.
\end{proposition}

\subsection{The generalized Kloosterman sum} %The \emph{partial generalized Kloosterman sum} is defined by
%\[
%\Kl_{k,f}^S(\xi,m)=
%  \sum_{\substack{a \bmod kf^2\\ a^2-4m\equiv 0\,(f^2)}}\legendresymbol{(a^2-4m)/f^2}{k}\rme\legendresymbol{a\xi}{kf^2}\rme_{q}\legendresymbol{a\xi}{kf^2},
%\]
%for $f,k\in \ZZ_{(S)}$ and $\xi,m\in \ZZ^S$. 

 The \emph{partial generalized Kloosterman sum} is defined by
\[
\Kl_{k,f}^S(\xi,m)=
  \sum_{\substack{a \bmod kf^2\\ a^2-4m\equiv 0\,(f^2)}}\legendresymbol{(a^2-4m)/f^2}{k}\rme\legendresymbol{a\xi}{kf^2}\rme_{q}\legendresymbol{a\xi}{kf^2},
\]
for $k,f\in \ZZ_{(S)}$ and $\xi,m\in \ZZ^S$. 

For any prime $p\notin S$ and $\xi,m\in \ZZ_p$, we define the \emph{local generalized Kloosterman sum} to be
\begin{equation}\label{eq:localgeneralizedkloosterman}
\Kl_{p^u,p^v}^{(p)}(\xi,m)=
\sum_{\substack{a \bmod p^{u+2v}\\ a^2-4m\equiv 0\,(p^{2v})}}\legendresymbol{(a^2-4m)/p^{2v}}{p^u}\rme_p\legendresymbol{-a\xi}{p^{u+2v}}.
\end{equation}

We have proved in \cite{cheng2025b} that
\begin{equation}\label{eq:localgeneralizedkloostermaneq0}
\Kl_{p^u,p^v}^{(p)}(\xi,m)=
\sum_{\substack{a \bmod p^{u+2v}\\ a^2-4m\equiv 0\,(p^{2v})}}\legendresymbol{(a^2-4m)/p^{2v}}{p^u} \rme\legendresymbol{a\xi}{p^{u+2v}}\rme_{q}\legendresymbol{a\xi}{p^{u+2v}}.
\end{equation}

\subsection{Norm with respect to the Mellin transform}\label{subsec:mellinnorm}
 Fix $\delta>0$. We consider functions $G$ in $C_c^\infty(\lopen \delta^{-1},\delta\ropen)$.

\begin{lemma}\label{lem:grapiddeacy}
Let $\widetilde{G}(u)$ be the Mellin transform of $G$, namely
\[
\widetilde{G}(u)=\int_{0}^{+\infty}G(x)x^{u}\frac{\rmd x}{x}.
\]
Then $\widetilde{G}(u)$ is well defined on the whole complex plane. Moreover $\widetilde{G}$ has rapid decay vertically. More precisely, for any $\sigma\in \RR$ and $A\geq 0$,
\[
\widetilde{G}(\sigma+\rmi t)\ll_{\sigma,A,\delta}\|G\|_{\lceil A\rceil,1}(1+|t|)^{-A}.
\]
\end{lemma}
\begin{proof}
The first assertion holds since $G\in C_c^\infty(\lopen \delta^{-1},\delta\ropen)$. Now we prove the second assertion. We have
\begin{equation}\label{eq:mellingestimate}
\widetilde{G}(\sigma+\rmi t)=\int_{0}^{+\infty}G(x)x^{\sigma+\rmi t}\frac{\rmd x}{x}=\int_{\RR}G(\rme^a)\rme^{a\sigma}\rme^{\rmi at}\rmd a.
\end{equation}
For any $N\in \ZZ_{\geq 0}$, we use integration by parts $N$ times and obtain
\[
\widetilde{G}(\sigma+\rmi t)=\frac{1}{(\rmi t)^N}\int_{\RR}\frac{\rmd^N}{\rmd a^N}(G(\rme^a)\rme^{a\sigma})\rme^{\rmi at}\rmd a.
\] 
Using induction we know that for any $n\in \ZZ_{>0}$,
\[
\frac{\rmd^n}{\rmd a^n}G(\rme^a)=\sum_{j=0}^{n}\sstirling{n}{j}\rme^{ja}G^{(j)}(\rme^a),
\]
where $\sstirling{n}{j}\in \ZZ_{\geq 0}$ depend only on $j$ and $n$ (they are actually the \emph{Stirling numbers of the second kind}). Hence
\[
\frac{\rmd^N}{\rmd a^N}(G(\rme^a)\rme^{a\sigma})= \sum_{n=0}^{N}\binom{N}{n}\sum_{j=0}^{n}\sstirling{n}{j}\rme^{ja}G^{(j)}(\rme^a)\rme^{a\sigma}\sigma^{N-n}
\]

Since $G\in C_c^\infty(\lopen\delta^{-1},\delta\ropen)$, for $|t|\gg 1$ we have
\begin{align*}
|\widetilde{G}(\sigma+\rmi t)|&\leq \frac{1}{|t|^N}\int_{\RR} \sum_{n=0}^{N}\binom{N}{n}\sum_{j=0}^{n}\sstirling{n}{j}|G^{(j)}(\rme^a)|\rme^{a\sigma+aj}|\sigma|^{N-n}\rmd a\\
&\ll_{\sigma,N} \frac{1}{|t|^N}\int_{0}^{+\infty}\sum_{n=0}^{N}\binom{N}{n}\sum_{j=0}^{n}\sstirling{n}{j} |G^{(j)}(x)|x^{\sigma+j}\frac{\rmd x}{x}\\
&\ll_{\sigma,N,\delta}\frac{1}{|t|^N}\int_{0}^{+\infty}\sum_{j=0}^{N}|G^{(j)}(x)|\rmd x \ll_{\sigma,N,\delta}\|G\|_{N,1}(1+|t|)^{-N}.
\end{align*}
If $|t|\ll 1$ we just bound trivially on \eqref{eq:mellingestimate} and obtain the same bound as above. Now we take $N=\lceil A\rceil$ and obtain the result.
\end{proof}

For $G\in C_c^\infty(\lopen\delta^{-1},\delta\ropen)$ and $K,\sigma\in \RR$, we define the norm
\[
\|G\|_{M^K_\sigma}=\int_{(\sigma)}|\widetilde{G}(u)|(1+|u|)^{K}\rmd |u|.
\]
\begin{corollary}\label{cor:mellinnorm}
For any $G\in C_c^\infty(\lopen\delta^{-1},\delta\ropen)$ and $K\geq -1$, we have
\[
\|G\|_{M^K_\sigma}\ll \|G\|_{\lfloor K\rfloor +2,1},
\]
and for $K<-1$ we have
\[
\|G\|_{M^K_\sigma}\ll \|G\|_{1},
\]
where the implied constants only depend on $K,\sigma$ and $\delta$.
\end{corollary}
\begin{proof}
The first assertion follows by taking $A=K+1+\varepsilon$ in the above lemma, where $\varepsilon>0$ is sufficiently small. For the second assertion, we have
\[
\|G\|_{M^K_\sigma}=\int_{(\sigma)}|\widetilde{G}(u)|(1+|u|)^{K}\rmd |u|\ll_{K} \sup_{\Re u=\sigma} |\widetilde{G}(u)|.
\]
By taking $N=0$ in the proof of the above lemma we know that the right hand side is $\ll_{\sigma,\delta} \|G\|_{1}$, as desired.
\end{proof}

\subsection{The height on $\ZZ^S$}
For $\xi\in \ZZ^S$ and $a\in \RR$, we define
\[
\llbracket a \star \xi\rrbracket=(1+|a\xi|_\infty)\prod_{i=1}^{r}(1+|\xi|_{q_i}),
\]
which is $\asymp|a\xi|$ if $\xi\in \ZZ-\{0\}$ and $a\gg 1$.
Note that if $a\in \ZZ_{(S)}$, then $\llbracket a \star \xi\rrbracket=\llbracket 1 \star a\xi\rrbracket$.

We give some lemmas involving the estimates of the sums of such heights. See Lemma 5.9 and Lemma 5.11 of \cite{cheng2025b} for the proofs of the lemmas below.

\begin{lemma}\label{lem:estimatexilargesum}
Suppose that $M>1$ and $\varepsilon>0$.
\begin{enumerate}[itemsep=0pt,parsep=0pt,topsep=0pt, leftmargin=0pt,labelsep=2.5pt,itemindent=15pt,label=\upshape{(\arabic*)}]
\item We have
\[
\sum_{\xi\in \ZZ^S-\{0\}}\llbracket a\star \xi\rrbracket^{-M} \ll_{\varepsilon} |a|^{-M+\varepsilon}.
\]
\item We have
\[
\sum_{\llbracket a\star\xi\rrbracket\gg b}\llbracket a\star\xi\rrbracket^{-M}\ll_\varepsilon |a|^{-1}|b|^{-M+1+\varepsilon}.
\]
\end{enumerate}
\end{lemma}

\begin{lemma}\label{lem:estimatexismallsum}
Suppose that $b\gg 1$ and $\alpha\in \lopen -\infty, -1\ropen$. Then we have
\[
\sum_{\substack{\xi\in \ZZ^S\\\llbracket a\star \xi\rrbracket \ll b}}\llbracket a\star \xi\rrbracket^{-\alpha}\ll  \frac{b^{1-\alpha}}{a}\log^r b.
\]
\end{lemma}

\section{A second Poisson summation}\label{sec:secondpoisson}
Fix $\vartheta\in \lopen 0,1\ropen$. We will consider $\vartheta$ as a parameter for the contribution of the elliptic part. For simplicity we set $\vartheta'=1-\vartheta$. In \cite[Section 4]{cheng2025} we computed that
\[
\Sigma^n(\xi)=\sum_{\pm}\sum_{\nu\in \ZZ^r}\sum_{k,f\in \ZZ_{(S)}^{>0}}\frac{1}{k^2f^3}\sum_{\xi\in \ZZ^S}\Kl_{k,f}^S(\xi, \pm nq^\nu)I_{k,f}(\xi, \pm nq^\nu),
\]
where 
\begin{align*}
&I_{k,f}(\xi,m)=2|4m|^{1/2}\int_{x\in\RR}\int_{y\in\QQ_{S_\fin}}\theta_\infty^{\sgn m}(x)\theta_{q}(y, m)\left[F\legendresymbol{kf^2|4m|^{-\vartheta}}{|x^2\mp 1|_\infty^{\vartheta}|y^2- 4m|_q'^{\vartheta}}\right.\\
+&\left.\frac{kf^2|4m|^{-1/2}}{\sqrt{|x^2\mp 1|_\infty|y^2- 4m|_q'}}V\legendresymbol{kf^2|4m|^{-\vartheta'}}{|x^2\mp 1|_\infty^{\vartheta'}|y^2- 4m|_q'^{\vartheta'}}\right]\rme\legendresymbol{-x\xi |4m|^{1/2}}{kf^2}\rme_{q}\legendresymbol{-y\xi}{kf^2}\rmd x\rmd y.
\end{align*}
After making the change of variable $x\mapsto |4m|^{-1/2}x$, and recalling that $\theta_\infty^\pm(x)=\theta_\infty(x,\pm 1/4)$, we obtain
\begin{align*}
&I_{k,f}(\xi,m)=2\int_{x\in\RR}\int_{y\in\QQ_{S_\fin}}\theta_\infty(x,m)\theta_{q}(y, m)\left[F\legendresymbol{kf^2}{|x^2-4m|_\infty^{\vartheta}|y^2-4m|_q'^{\vartheta}}\right.\\
+&\left.\frac{kf^2}{\sqrt{|x^2-4m|_\infty|y^2-4m|_q'}}V\legendresymbol{kf^2}{|x^2-4m|_\infty^{\vartheta'} |y^2- 4m|_q'^{\vartheta'}}\right]\rme\legendresymbol{-x\xi }{kf^2}\rme_{q}\legendresymbol{-y\xi}{kf^2}\rmd x\rmd y.
\end{align*}

Let $K=\lfloor\log_2 X\rfloor$. Then we have
\begin{align*}
  S(X)=\sum_{\substack{n<X\\\gcd(n,S)=1}}\Sigma^n(\xi) & =\sum_{j=1}^{K}\sum_{\substack{2^{j-1}\leq n<2^{j}\\n\in \ZZ_{(S)}^{>0}}}\sum_{\pm}\sum_{\nu\in \ZZ^r}\sum_{k,f\in \ZZ_{(S)}^{>0}}\sum_{\xi\in \ZZ^S}\frac{1}{k^2f^3}\Kl_{k,f}^S(\xi,\pm nq^\nu)I_{k,f}(\xi,\pm nq^\nu)\\
  &+\sum_{\substack{2^{K}\leq n<X\\n\in \ZZ_{(S)}^{>0}}}\sum_{\pm}\sum_{\nu\in \ZZ^r}\sum_{k,f\in \ZZ_{(S)}^{>0}}\sum_{\xi\in \ZZ^S}\frac{1}{k^2f^3}\Kl_{k,f}^S(\xi,\pm nq^\nu)I_{k,f}(\xi,\pm nq^\nu).
\end{align*}
Now we define
\[
S(X/2,X)=\sum_{\substack{X/2\leq n<X\\n\in \ZZ_{(S)}^{>0}}}\sum_{\pm}\sum_{\nu\in \ZZ^r}\sum_{k,f\in \ZZ_{(S)}^{>0}}\sum_{\xi\in \ZZ^S}\frac{1}{k^2f^3}\Kl_{k,f}^S(\xi,\pm nq^\nu)I_{k,f}(\xi,\pm nq^\nu).
\]
Since we have a canonical decomposition $\ZZ^{S}-\{0\}=\{\pm 1\}\times \ZZ_{(S)}^{>0}\times q^{\ZZ^r}$, we obtain
\[
S(X/2,X)=\sum_{\substack{X/2\leq |n^{(q)}|<X\\n\in \ZZ^S}}\sum_{k,f\in \ZZ_{(S)}^{>0}}\sum_{\xi\in \ZZ^S}\frac{1}{k^2f^3}\Kl_{k,f}^S(\xi,n)I_{k,f}(\xi,n).
\]
Similarly, the last term can be written as
\[
S(2^K,X)=\sum_{\substack{2^{K}\leq |n^{(q)}|<X\\n\in \ZZ^S}}\sum_{k,f\in \ZZ_{(S)}^{>0}}\sum_{\xi\in \ZZ^S}\frac{1}{k^2f^3}\Kl_{k,f}^S(\xi,n)I_{k,f}(\xi,n).
\]

We will approximate the following two sums by using 
\[
S_G(X)=\sum_{\substack{n=1\\\gcd(n,S)=1}}^{+\infty}G\legendresymbol{n}{X}\Sigma^n(\xi)=\sum_{n\in \ZZ^S}G\legendresymbol{|n^{(q)}|}{X}\sum_{k,f\in \ZZ_{(S)}^{>0}}\sum_{\xi\in \ZZ^S}\frac{1}{k^2f^3}\Kl_{k,f}^S(\xi,n)I_{k,f}(\xi,n)
\]
that is roughly $\triv_{[1/2,1]}$ in the first case and is roughly $\triv_{[2^K/X,1]}$ in the second case. 

The main task of the following sections is to estimate $S_G(X)$ (see \autoref{thm:asymptoticsg} for the result). We assume that $G\in C_c^\infty(\lopen 1/4,5/4\ropen)$ in the following.

\subsection{A second Poisson summation}
Recall that
\[
\Kl_{k,f}^S(\xi,n)=
  \sum_{\substack{a \bmod kf^2\\ a^2-4n\equiv 0\,(f^2)}}\legendresymbol{(a^2-4n)/f^2}{k}\rme\legendresymbol{a\xi}{kf^2}\rme_{q}\legendresymbol{a\xi}{kf^2}.
\]
Hence $\Kl_{k,f}^S(\xi,n)$ is $kf^2$-periodic with respect to the variable $n$. 

Also, recall that
\begin{align*}
&I_{k,f}(\xi,n)=2\int_{x\in\RR}\int_{y\in\QQ_{S_\fin}}\theta_\infty(x,n)\theta_{q}(y, n)\left[F\legendresymbol{kf^2}{|x^2-4n|_\infty^{\vartheta}|y^2-4n|_q'^{\vartheta}}\right.\\
+&\left.\frac{kf^2}{\sqrt{|x^2-4n|_\infty|y^2- 4n|_q'}}V\legendresymbol{kf^2}{|x^2-4n|_\infty^{\vartheta'}|y^2- 4n|_q'^{\vartheta'}}\right]\rme\legendresymbol{-x\xi}{kf^2}\rme_{q}\legendresymbol{-y\xi}{kf^2}\rmd x\rmd y.
\end{align*}

Consider the following function defined on $(a,b)=(a,b_1,\dots,b_r)\in \QQ_S=\RR\times \QQ_{q_1}\times\dots\times \QQ_{q_r}$:
\begin{align*}
&J_{k,f}(X,\xi,\eta,a,b)=2G\legendresymbol{|a|_\infty|b|_q}{X}\int_{x\in\RR}\int_{y\in\QQ_{S_\fin}} \theta_\infty(x,a)\theta_{q}(y,b)\left[F\legendresymbol{kf^2}{|x^2-4a|_\infty^{\vartheta}|y^2- 4b|_q'^{\vartheta}}\right.\\
+&\left.\frac{kf^2}{\sqrt{|x^2-4a|_\infty|y^2- 4b|_q'}}V\legendresymbol{kf^2}{|x^2-4a|_\infty^{\vartheta'}|y^2- 4b|_q'^{\vartheta'}}\right]\rme\legendresymbol{-x\xi}{kf^2}\rme_{q}\legendresymbol{-y\eta}{kf^2}\rmd x\rmd y.
\end{align*}
\begin{lemma}\label{lem:secondsmooth}
$J_{k,f}(X,\xi,\eta,a,b)$ is a smooth and compactly supported function for $(a,b)\in \QQ_S$.
\end{lemma}
\begin{proof}
We first prove that the set when $J_{k,f}(X,\xi,\eta,a,b)\neq 0$ for $(a,b)$ is compactly supported. If $J_{k,f}(X,\xi,\eta,a,b)\neq 0$, then we must have $|b_i|_{q_i}\asymp 1$ for all $i$. Since $G\in C_c^\infty(\lopen 1/4,5/4\ropen)$, it follows that $|a|\asymp 1$. Hence $J_{k,f}(X,\xi,\eta,a,b)\neq 0$ for $(a,b)$ is compactly supported.

Next we prove smoothness. Since $J_{k,f}(X,\xi,\eta,a,b)=0$ if $b_i$ nears $0$ for some $i$, it suffices to consider the case $b_i\neq 0$ for all $i$. It suffices to prove that it is smooth with respect to $a\in \RR$ and is smooth with respect to $b\in \QQ_{S_\fin}$ that is uniform in $a$.

For the archimedean case, one may use the argument in \cite[Lemma 4.1]{altug2015} to show that the integral is smooth. Now we consider the nonarchimedean case.

By \autoref{lem:orbitalintegralsmooth} we can find $L_1,\dots,L_r>2$ such that $\theta_{q_i}(c_iy_i,c_i^2b_i)=\theta_{q_i}(y_i,b_i)$. Also we have $|y_i^2-4b_i|_{q_i}'=|c_i^2y_i^2-4c_ib_i|_{q_i}'$ and $\omega_i(c_iy_i,c_i^2b_i)=\omega_i(y_i,b_i)$. (Here $\omega_i(T,N)=\omega_i(T^2-4N)$ if we use the definition of this paper. $\omega_i(y_i)$ in \cite{cheng2025b} is thus $\omega_i(y_i,\pm nq^\nu)$.)

For fixed $\eta$ we can further assume that $\rme_{q_i}(-y_i\eta_i/kf^2)$ is invariant under the transformation $y_i\mapsto c_iy_i$ for all $c_i\in 1+q_i^{L_i}\ZZ_{q_i}$.

Hence the function defining  $J_{k,f}(X,\xi,\eta,a,b)$ is invariant under the transformation $y_i\mapsto c_iy_i$ and $b_i\mapsto c_i^2b_i$. Making the change of variable $y_i\mapsto c_i^{-1}y_i$, we obtain $J_{k,f}(X,\xi,\eta,a,b)=J_{k,f}(X,\xi,\eta,a,c_ib)$, where $c_i\in \QQ_{q_i}\hookrightarrow \QQ_S$. Since $|b_i|\asymp 1$ we obtain smoothness.
\end{proof}

Clearly $|n^{(q)}|=|n|_\infty|n|_q$. Hence
\[
S_G(X)=\sum_{k,f\in \ZZ_{(S)}^{>0}}\frac{1}{k^2f^3}\sum_{\xi,n\in \ZZ^S}J_{k,f}(X,\xi,n)\Kl_{k,f}^S(\xi,n),
\]
where $J_{k,f}(X,\xi,n)=J_{k,f}(X,\xi,\xi,n,n)$.

By \autoref{lem:secondsmooth} we can apply the Poisson summation formula to the above sum. Since $\Kl_{k,f}^S(\xi,n)$ is $kf^2$-periodic, by Corollary B.4 of \cite{cheng2025} we obtain
\[
S_G(X)=\sum_{k,f\in \ZZ_{(S)}^{>0}}\frac{1}{k^3f^5}\sum_{\xi,\alpha\in \ZZ^S}\widehat{J}_{k,f}(X,\xi,\alpha)\widehat{\Kl}_{k,f}^S(\xi,\alpha)
\]
and the sum converges absolutely, where
\begin{align*}
&\widehat{J}_{k,f}(X,\xi,\eta,\alpha,\beta)=2\int_{(x,a)\in\RR^2}\int_{(y,b)\in\QQ_{S_\fin}^2}G\legendresymbol{|a|_\infty|b|_q}{X} \theta_\infty(x,a)\theta_{q}(y,b)\left[F\legendresymbol{kf^2}{|x^2-4a|_\infty^{\vartheta}|y^2- 4b|_q'^{\vartheta}}\right.\\
+&\left.\frac{kf^2}{\sqrt{|x^2-4a|_\infty|y^2- 4b|_q'}}V\legendresymbol{kf^2}{|x^2-4a|_\infty^{\vartheta'}|y^2- 4b|_q'^{\vartheta'}}\right]\rme\legendresymbol{-x\xi-a\alpha}{kf^2}\rme_{q}\legendresymbol{-y\eta-b\beta}{kf^2}\rmd x\rmd a\rmd y\rmd b,
\end{align*}
$\widehat{J}_{k,f}(X,\xi,\alpha)=\widehat{J}_{k,f}(X,\xi,\xi,\alpha,\alpha)$, and the \emph{transformed Kloosterman sum} (for the standard representation) is defined as
\[
\widehat{\Kl}_{k,f}^S(\xi,\alpha)=\sum_{m\bmod kf^2}\Kl_{k,f}^S(\xi,m)\rme\legendresymbol{m\alpha}{kf^2}\rme_q\legendresymbol{m\alpha}{kf^2}.
\]

\subsection{The critical transformation on the Hitchin-Steinberg base}
We make the change of variable 
\begin{equation}\label{eq:harishchandratransform}
x\mapsto a,\quad a\mapsto \frac{1}{4}a^2-x\qquad\text{and}\qquad y\mapsto b,\quad b\mapsto \frac{1}{4}b^2-y
\end{equation}
in $\widehat{J}_{k,f}(X,\xi,\eta,\alpha,\beta)$. Since making change of variable in the $p$-adic case satisfies a similar formula as in the real case \cite{weil1982adele}, we can proceed as in the real case. 

The Jacobians of the transformations are both $1$. Hence $\widehat{J}_{k,f}(X,\xi,\eta,\alpha,\beta)$ becomes
\begin{equation}\label{eq:transformedj}
\begin{split}
2\int_{(x,a)\in\RR^2}\int_{(y,b)\in\QQ_{S_\fin}^2}& G\left(\frac{1}{X}\left|\frac{a^2}{4}-x\right|_\infty\left|\frac{b^2}{4}-y\right|_q \right) \theta_\infty\left(a,\frac{a^2}{4}-x\right)\theta_{q}\left(b,\frac{b^2}{4}-y\right)\\
\times&\left[F\legendresymbol{kf^2}{|4x|_\infty^{\vartheta}|4y|_q'^{\vartheta}}
+\frac{kf^2}{\sqrt{|4x|_\infty|4y|_q'}}V\legendresymbol{kf^2}{|4x|_\infty^{\vartheta'}|4y|_q'^{\vartheta'}}\right] \\
\times&\rme\legendresymbol{-a\xi-a^2\alpha/4+x\alpha}{kf^2}\rme_q\legendresymbol{-b\eta-b^2\beta/4+y\beta}{kf^2}\rmd x\rmd a\rmd y\rmd b.
\end{split}
\end{equation}

We will see in the following sections that the above change of variable simplifies our computation.
\begin{remark}
$\widehat{J}_{k,f}$ is an integral over the $\QQ_S$-point of the $\QQ$-scheme $\mf{g}\sslash \G\cong\mf{a}\sslash W\cong\mathbf{A}^2$ (which is the \emph{Hitchin-Steinberg base}), where $\mathbf{A}$ denotes the affine space, $\mf{g}$ denotes the Lie algebra of $\G$, $\mf{a}$ denotes the Lie algebra of the diagonal torus $\A$ of $\G$ and $W$ denotes the Weyl group. The identification is given by the trace and the determinant. This transformation essentially changes the coordinate $(T,N)$ to $(T,\Delta)$, where $\Delta$ denotes the discriminant. %Recall that we have the following Harish-Chandra isomorphism
%\[
%\mf{a}\sslash W\cong \Spec(Z(U\mf{g})),
%\]
%where $Z(U\mf{g})$ denotes center of the universal enveloping algebra of $\mf{g}=\mathfrak{gl}_2$. More precisely, $Z(U\mf{g})=\QQ[Z,\Delta]$, where $Z$ is the identity matrix and $\Delta$ is the Casimir element. The Harish-Chandra isomorphism with respect to the Hitchin-Steinberg basis $\{T,N\}$ on $\mf{a}\sslash W$ and the basis $\{Z,\Delta\}$ on $\Spec(Z(U\mf{g}))$ is essentially \eqref{eq:harishchandratransform}, hence the name of the transformation. The  Harish-Chandra transformation "changes the base of integration from the space $\mf{a}\sslash W$ to $\Spec(Z(U\mf{g}))$".

%Such method might work for more general algebraic groups by constructing certain Casimir elements in $Z(U\mf{g})=\QQ[Z]\otimes Z(U(\mf{sl}_n))$ as the basis of $\Spec(Z(U\mf{g}))$.
\end{remark}

\section{Computation of the transformed Kloosterman sum}
In this section, we compute the transformed Kloosterman sum $\widehat{\Kl}_{k,f}^S(\xi,\alpha)$. Recall that
\[
\widehat{\Kl}_{k,f}^S(\xi,\alpha)=\sum_{m\bmod kf^2}\Kl_{k,f}^S(\xi,m)\rme\legendresymbol{m\alpha}{kf^2} \rme_q\legendresymbol{m\alpha}{kf^2}.
\]
For any prime $p\notin S$ and $\xi,\alpha\in \ZZ_p$, we define
\begin{equation}\label{eq:deflocalkloostermancharacter}
\widehat{\Kl}_{p^u,p^v}^{(p)}(\xi,\alpha)=\sum_{m\bmod p^{u+2v}}\Kl_{p^u,p^v}^{(p)}(\xi,m)\rme_p\legendresymbol{-m\alpha}{p^{u+2v}}.
\end{equation}
Suppose that  $\alpha\in \ZZ^S$. Since $\rme(x)\prod_{p}\rme_p(x)=1$ for $x\in \QQ$ and $m\alpha/p^{u+2v}\in \ZZ_{\ell}$ if $\ell\notin S\cup\{p\}$, we have
\[
\rme_p\legendresymbol{-m\alpha}{p^{u+2v}}=\prod_{\ell\notin S}\rme_{\ell}\legendresymbol{m\alpha}{p^{u+2v}}^{-1}= \rme\legendresymbol{m\alpha}{p^{u+2v}}\rme_{q}\legendresymbol{m\alpha}{p^{u+2v}}.
\]
Hence for $\alpha\in \ZZ^S$, we have
\begin{equation}\label{eq:localgeneralizedkloostermaneq}
\widehat{\Kl}_{p^u,p^v}^{(p)}(\xi,\alpha)=
\sum_{m \bmod p^{u+2v}}\Kl_{p^u,p^v}^{(p)}(\xi,m) \rme\legendresymbol{m\alpha}{p^{u+2v}}\rme_{q}\legendresymbol{m\alpha}{p^{u+2v}}.
\end{equation}

\begin{proposition}\label{prop:prodkloostermanlocal}
We have
\[
\widehat{\Kl}_{k,f}^S(\xi,\alpha)=\prod_{p\notin S}\widehat{\Kl}_{k_{(p)},f_{(p)}}^{(p)}\left(((kf^2)^{(p)})^{-1}\xi,((kf^2)^{(p)})^ {-1}\alpha\right),
\]
where $(a^{(p)})^{-1}$ denotes the inverse of $a^{(p)}$ modulo $a_{(p)}$.
\end{proposition}
\begin{proof}
Let
\[
kf^2=\prod_{j=1}^{t}(kf^2)_{(p_j)}
\]
be the prime factorization. By Chinese remainder theorem, we have an isomorphism
\[
\begin{split}
   \varphi: \prod_{j=1}^{t}\ZZ/(kf^2)_{(p_j)}&\to \ZZ/kf^2 \\
     (a_1,\dots,a_t) &\mapsto \sum_{j=1}^{t} a_j(kf^2)^{(p_j)}((kf^2)^{(p_j)})^{-1}.
\end{split}
\]
Note that $\varphi(a_1,\dots,a_t)\equiv a_j\pmod{(kf^2)_{(p_j)}}$.

Since
\[
\frac{\varphi(a_1,\dots,a_t)\alpha}{kf^2}=\frac{\alpha}{kf^2}\sum_{j=1}^{t} a_j(kf^2)^{(p_j)}((kf^2)^{(p_j)})^{-1}=\sum_{j=1}^{t} \frac{a_j((kf^2)^{(p_j)})^{-1}\alpha}{(kf^2)_{(p_j)}},
\]
we obtain
\[
\rme\legendresymbol{\varphi(a_1,\dots,a_t)\alpha}{kf^2}=\prod_{j=1}^{t} \rme\legendresymbol{a_j((kf^2)^{(p_j)})^{-1}\alpha}{(kf^2)_{(p_j)}}.
\]
By Proposition B.1 in \cite{cheng2025b} we have
\[
\Kl_{k,f}^S(\xi,\varphi(a_1,\dots,a_t) )=\prod_{j=1}^{t}\Kl_{k_{(p_j)},f_{(p_j)}}^{(p_j)}(((kf^2)^{(p_j)})^{-1}\xi,\varphi(a_1,\dots,a_t) )=\prod_{j=1}^{t}\Kl_{k_{(p_j)},f_{(p_j)}}^{(p_j)}(((kf^2)^{(p_j)})^{-1}\xi, a_j).
\]
Hence by \eqref{eq:localgeneralizedkloostermaneq},
\begin{align*}
  \widehat{\Kl}_{k,f}^S(\xi,\alpha)&=\sum_{\substack{a_j \bmod (kf^2)_{(p_j)}\\ 1\leq j\leq t}}\Kl_{k,f}^S(\xi,\varphi(a_1,\dots,a_t))\rme\legendresymbol{\varphi(a_1,\dots,a_t) \alpha}{kf^2} \rme_q\legendresymbol{\varphi(a_1,\dots,a_t) \alpha}{kf^2}\\
   & =\sum_{\substack{a_j \bmod (kf^2)_{(p_j)}\\ 1\leq j\leq t}} \prod_{j=1}^{t} \Kl_{k_{(p_j)},f_{(p_j)}}^{(p_j)}(((kf^2)^{(p_j)})^{-1}\xi, a_j) \rme\legendresymbol{a_j((kf^2)^{(p_j)})^{-1}\alpha}{(kf^2)_{(p_j)}} \rme_q\legendresymbol{a_j((kf^2)^{(p_j)})^{-1}\alpha}{(kf^2)_{(p_j)}}\\
   & = \prod_{j=1}^{t}\widehat{\Kl}_{k_{(p_j)},f_{(p_j)}}^{(p_j)}\left(((kf^2)^{(p_j)})^{-1}\xi, ((kf^2)^{(p_j)})^{-1}\alpha\right).\qedhere
\end{align*}
\end{proof}

Hence it suffices to consider the local sums $\widehat{\Kl}_{p^u,p^v}^{(p)}(\xi,\alpha)$.
The local computations were mainly done by Altu\u{g} \cite{altug2020} and we only need to do slight modifications in the ramified case.
\begin{proposition}\label{prop:kloostermancongruent}
Suppose that $\xi,\alpha\in \ZZ_p$. Let $\xi',\alpha'\in \ZZ$ such that $\xi\equiv \xi'\pmod {p^{u+2v}}$ and $\alpha\equiv \alpha'\pmod {p^{u+2v}}$. Then
\[
\widehat{\Kl}_{p^u,p^v}^{(p)}(\xi,\alpha)=\widehat{\Kl}_{p^u,p^v}^{(p)}(\xi',\alpha')=
                               \sum_{m \bmod p^{u+2v}}\Kl_{p^u,p^v}^{(p)}(\xi',m) \rme\legendresymbol{m\alpha'}{p^{u+2v}}.
\]
\end{proposition}
The right hand side is precisely the character sum defined by Altu\u{g} in \cite[Section 5.2]{altug2020}.
\begin{proof}
Since $\alpha\equiv \alpha'\pmod {p^{u+2v}}$, we obtain
\[
\widehat{\Kl}_{p^u,p^v}^{(p)}(\xi,\alpha)=\sum_{m\bmod p^{u+2v}}\Kl_{p^u,p^v}^{(p)}(\xi,m)\rme_p\legendresymbol{-m\alpha}{p^{u+2v}}=\sum_{m\bmod p^{u+2v}}\Kl_{p^u,p^v}^{(p)}(\xi,m)\rme_p\legendresymbol{-m\alpha'}{p^{u+2v}}.
\]
Since $\xi\equiv \xi'\pmod {p^{u+2v}}$, by Proposition 2.12 of \cite{cheng2025b},
\[
\widehat{\Kl}_{p^u,p^v}^{(p)}(\xi,\alpha)=\sum_{m\bmod p^{u+2v}}\Kl_{p^u,p^v}^{(p)}(\xi',m)\rme_p\legendresymbol{-m\alpha'}{p^{u+2v}} =\widehat{\Kl}_{p^u,p^v}^{(p)}(\xi',\alpha').
\]
Hence the first equality follows. The second equality follows from
\[
\rme_p\left(\frac{-m\alpha'}{p^{u+2v}}\right)=\rme\left(\left\langle \frac{m\alpha'}{p^{u+2v}}\right\rangle_p\right)=\rme\left(\frac{m\alpha'}{p^{u+2v}}\right).\qedhere
\]
\end{proof}

\begin{proposition}\label{prop:kloostermanlocalcompute}
Let $p\notin S$ and $\xi\in \ZZ_p$.
\begin{enumerate}[itemsep=0pt,parsep=0pt,topsep=0pt,leftmargin=0pt,labelsep=2.5pt,itemindent=15pt,label=\upshape{(\arabic*)}]
\item We have
\begin{equation}\label{eq:kloostermancompute0}
\widehat{\Kl}_{p^u,p^v}^{(p)}(\xi,0)= \begin{dcases}
p^{u+2v}\bm\phi(p^u), & \text{if $v_p(\xi)\geq u+2v$ and $u$ is even}, \\
                                                        0, & \text{otherwise},
                                                      \end{dcases}
\end{equation}
where $\bm{\phi}(k)$ denotes the Euler totient function.
\item For $\alpha\in \ZZ_p-\{0\}$ we have
\begin{equation}\label{eq:kloostermancomputeneq0}
\widehat{\Kl}_{p^u,p^v}^{(p)}(\xi,\alpha)= \begin{dcases}
c^{(p)}_{p^u,p^v}(\alpha)\rme_p\legendresymbol{\xi^2}{\alpha p^{u+2v}}, & \text{if $v_p(\xi)\geq\min\{v_p(\alpha),u+2v\}$}, \\
                                                        0, & \text{otherwise},
                                                      \end{dcases}
\end{equation}
where $c^{(p)}_{p^u,p^v}(\alpha)$ is defined by
\begin{equation}\label{eq:kloostermancomputecell}
p^{\frac{u+2v+w}{2}}\kappa(p^{u+2v-w}) \legendresymbol{p^{-v_p(\alpha)}\alpha}{p^{u+2v-w}}\cdot \begin{dcases}
\bm\phi(p^u), & \text{$v_p(\alpha)\geq u$, $u\equiv 0\,(2)$}, \\
-p^{u-1},&\text{$v_p(\alpha)= u-1$, $u\equiv 0\,(2)$},\\
\legendresymbol{-p^{-v_p(\alpha)}\alpha}{p},&\text{$v_p(\alpha)= u-1$, $u\equiv 1\,(2)$},\\
                                                        0, & \text{otherwise},
                                                      \end{dcases}
\end{equation}
where $w=\min\{v_p(\alpha),u+2v\}$ and
\[
\kappa(m)=\begin{cases}
            1, & m\equiv 1\,(4), \\
            \rmi, & m\equiv 3\,(4).
          \end{cases}
\]
\end{enumerate} 
\end{proposition}
\begin{proof}
Choose $\xi',\alpha'\in \ZZ$ such that $\xi\equiv \xi'\pmod {p^{u+2v}}$ and $\alpha\equiv \alpha'\pmod {p^{u+2v}}$. By \autoref{prop:kloostermancongruent}, 
\[
\widehat{\Kl}_{p^u,p^v}^{(p)}(\xi,0)=\widehat{\Kl}_{p^u,p^v}^{(p)}(\xi',0) \quad\text{and}\quad 
\widehat{\Kl}_{p^u,p^v}^{(p)}(\xi,\alpha)=\widehat{\Kl}_{p^u,p^v}^{(p)}(\xi',\alpha').
\]

(1) By the first statement of \cite[Proposition 5.6]{altug2020} we have
\[
\widehat{\Kl}_{p^u,p^v}^{(p)}(\xi,0)=\widehat{\Kl}_{p^u,p^v}^{(p)}(\xi',0)= \begin{dcases}
p^{u+2v}\bm\phi(p^u), & \text{if $v_p(\xi')\geq u+2v$ and $u$ is even}, \\
                                                        0, & \text{otherwise},
                                                      \end{dcases}
\]
which equals the right hand side of \eqref{eq:kloostermancompute0} since $\xi\equiv \xi'\pmod {p^{u+2v}}$.

(2) By the second statement of \cite[Proposition 5.6]{altug2020} we have
\[
\widehat{\Kl}_{p^u,p^v}^{(p)}(\xi',\alpha')= \begin{dcases}
c^{(p)}_{p^u,p^v}(\alpha')\rme\legendresymbol{-\alpha'(\xi_0\alpha_0^{-1})^{2}}{p^{u+2v}}, & \text{if $v_p(\xi')\geq\min\{v_p(\alpha'),u+2v\}$}, \\
                                                        0, & \text{otherwise},
                                                      \end{dcases}
\]
where $\xi'=\gcd(\xi',\alpha')\xi_0$ and $\alpha'=\gcd(\xi',\alpha')\alpha_0$. (Note that there are some misprints in \cite[Proposition 5.6]{altug2020}, some $v_p(\alpha)$ stated in loc. cit. should be $\min\{v_p(\alpha),u+2v\}$ due to the case $v_p(\alpha)\geq u+2v$.)

We have
\[
\rme\legendresymbol{-\alpha'(\xi_0\alpha_0^{-1})^{2}}{p^{u+2v}}= \rme\left(-\left\langle\frac{\alpha'(\xi_0\alpha_0^{-1})^{2}}{p^{u+2v}}\right\rangle_p\right)= \rme_p\legendresymbol{\alpha'(\xi_0\alpha_0^{-1})^{2}}{p^{u+2v}}=\rme_p\legendresymbol{\xi'^2}{\alpha'p^{u+2v}}.
\]
Since  $\xi\equiv \xi'\pmod {p^{u+2v}}$ and $\alpha\equiv \alpha'\pmod {p^{u+2v}}$,
\[
\rme_p\legendresymbol{\xi'^2}{\alpha'p^{u+2v}}=\rme_p\legendresymbol{\xi^2}{\alpha p^{u+2v}}.
\]
By \eqref{eq:kloostermancomputecell}, $c^{(p)}_{p^u,p^v}(\alpha)=c^{(p)}_{p^u,p^v}(\alpha')$. Recall that $\alpha\equiv \alpha'\pmod {p^{u+2v}}$. Hence $\widehat{\Kl}_{p^u,p^v}^{(p)}(\xi,\alpha)$ equals the right hand side of \eqref{eq:kloostermancomputeneq0}.
\end{proof} 

Now we can derive global computations of the transformed Kloosterman sum.
\begin{corollary}\label{cor:kloostermanlocalcompute}
Let $\xi,\alpha\in \ZZ^S$ and $k,f\in \ZZ_{(S)}^{>0}$.
\begin{enumerate}[itemsep=0pt,parsep=0pt,topsep=0pt,leftmargin=0pt,labelsep=2.5pt,itemindent=15pt,label=\upshape{(\arabic*)}]
\item We have
\[
\widehat{\Kl}_{k,f}^{S}(\xi,0)= \begin{dcases}
kf^2\bm{\phi}(k), & \text{if $kf^2\mid \xi$ and $k$ is a square}, \\
                                                        0, & \text{otherwise},
                                                      \end{dcases}
\]
\item If $\alpha\neq 0$, we have
\[
\widehat{\Kl}_{k,f}^S(\xi,\alpha)= \begin{dcases}
c_{k,f}(\alpha)\rme\legendresymbol{-\xi^2}{\alpha kf^2}\rme_q\legendresymbol{-\xi^2}{\alpha kf^2}, & \text{if $\gcd(\alpha,kf^2)\mid \xi$}, \\
                                                        0, & \text{otherwise},
                                                      \end{dcases}
\]
where
\[
c_{k,f}(\alpha)=\prod_{p\notin S}c^{(p)}_{k_{(p)},f_{(p)}}\left(((kf^2)^{(p)})^{-1}\alpha\right).
\]
\end{enumerate} 
\end{corollary}
\begin{proof}
(1) By \autoref{prop:prodkloostermanlocal} we have
\[
\widehat{\Kl}_{k,f}^S(\xi,0)=\prod_{p\notin S}\widehat{\Kl}_{k_{(p)},f_{(p)}}^{(p)}\left(((kf^2)^{(p)})^{-1}\xi,0\right).
\]
Since $v_p(\xi)=v_p(((kf^2)^{(p)})^{-1}\xi)$, by \autoref{prop:kloostermanlocalcompute} (1),
\[
\widehat{\Kl}_{k_{(p)},f_{(p)}}^{(p)}\left(((kf^2)^{(p)})^{-1}\xi,0\right)=(kf^2)_{(p)} \bm{\phi}(k_{(p)})
\]
if $v_p(\xi)\geq v_p(kf^2)$ and $v_p(k)$ is even, and is $0$ otherwise. By taking product over $p$, 
\[
\widehat{\Kl}_{k,f}^S(\xi,0)=\prod_{p\notin S}(kf^2)_{(p)} \bm\phi(k_{(p)})=kf^2\bm{\phi}(k)
\]
if $kf^2\mid \xi$ and $k$ is a square, and is $0$ otherwise.

(2) By \autoref{prop:prodkloostermanlocal} we have
\[
\widehat{\Kl}_{k,f}^S(\xi,\alpha)=\prod_{p\notin S}\widehat{\Kl}_{k_{(p)},f_{(p)}}^{(p)}\left(((kf^2)^{(p)})^{-1}\xi,((kf^2)^{(p)})^{-1}\alpha \right).
\]
Since $v_p(\xi)=v_p(((kf^2)^{(p)})^{-1}\xi)$ and $v_p(\alpha)=v_p(((kf^2)^{(p)})^{-1}\alpha)$, by \autoref{prop:kloostermanlocalcompute} (2),
\[
\widehat{\Kl}_{k_{(p)},f_{(p)}}^{(p)}\left(((kf^2)^{(p)})^{-1}\xi,((kf^2)^{(p)})^{-1}\alpha\right) =c^{(p)}_{p^u,p^v}(\alpha) \rme_p\legendresymbol{((kf^2)^{(p)})^{-2}\xi^2}{((kf^2)^{(p)})^{-1}\alpha(kf^2)_{(p)}} =c^{(p)}_{k_{(p)},f_{(p)}}(\alpha) \rme_p\legendresymbol{\xi^2}{\alpha kf^2}
\]
if $v_p(\xi)\geq \min\{v_p(\alpha),v_p(kf^2)\}$ and is $0$ otherwise. Hence by taking product over $p$, 
\[
\widehat{\Kl}_{k,f}^S(\xi,\alpha)=\prod_{p\notin S}c^{(p)}_{k_{(p)},f_{(p)}}(\alpha) \rme_p\legendresymbol{\xi^2}{\alpha kf^2}=c_{k,f}(\alpha)\rme\legendresymbol{-\xi^2}{\alpha kf^2}\rme_q\legendresymbol{-\xi^2}{\alpha kf^2}
\]
if $\gcd(\alpha,kf^2)\mid \xi$, and is $0$ otherwise, where in the last step we used the fact that $\rme(x)\rme_q(x)\prod_{p\notin S}\rme_p(x)=1$ for $x\in \QQ$.
\end{proof}

Finally we bound $c_{k,f}(\alpha)$ for $\alpha\neq 0$.

\begin{proposition}\label{prop:estimateckf}
Suppose that $\alpha\in \ZZ^S$ with $\alpha\neq 0$, $k,f\in \ZZ_{(S)}^{>0}$. Then we have
\[
|c_{k,f}(\alpha)|\leq\begin{dcases}
                    k^{3/2}f\sqrt{\gcd(kf^2,\alpha)}, & \frac{k}{\rad k}\mid \alpha, \\
                    0, & \text{otherwise},
                  \end{dcases}
\]
where for $a\in \ZZ_{>0}$,
\[
\rad a=\prod_{p\mid a}p.
\]
\end{proposition}
\begin{proof}
We have
\[
c_{k,f}(\alpha)=\prod_{p\mid kf^2}c^{(p)}_{k_{(p)},f_{(p)}}\left(((kf^2)^{(p)})^{-1}\alpha\right).
\]
Since $v_p(((kf^2)^{(p)})^{-1}\alpha)=v_p(\alpha)$, by \eqref{eq:kloostermancomputecell} we obtain
\[
\left|c^{(p)}_{k_{(p)},f_{(p)}}\left(((kf^2)^{(p)})^{-1}\alpha\right)\right|\leq (kf^2)_{(p)}^{1/2}\gcd(kf^2,\alpha)_{(p)}^{1/2} k_{(p)}=k_{(p)}^{3/2}f_{(p)}\gcd(kf^2,\alpha)_{(p)}^{1/2} 
\]
if $v_p(\alpha)\geq v_p(k)-1$, and is $0$ otherwise. Therefore
\[
|c_{k,f}(\alpha)|\leq \prod_{p\mid kf^2}k_{(p)}^{3/2}f_{(p)}\gcd(kf^2,\alpha)_{(p)}^{1/2}= k^{3/2}f\sqrt{\gcd(kf^2,\alpha)}
\]
if $\frac{k}{\rad k}\mid \alpha$, and is $0$ otherwise.
\end{proof}

\section{Analysis of the main term}
Recall that
\[
S_G(X)=\sum_{k,f\in \ZZ_{(S)}^{>0}}\frac{1}{k^3f^5}\sum_{\xi,\alpha\in \ZZ^S}\widehat{J}_{k,f}(X,\xi,\alpha)\widehat{\Kl}_{k,f}^S(\xi,\alpha).
\]
In this section, we will analyze the $\alpha=0$ term, namely
\[
S_G^{\alpha=0}(X)=\sum_{k,f\in \ZZ_{(S)}^{>0}}\frac{1}{k^3f^5}\sum_{\xi\in \ZZ^S}\widehat{J}_{k,f}(X,\xi,0)\widehat{\Kl}_{k,f}^S(\xi,0).
\]
By \autoref{cor:kloostermanlocalcompute} (1), we obtain
\begin{equation}\label{eq:alpha0}
S_G^{\alpha=0}(X)=\sum_{\substack{k,f\in \ZZ_{(S)}^{>0}\\k=\square}}\frac{\bm{\phi}(k)}{k^2f^3}\sum_{\xi\in kf^2\ZZ^S}\widehat{J}_{k,f}(X,\xi,0).
\end{equation}

We will give an estimate of $S_G^{\alpha=0}(X)$ in this section. 
\begin{theorem}\label{thm:contributealpha0}
For any $A>0$ and $\varepsilon>0$, we have the following asymptotic formula for $S_G^{\alpha=0}(X)$:
\begin{align*}
  S_G^{\alpha=0}(X)=&\int_{\RR}\int_{\QQ_{S_\fin}} \widetilde{G}\legendresymbol{3}{2}|1-x|_\infty^{-\frac32} |1-y|_q^{-\frac32}\widehat{\Theta}_\infty(x)\widehat{\Theta}_{q}(y)\zeta^S(2)\rmd x\rmd yX^{\frac32}+\mf{M}(G)\\
+&O(X^{-A}\|G\|_{\lfloor \frac{15}{2}+2A\rfloor+2,1}+X^{\frac12+\varepsilon}\|G\|_1),
\end{align*}
where
\begin{equation}\label{eq:formulaag}
\mf{M}(G)=\mf{B}\vartheta'(\widetilde{G}'(1)X+ \widetilde{G}(1)X\log X)+(\mf{C}+\mf{D}\vartheta')\widetilde{G}(1)X
\end{equation}
with
\[
\mf{B}=-2^{r}\prod_{i=1}^{r}(1-q_i^{-1})\int_{X_0}\int_{Y_\mathbf{1}} |1-x|_\infty^{-1}|1-y|_q^{-1}\widehat{\Theta}_\infty(x)\widehat{\Theta}_{q}(y) |x|_\infty^{-\frac{1}{2}}|y|_q'^{-\frac{1}{2}}\rmd x\rmd y
\]
and
\begin{align*}
  \mf{C}=&-\frac12\left(2\upgamma-2\log (2\uppi)+\sum_{i=1}^{r}(1+q_i^{-1})\frac{\log q_i}{1-q_i^{-1}}\right)\mf{B}\\
-&2^{r-1}\prod_{i=1}^{r}(1-q_i^{-1})\uppi\int_{X_1}\int_{Y_{\mathbf{1}}} |1-x|_\infty^{-1} |1-y|_q^{-1}\widehat{\Theta}_\infty(x)\widehat{\Theta}_{q}(y) |x|_\infty^{-\frac{1}{2}}|y|_q'^{-\frac{1}{2}}\rmd x\rmd y-2^{r}\prod_{i=1}^{r}(1-q_i^{-1})\\
\times&\sum_{i=1}^{r}\sum_{\epsilon_i\in\{0,-1\}} \frac{1-\epsilon_iq_i^{-1}}{1-\epsilon_i}\frac{\log q_i}{1-q_i^{-1}}\int_{X_0}\int_{Y_{\epsilon^i}}|1-x|_\infty^{-1} |1-y|_q^{-1}\widehat{\Theta}_\infty(x)\widehat{\Theta}_{q}(y) |x|_\infty^{-\frac{1}{2}}|y|_q'^{-\frac{1}{2}} \rmd x\rmd y,
\end{align*}
\begin{align*}
\mf{D}=&2^{r}\prod_{i=1}^{r}(1-q_i^{-1})\int_{X_0}\int_{Y_{\mathbf{1}}}(\log|1-x|_\infty+\log |1-y|_q)|1-x|_\infty^{-1} |1-y|_q^{-1} \widehat{\Theta}_\infty(x)\widehat{\Theta}_{q}(y) |x|_\infty^{-\frac{1}{2}}|y|_q'^{-\frac{1}{2}}\rmd x\rmd y\\
-&2^{r}\prod_{i=1}^{r}(1-q_i^{-1})\int_{X_0}\int_{Y_{\mathbf{1}}}|1-x|_\infty^{-1} |1-y|_q^{-1}\widehat{\Theta}_\infty(x)\widehat{\Theta}_{q}(y)(\log |x|_\infty+\log |y|'_q )|x|_\infty^{-\frac{1}{2}}|y|_q'^{-\frac{1}{2}}\rmd x\rmd y,
\end{align*}
$\mathbf{1}=(1,\dots,1)$, $\epsilon^i=(1,\dots,\epsilon_i,\dots,1)$ with $\epsilon_i$ in the $i^{\mathrm{th}}$ place, and
\[
\widehat{\Theta}_\infty(x)=\sum_{\pm}\theta_\infty\left(\pm 1 ,\frac{1-x}{4}\right),
\]
\[
\widehat{\Theta}_{q_i}(y_i)=\int_{\QQ_{q_i}}\frac{1}{|z_i|_{q_i}}\theta_{q_i}\left( z_i,\frac{z_i^2}{4}(1-y_i)\right)\rmd z_i,\qquad \widehat{\Theta}_{q}(y)=\prod_{i=1}^{r}\widehat{\Theta}_{q_i}(y_i).
\]
The implied constant only depends on $A$, $\vartheta$ and $\varepsilon$.
\end{theorem}

We need some preparations to derive \autoref{thm:contributealpha0}. The reader may go directly to the proof of \autoref{thm:contributealpha0} and go back to the proof of the lemmas below if necessary.

\subsection{Preparations}
\begin{lemma}\label{lem:poisson2smooth}
For any $\alpha\in \ZZ^S$, the function
\begin{align*}
\Phi_\alpha(a,b)=&2\int_{x\in\RR}\int_{y\in\QQ_{S_\fin}} G\left(\frac{1}{X}\left|\frac{a^2}{4}-x\right|_\infty\left|\frac{b^2}{4}-y\right|_q \right) \theta_\infty\left(a,\frac{a^2}{4}-x\right)\theta_{q}\left(b,\frac{b^2}{4}-y\right)\\
\times&\left[F\legendresymbol{kf^2}{|4x|_\infty^{\vartheta}|4y|_q'^{\vartheta}}
+\frac{kf^2}{\sqrt{|4x|_\infty|4y|_q'}}V\legendresymbol{kf^2}{|4x|_\infty^{\vartheta'}|4y|_q'^{\vartheta'}}\right]\rme \legendresymbol{\alpha x}{kf^2}\rme_q\legendresymbol{\alpha y}{kf^2} \rmd x\rmd y
\end{align*}
is smooth and compactly supported on $(a,b)\in \QQ_S$.
\end{lemma}
\begin{proof}
Suppose that $\Phi_\alpha(a,b)\neq 0$.
Since $\theta_q(\gamma)$ is compactly supported, we obtain
\[
|b|\ll 1\quad\text{and}\quad\left|\frac{b^2}{4}-y\right|_q\asymp 1.
\]
Since $G$ is compactly supported on $\lopen 1/4,5/4\ropen$,
\[
\left|\frac{a^2}{4}-x\right|_\infty\asymp 1.
\]
Since $\theta_\infty(T,N)$ is compactly supported modulo center, we obtain $|a|\ll 1$. Hence $\Phi_\alpha(a,b)$ is compactly supported.

Next we prove that $\Phi_\alpha(a,b)$ is smooth. It suffices to prove that it is smooth with respect to $a\in \RR$ and is smooth with respect to $b\in \QQ_{S_\fin}$ that is uniform in $a$.

By the results of \autoref{subsec:singularities}, $\theta_\infty(a,a^2/4-x)$ is smooth on $(a,x)\in \RR^2$ unless $x=0$. If $x=0$, one may use the argument in \cite[Lemma 4.1]{altug2015} to show that the integral is smooth on $(a,x)\in \RR^2$. Hence the integrand is smooth with respect to the archimedean place and thus the first claim follows.

For the second claim we argue as follows.
For any $q_i\in S$, we choose a corresponding $L_i$ in \autoref{lem:orbitalintegralsmooth}. Hence
\[
\theta_{q}\left(cb,\frac{c^2b^2}{4}-c^2y\right)= \theta_{q}\left(b,\frac{b^2}{4}-y\right)
\]
for any $c\in U=\prod_{i=1}^{r} 1+q_i^{L_i}\ZZ_{q_i}$, where the multiplication on $\QQ_{S_\fin}$ is defined componentwise.

By making the change of variable $y\mapsto c^{-2}y$ and noting that for any $c\in U$, we have
\[
|c^2y|_q=|y|_q\qquad\text{and}\qquad |c^2y|'_q=|y|'_q,
\]
we find that $\Phi_\alpha(a,b)=\Phi_\alpha(a,cb)$ for $c\in U$.

If $b\neq 0$, then there exists a neighborhood $V$ of $b$ (independent of $a$) such that $b'/b\in U$ for all $b'\in V$, so that
\[
\Phi_\alpha(a,b)=\Phi_\alpha(a,b')
\]
for all $b'\in V$. Hence $\Phi_\alpha(a,b)$ is smooth at $b$ which is independent of $a$.

Now suppose that $b=0$. Since $\theta_{q_i}(\gamma)$ is compactly supported, $|b_i^2/4-y_i|_{q_i}\asymp 1$ for all $i$ if $y_i$ is in the support. Hence $|y_i|_{q_i}\asymp 1$ for all $i$ in this case. By \autoref{subsec:singularities}, $\theta_{q_i}(\gamma)$ and $|b^2/4-y|_q$ are smooth away from the parabola $\{(b,y)\in \QQ_{S_\fin}^2\,|\,b^2/4-y=0\}$. Hence there exists a neighborhood $V$ of $b$ such that
\[
\theta_{q}\left(b,\frac{b^2}{4}-y\right)=\theta_{q}\left(b',\frac{b'^2}{4}-y\right)\quad\text{and}\quad \left|\frac{b^2}{4}-y\right|_q=\left|\frac{b'^2}{4}-y\right|_q
\]
for all $b'\in V$ and all $|y|_q\asymp 1$. Hence $\Phi_\alpha(a,b)$ is smooth at $b$ which is independent of $a$. Therefore the second claim follows and hence $\Phi_\alpha(a,b)$ is smooth.
\end{proof}

\begin{lemma}\label{lem:zetaphi}
Suppose that $\Re s>2$. Then we have
\begin{equation}\label{eq:seriesphi}
\sum_{\substack{k,f\in \ZZ_{(S)}^{>0}\\ k=\square}}\frac{\bm{\phi}(k)}{k^{1+s}f^{1+2s}}=\zeta^S(2s).
\end{equation}
\end{lemma}
\begin{proof}
We have
\[
\sum_{u=0}^{+\infty}\frac{\bm\phi(p^{2u})}{p^{2u(1+s)}}=1+\sum_{u=1}^{+\infty} \frac{p^{2u}-p^{2u-1}}{p^{2u(s+1)}} =1+\frac{(1-p^{-1})p^{-2s}}{1-p^{-2s}}=\frac{1-p^{-2s-1}}{1-p^{-2s}}
\]
Hence for $\Re s>2$,
\[
\sum_{\substack{k\in \ZZ_{(S)}^{>0}\\ k=\square}}\frac{\bm{\phi}(k)}{k^{1+s}}=\prod_{p\notin S}\sum_{u=0}^{+\infty}\frac{\bm\phi(p^{2u})}{p^{2u(1+s)}}=\prod_{p\notin S}\frac{1-p^{-2s-1}}{1-p^{-2s}}=\frac{\zeta^S(2s)}{\zeta^S(2s+1)}.
\]
The lemma now follows from
\[
\sum_{f\in \ZZ_{(S)}^{>0}}\frac{1}{f^{1+2s}}=\zeta^S(1+2s).\qedhere
\]
\end{proof}

\subsection{Proof of the asymptotic formula of the $\alpha=0$ term}
We now give the proof of the main theorem \autoref{thm:contributealpha0} in this section.

\begin{proof}[Proof of \autoref{thm:contributealpha0}]
By \eqref{eq:alpha0} and the expression \eqref{eq:transformedj} of $\widehat{J}_{k,f}(X,\xi,0)$ after the transformation \eqref{eq:harishchandratransform}, we have
\begin{align*}
&S_G^{\alpha=0}(X)=2\sum_{\substack{k,f\in \ZZ_{(S)}^{>0}\\k=\square}}\frac{\bm{\phi}(k)}{k^2f^3}\sum_{\xi\in kf^2\ZZ^S}\int_{(x,a)\in\RR^2}\int_{(y,b)\in\QQ_{S_\fin}^2} G\left(\frac{1}{X}\left|\frac{a^2}{4}-x\right|_\infty\left|\frac{b^2}{4}-y\right|_q \right) \theta_\infty\left(a,\frac{a^2}{4}-x\right)\\
\times&\theta_{q}\left(b,\frac{b^2}{4}-y\right)\left[F\legendresymbol{kf^2}{|4x|_\infty^{\vartheta} |4y|_q'^{\vartheta}}
+\frac{kf^2}{\sqrt{|4x|_\infty|4y|_q'}}V\legendresymbol{kf^2}{|4x|_\infty^{\vartheta'}|4y|_q'^{\vartheta'}}\right] \rme\legendresymbol{-a\xi}{kf^2}\rme_q\legendresymbol{-b\xi}{kf^2}\rmd x\rmd a\rmd y\rmd b.
\end{align*}

By \autoref{lem:poisson2smooth}, the integral over $(x,y)\in \RR\times \QQ_{S_\fin}$ of the function inside is smooth in $(a,b)\in \RR\times \QQ_{S_\fin}$. Hence we can use the Poisson summation formula in the reverse direction and obtain
\begin{align*}
S_G^{\alpha=0}(X)=&2\sum_{\substack{k,f\in \ZZ_{(S)}^{>0}\\k=\square}}\frac{\bm{\phi}(k)}{k^2f^3}\sum_{n\in \ZZ^S}\int_{\RR}\int_{\QQ_{S_\fin}} G\left(\frac{1}{X}\left|\frac{n^2}{4}-x\right|_\infty\left|\frac{n^2}{4}-y\right|_q \right) \theta_\infty\left(n,\frac{n^2}{4}-x\right)\theta_{q}\left(n,\frac{n^2}{4}-y\right)\\
\times&\left[F\legendresymbol{kf^2}{|4x|_\infty^{\vartheta}|4y|_q'^{\vartheta}}
+\frac{kf^2}{\sqrt{|4x|_\infty|4y|_q'}}V\legendresymbol{kf^2}{|4x|_\infty^{\vartheta'}|4y|_q'^{\vartheta'}}\right] \rmd x\rmd y.
\end{align*}
By making the change of variable $x\mapsto n^2 x/4$ and $y\mapsto n^2 y/4$ and using the fact that $|a^2y|_p'=|a|_p^2|y|_p'$ for any $a,y\in \QQ_p$, we obtain
\begin{align*}
S_G^{\alpha=0}(X)=&2\sum_{\substack{k,f\in \ZZ_{(S)}^{>0}\\k=\square}}\frac{\bm{\phi}(k)}{k^2f^3}\sum_{n\in \ZZ^S}\int_{\RR}\int_{\QQ_{S_\fin}} \left|\frac{n^2}{4}\right|_\infty\left|\frac{n^2}{4}\right|_q G\left(\frac{1}{X}\left|\frac{n^2}{4}\right|_\infty\left|\frac{n^2}{4}\right|_q |1-x|_\infty |1-y|_q \right)\\
\times&\theta_\infty\left(n,\frac{n^2}{4}(1-x)\right) \theta_{q}\left(n,\frac{n^2}{4}(1-y)\right)\left[F\legendresymbol{kf^2} {|n|_\infty^{2\vartheta}|n|_q^{2\vartheta}|x|_\infty^{\vartheta}|y|_q'^{\vartheta}}\right.\\
+&\left.\frac{kf^2}{|n|_\infty|n|_q\sqrt{|x|_\infty|y|_q'}}V\legendresymbol{kf^2}{|n|_\infty^{2\vartheta'} |n|_q^{2\vartheta'}|x|_\infty^{\vartheta'}|y|_q'^{\vartheta'}}\right] \rmd x\rmd y\\
=&2\sum_{\substack{k,f\in \ZZ_{(S)}^{>0}\\k=\square}}\frac{\bm{\phi}(k)}{k^2f^3}\sum_{\pm}\sum_{\nu\in \ZZ^r}\sum_{n\in \ZZ_{(S)}^{>0}}\int_{\RR}\int_{\QQ_{S_\fin}} n^2 G\left(\frac{n^2}{X} |1-x|_\infty |1-y|_q \right)\theta_\infty\left(\pm nq^\nu,\frac{n^2q^{2\nu}}{4}(1-x)\right)\\
\times& \theta_{q}\left(\pm nq^\nu,\frac{n^2q^{2\nu}}{4}(1-y)\right)\left[F\legendresymbol{kf^2}{n^{2\vartheta}|x|_\infty^{\vartheta} |y|_q'^{\vartheta}}
+\frac{kf^2}{n\sqrt{|x|_\infty|y|_q'}}V\legendresymbol{kf^2}{n^{2\vartheta'}|x|_\infty^{\vartheta'} |y|_q'^{\vartheta'}}\right] \rmd x\rmd y,
\end{align*}
where in the last step we used the factorization $\ZZ^S=\{\pm 1\}\times q^{\ZZ^r}\times \ZZ_{(S)}^{>0}$.

Since $\theta_\infty$ is invariant under the action of $Z_+$, we have
\[
\theta_\infty\left(\pm nq^\nu,\frac{n^2q^{2\nu}}{4}(1-x)\right)=\theta_\infty\left(\pm 1 ,\frac{1-x}{4}\right)
\]
and we define this quantity to be $\Theta_\infty^\pm(x)$.

For the nonarchimedean part, we define
\[
\Theta_q^{\pm,\nu}(z,y)=\theta_{q}\left(\pm zq^\nu,\frac{z^2q^{2\nu}}{4}(1-y)\right)=\prod_{i=1}^{r}\theta_{q_i}\left(\pm z_iq^\nu,\frac{z_i^2q^{2\nu}}{4}(1-y_i)\right)
\]
for $z=(z_1,\dots,z_r)\in \ZZ_{S_\fin}^\times=\ZZ_{q_1}^\times\times\dots\times\ZZ_{q_r}^\times$. 
By \autoref{lem:orbitalintegralsmooth}, there exist $L_1,\dots,L_r\geq 2$ independent of $\pm,\nu$ and $y$, such that for any $c_i\in 1+q_i^{L_i}\ZZ_{q_i}$, 
\[
\theta_{q_i}\left(\pm c_iz_iq^\nu,\frac{c_i^2z_i^2q^{2\nu}}{4}(1-y_i)\right)=\theta_{q_i}\left(\pm z_iq^\nu,\frac{z_i^2q^{2\nu}}{4}(1-y_i)\right).
\]
Hence $\Theta_q^{\pm,\nu}(z,y)$ is invariant under the multiplication of 
\[
U=\prod_{i=1}^{r}(1+q_i^{L_i}\ZZ_{q_i})
\]
with respect to $z$. By Fourier inversion formula,
\[
\Theta_q^{\pm,\nu}(z,y)=\prod_{i=1}^{r}(1-q_i^{-1})^{-1}\sum_{\chi}\widehat{\Theta}_q^{\pm,\nu}(\chi,y)\chi(z),
\]
where the Fourier transform is defined by
\[
\widehat{\Theta}_q^{\pm,\nu}(\chi,y)=\int_{\ZZ_{S_\fin}^\times}\Theta_q^{\pm,\nu}(z,y)\overline{\chi(z)}\rmd z
\]
and $\chi$ runs over all characters on $\ZZ_{S_\fin}^\times$ that are trivial on $U$. Since $\ZZ_{S_\fin}^\times/U$ is finite, the sum over $\chi$ is finite (which is independent of $\pm,\nu$ and $y$). We also consider them as Dirichlet characters such that $\chi(m)=0$ if $\gcd(m,S)\neq 1$. By considering the variable $N$ of $\theta_{q_i}(T,N)$ we know that $\Theta_q^{\pm,\nu}(z,y)$ is compactly supported in $y$ when $\pm$ and $\nu$ are fixed and the support is independent of $z$. Hence so is $\widehat{\Theta}_q^{\pm,\nu}(\chi,y)$ for each $\chi$.

Hence we obtain
\begin{align*}
&S_G^{\alpha=0}(X)
=2\sum_{\chi}\sum_{\substack{k,f\in \ZZ_{(S)}^{>0}\\k=\square}}\frac{\bm{\phi}(k)}{k^2f^3}\sum_{\pm}\sum_{\nu\in \ZZ^r}\sum_{n\in \ZZ_{(S)}^{>0}}\int_{\RR}\int_{\QQ_{S_\fin}} n^2 G\left(\frac{n^2}{X} |1-x|_\infty |1-y|_q \right)\Theta_\infty^\pm(x)\\
\times& \prod_{i=1}^{r}(1-q_i^{-1})^{-1}\widehat{\Theta}_{q}^{\pm,\nu}(\chi,y)\chi(n) \left[F\legendresymbol{kf^2}{n^{2\vartheta}|x|_\infty^{\vartheta} |y|_q'^{\vartheta}}
+\frac{kf^2}{n\sqrt{|x|_\infty|y|_q'}}V\legendresymbol{kf^2}{n^{2\vartheta'}|x|_\infty^{\vartheta'} |y|_q'^{\vartheta'}}\right] \rmd x\rmd y.
\end{align*}

We claim that the sum over $\nu$ is finite in the above sum. Suppose that the integrand is not zero. Since $\theta_\infty$ is compactly supported modulo the center, we must have $|1-x|\asymp 1$. Since $G\in C_c^\infty(\lopen 1/4,5/4\ropen)$, we have $n^2|1-y|_q\asymp 1$ and hence $|1-y|_q\ll 1$. 
Recall that
\[
\Theta_q^{\pm,\nu}(z,y)=\theta_{q}\left(\pm zq^\nu,\frac{z^2q^{2\nu}}{4}(1-y)\right).
\]
Hence $\Theta_q^{\pm,\nu}(z,y)=0$ unless $|q^\nu|_{q_i}\ll 1$ and $|q^{2\nu}(1-y_i)|_{q_i}\asymp 1$. Therefore
\[
|q^{2\nu}|_{q_i}\gg \frac{1}{|1-y_i|_{q_i}}\gg \prod_{j\neq i}|1-y_j|_{q_j}\gg \prod_{j\neq i}|q^{2\nu}|_{q_j}\gg 1.
\]
Hence $q_i^{-\nu_i}\asymp 1$. Therefore, the sum over $\nu$ is finite.

Under the transformation \eqref{eq:harishchandratransform}, the definition of $V$ becomes
\[
V\legendresymbol{kf^2}{n^{2\vartheta'}|x|_\infty^{\vartheta'}|y|_q'^{\vartheta'}} =\frac{\uppi^{1/2}}{\dpii}\int_{(2)} \widetilde{F}(s)\prod_{i=1}^{r}\frac{1-\epsilon_i q_i^{s-1}}{1-\epsilon_i q_i^{-s}}\frac{\Gamma((\iota+s)/2)}{\Gamma((\iota+1-s)/2)} \legendresymbol{\uppi kf^2}{n^{2\vartheta'}|x|_\infty^{\vartheta'}|y|_q'^{\vartheta'}}^{-s}\rmd s,
\]
where 
\[
\iota=\omega_\infty(n^2x)=\omega_\infty(x)=\begin{cases}
             0, & x>0, \\
             1, & x<0,
           \end{cases}
\quad
\text{and}
\quad
\epsilon_i=\omega_i(n^2y_i)=\omega_i(y_i)=\legendresymbol{y_i|y_i|_{q_i}'}{q_i}.
\]
Combining the Mellin inversion formula for $G$ and $F$, we obtain
\begin{align*}
&S_G^{\alpha=0}(X)
=2\sum_{\chi}\sum_{\substack{k,f\in \ZZ_{(S)}^{>0}\\k=\square}}\frac{\bm{\phi}(k)}{k^2f^3}\sum_{\pm}\sum_{\nu\in \ZZ^r}\sum_{n\in \ZZ_{(S)}^{>0}}\int_{\RR}\int_{\QQ_{S_\fin}}n^2\frac{1}{\dpii}\int_{(4)} \widetilde{G}(u)\left(\frac{n^2}{X} |1-x|_\infty |1-y|_q \right)^{-u}\rmd u\\
\times& \Theta_\infty^\pm(x)\prod_{i=1}^{r}(1-q_i^{-1})^{-1} \widehat{\Theta}_{q}^{\pm,\nu}(\chi,y)\chi(n) \left[ \frac{1}{\dpii}\int_{(2)} \widetilde{F}(s)\legendresymbol{kf^2}{n^{2\vartheta}|x|_\infty^{\vartheta}|y|_q'^{\vartheta}}^{-s}\rmd s+\frac{kf^2}{n|x|_\infty^{1/2}|y|_q'^{1/2}}\right.\\
\times&\left.\frac{\uppi^{1/2}}{\dpii}\int_{(2)} \widetilde{F}(s)\prod_{i=1}^{r}\frac{1-\epsilon_i q_i^{s-1}}{1-\epsilon_i q_i^{-s}}\frac{\Gamma((\iota+s)/2)}{\Gamma((\iota+1-s)/2)} \legendresymbol{\uppi kf^2}{n^{2\vartheta'}|x|_\infty^{\vartheta'}|y|_q'^{\vartheta'}}^{-s}\rmd s\right] \rmd x\rmd y\\
=&2\sum_{\chi}\sum_{\substack{k,f\in \ZZ_{(S)}^{>0}\\k=\square}}\sum_{\pm}\sum_{\nu\in \ZZ^r}\sum_{n\in \ZZ_{(S)}^{>0}}\int_{\RR}\int_{\QQ_{S_\fin}}\frac{1}{\dpii}\int_{(4)} \widetilde{G}(u)|1-x|_\infty^{-u} |1-y|_q^{-u}\Theta_\infty^\pm(x)\prod_{i=1}^{r}(1-q_i^{-1})^{-1}\\
\times& \widehat{\Theta}_{q}^{\pm,\nu}(\chi,y)\chi(n) \left[ \frac{1}{\dpii}\int_{(2)} \widetilde{F}(s)\frac{1}{n^{2u-2-2\vartheta s}} \frac{\bm{\phi}(k)}{k^{2+s}f^{3+2s}} |x|_\infty^{\vartheta s}|y|_q'^{\vartheta s}\rmd s+\frac{1}{\dpii}\int_{(2)} \widetilde{F}(s)\uppi^{-s+\frac12}\right.\\
\times&\left.\prod_{i=1}^{r}\frac{1-\epsilon_i q_i^{s-1}}{1-\epsilon_i q_i^{-s}}\frac{\Gamma((\iota+s)/2)}{\Gamma((\iota+1-s)/2)} \frac{1}{n^{2u-1-2\vartheta's}}\frac{\bm{\phi}(k)}{k^{1+s}f^{1+2s}}|x|_\infty^{\vartheta's-\frac12} |y|_q'^{\vartheta's-\frac12}\rmd s\right] X^u\rmd u\rmd x\rmd y.
\end{align*}

Since $\theta_\infty(\gamma)$ is compactly supported modulo center and $\theta_{q_i}(\gamma)$ are compactly supported, the integrand is zero if $x$ is near $1$ or $y_i$ is near $1$ for some $i$. Hence the integrand has no singularities at $x=1$ and $y_i=1$ for any $u$. Also, the sum over $\nu$ is finite and the series
\[
\frac{\chi(n)}{n^{2u-2-2\vartheta s}} \frac{\bm{\phi}(k)}{k^{2+s}f^{3+2s}} \quad\text{and}\quad\frac{\chi(n)}{n^{2u-1-2\vartheta's}} \frac{\bm{\phi}(k)}{k^{1+s}f^{1+2s}} 
\]
both converge absolutely for $\Re u=4$ and $\Re s=2$. Hence the above integrals and sums converge absolutely and we can change the orders of summation and integration. By \autoref{lem:zetaphi} we have
\begin{align*}
S_G^{\alpha=0}(X)
=&2\sum_{\chi}\sum_{\pm}\sum_{\nu\in \ZZ^r}\int_{\RR}\int_{\QQ_{S_\fin}}\frac{1}{\dpii}\int_{(4)} \widetilde{G}(u)|1-x|_\infty^{-u} |1-y|_q^{-u}\Theta_\infty^\pm(x)\prod_{i=1}^{r}(1-q_i^{-1})^{-1} \\
\times&\widehat{\Theta}_{q}^{\pm,\nu}(\chi,y)  \left[ \frac{1}{\dpii}\int_{(2)} \widetilde{F}(s)L^S(2u-2-2\vartheta s,\chi) \zeta^S(2s+2) |x|_\infty^{\vartheta s}|y|_q'^{\vartheta s}\rmd s+\frac{1}{\dpii}\int_{(2)} \widetilde{F}(s)\uppi^{-s+\frac12}\right.\\
\times&\left.\prod_{i=1}^{r}\frac{1-\epsilon_i q_i^{s-1}}{1-\epsilon_i q_i^{-s}}\frac{\Gamma((\iota+s)/2)}{\Gamma((\iota+1-s)/2)} L^S(2u-1-2\vartheta' s,\chi)\zeta^S(2s)|x|_\infty^{\vartheta's-\frac12}|y|_q'^{\vartheta's-\frac12}\rmd s\right] X^u\rmd u\rmd x\rmd y.
\end{align*}

Now we consider separately the following two terms
\begin{equation}\label{eq:alpha0firstterm}
\begin{split}
&\sum_{\chi}\sum_{\pm}\sum_{\nu\in \ZZ^r}\int_{\RR}\int_{\QQ_{S_\fin}}\frac{1}{\dpii}\int_{(4)} \widetilde{G}(u)|1-x|_\infty^{-u} |1-y|_q^{-u}\Theta_\infty^\pm(x)\prod_{i=1}^{r}(1-q_i^{-1})^{-1}\\
\times&\widehat{\Theta}_{q}^{\pm,\nu}(\chi,y) \frac{1}{\dpii}\int_{(2)} \widetilde{F}(s)L^S(2u-2-2\vartheta s,\chi) \zeta^S(2s+2) |x|_\infty^{\vartheta s}|y|_q'^{\vartheta s}\rmd sX^u\rmd u\rmd x\rmd y.
\end{split}
\end{equation}
and
\begin{equation}\label{eq:alpha0secondterm}
\begin{split}
&\sum_{\chi}\sum_{\pm}\sum_{\nu\in \ZZ^r}\int_{\RR}\int_{\QQ_{S_\fin}}\frac{1}{\dpii}\int_{(4)} \widetilde{G}(u)|1-x|_\infty^{-u} |1-y|_q^{-u}\Theta_\infty^\pm(x)\prod_{i=1}^{r}(1-q_i^{-1})^{-1}\widehat{\Theta}_{q}^{\pm,\nu}(\chi,y)\frac{1}{\dpii}\int_{(2)} \widetilde{F}(s) \\
\times&\uppi^{-s+\frac12}\prod_{i=1}^{r}\frac{1-\epsilon_i q_i^{s-1}}{1-\epsilon_i q_i^{-s}}\frac{\Gamma(\frac{\iota+s}{2})}{\Gamma(\frac{\iota+1-s}{2})} L^S(2u-1-2\vartheta's,\chi)\zeta^S(2s)|x|_\infty^{\vartheta's-\frac12}|y|_q'^{\vartheta's-\frac12}\rmd s X^u\rmd u\rmd x\rmd y.
\end{split}
\end{equation}
so that $S_G^{\alpha=0}(X)$ is twice of the sum of \eqref{eq:alpha0firstterm} and \eqref{eq:alpha0secondterm}.

We will use the notation
\[
\widehat{\Theta}_{q}^{\pm}(y)=\sum_{\nu\in \ZZ^r}\widehat{\Theta}_{q}^{\pm,\nu}(\triv,y).
\]
By definition, we have
\[
\widehat{\Theta}_{q}^{\pm}(y)=\sum_{\nu\in \ZZ^r}\int_{\ZZ_{S_\fin}^\times}\Theta_q^{\pm,\nu}(z,y)\rmd z
=\sum_{\nu\in \ZZ^r}\int_{\ZZ_{S_\fin}^\times}\prod_{i=1}^{r}\theta_{q_i}\left(\pm z_iq^\nu,\frac{z_i^2q^{2\nu}}{4}(1-y_i)\right)\rmd z.
\]
By making change of variable $z_i\mapsto \pm \prod_{j\neq i}q^{-\nu_j}z_i$, we obtain
\[
\widehat{\Theta}_{q}^{\pm}(y)=\prod_{i=1}^{r}\left(\sum_{\nu\in \ZZ}\int_{\ZZ_{q_i}^\times}\theta_{q_i}\left(z_iq_i^\nu,\frac{z_i^2q_i^{2\nu}}{4}(1-y_i)\right)\rmd z_i\right)=\prod_{i=1}^{r}\int_{\QQ_{q_i}}\frac{1}{|z_i|_{q_i}}\theta_{q_i}\left( z_i,\frac{z_i^2}{4}(1-y_i)\right)\rmd z_i.
\]
Hence $\widehat{\Theta}_{q}^{+}(y)=\widehat{\Theta}_{q}^{-}(y)=\widehat{\Theta}_{q}(y)$. Also, by definition, $\widehat{\Theta}_\infty(x)=\Theta_\infty^+(x)+\Theta_\infty^-(x)$.

We will prove the following asymptotic formulas to conclude our result.
\begin{lemma}\label{lem:asympalpha0firstterm}
For any $A>0$ and $\varepsilon>0$, \eqref{eq:alpha0firstterm} equals
\begin{align*}
  &\frac14 \prod_{i=1}^{r}(1-q_i^{-1})\int_{\RR}\int_{\QQ_{S_\fin}} \widetilde{G}\legendresymbol{3-\vartheta}{2}|1-x|_\infty^{-\frac{3-\vartheta}{2}} |1-y|_q^{-\frac{3-\vartheta}{2}}\widehat{\Theta}_\infty(x)\widehat{\Theta}_{q}(y) \widetilde{F}\!\left(-\frac12\right) |x|_\infty^{-\frac{\vartheta}{2}}|y|_q'^{-\frac{\vartheta}{2}}X^{\frac{3-\vartheta}{2}}\rmd x\rmd y\\
  +&\int_{\RR}\int_{\QQ_{S_\fin}} \widetilde{G}\legendresymbol{3}{2}|1-x|_\infty^{-\frac32} |1-y|_q^{-\frac32}\widehat{\Theta}_\infty(x)\widehat{\Theta}_{q}(y) \frac12 \zeta^S(2)X^{\frac32}\rmd x\rmd y\\
+&O(X^{-A}\|G\|_{\lfloor \frac{15}{2}+2A\rfloor+2,1}+X^{\frac12+\varepsilon}\|G\|_1),
\end{align*}
where the implied constant only depends on $A$ and $\varepsilon$.
\end{lemma}
\begin{insertproof}
We begin by moving the $s$-contour from $(2)$ to $(3/\vartheta)$. The integrand is holomorphic on $\Re s\geq 2/\vartheta$ and $\Re u\geq 4$ except for the function $L^S(2u-2-2\vartheta s,\chi)$ when $\chi=\triv$, which has a pole at $s=(2u-3)/(2\vartheta)$ with residue
\[
\res_{s=\frac{2u-3}{2\vartheta}}L^S(2u-2-2\vartheta s,\mathbbm{1})=\res_{s=\frac{2u-3}{2\vartheta}}\prod_{i=1}^{r}\left(1-\frac{1}{q_i^{2u-2-2\vartheta s}}\right) \zeta(2u-2-2\vartheta s)=-\frac{1}{2\vartheta}\prod_{i=1}^{r}(1-q_i^{-1}).
\]
Hence by residue formula, we know that \eqref{eq:alpha0firstterm} equals
\begin{align*}
&\sum_{\chi}\sum_{\pm}\sum_{\nu\in \ZZ^r}\int_{\RR}\int_{\QQ_{S_\fin}}\frac{1}{\dpii}\int_{(4)} \widetilde{G}(u)|1-x|_\infty^{-u} |1-y|_q^{-u}\Theta_\infty^\pm(x)\prod_{i=1}^{r}(1-q_i^{-1})^{-1}\\
\times&\widehat{\Theta}_{q}^{\pm,\nu}(\chi,y) \frac{1}{\dpii}\int_{(\frac{3}{\vartheta})} \widetilde{F}(s)L^S(2u-2-2\vartheta s,\chi) \zeta^S(2s+2) |x|_\infty^{\vartheta s}|y|_q'^{\vartheta s}\rmd sX^u\rmd u\rmd x\rmd y\\
  +&\frac{1}{2\vartheta}\sum_{\pm}\sum_{\nu\in \ZZ^r}\int_{\RR}\int_{\QQ_{S_\fin}}\frac{1}{\dpii}\int_{(4)} \widetilde{G}(u)|1-x|_\infty^{-u} |1-y|_q^{-u}\Theta_\infty^\pm(x)\\
\times&\widehat{\Theta}_{q}^{\pm,\nu}(\triv,y)  \widetilde{F}\legendresymbol{2u-3}{2\vartheta} \zeta^S\left(\frac{2u-3}{\vartheta}+2\right) |x|_\infty^{\frac{2u-3}{2}}|y|_q'^{\frac{2u-3}{2}}X^u\rmd u\rmd x\rmd y.
\end{align*}
Note that since we move the contour from the left to the right, the residue formula has a minus sign.

By \eqref{eq:frapiddecay} and \autoref{lem:grapiddeacy}, $\widetilde{F}(s)$ and $\widetilde{G}(u)$ have rapid decay. Since the sums over $\nu$ and $\chi$ are finite, $\theta(\pm 1,(1-x)/4)$ is compactly supported in $x$, and $\widehat{\Theta}_{q}^{\pm,\nu}(\chi,y)$ is compactly supported in $y$. Hence the integrand of the first term above is holomorphic on $\Re s=3/\vartheta$ and $\Re u\leq 4$. Hence we may move the $u$-contour from $(4)$ to $(-A)$ for any $A>0$. 

Using the functional equation of Dirichlet $L$-functions and the Stirling formula, we have
\begin{equation}\label{eq:boundlvertical}
L^S(\sigma+\rmi t,\chi)\ll_\sigma (1+|t|)^{\frac12-\sigma}
\end{equation}
for $\sigma<0$.
Hence for fixed $\chi$ and $\nu$,
\begin{align*}
&\int_{\RR}\int_{\QQ_{S_\fin}}\left|\Theta_\infty^\pm(x)\right||\widehat{\Theta}_{q}^{\pm,\nu}(\chi,y)| |x|_\infty^{4}|y|_q'^{4}\rmd x\rmd y\\
\times&\int_{(-A)} |\widetilde{G}(u)|\int_{(3/\vartheta)}|\widetilde{F}(s)||L^S(2u-2-2\vartheta s,\chi)||\zeta^S(2s+2)| \rmd |s|X^{-A}\rmd |u|\\
\ll &\int_{(-A)} |\widetilde{G}(u)|\int_{(3/\vartheta)}|\widetilde{F}(s)|(1+|2u-2-2\vartheta s|)^{\frac{15}{2}+2A}|\zeta^S(2s+2)|\rmd |s|\rmd |u|X^{-A}\\
\ll &X^{-A}\int_{(-A)} |\widetilde{G}(u)|(1+|2u|)^{\frac{15}{2}+2A}\rmd|u|\int_{(3/\vartheta)}|\widetilde{F}(s)|(1+|2+2\vartheta s|)^{\frac{15}{2}+2A}|\zeta^S(2s+2)|\rmd |s|\\
\ll& X^{-A}\|G\|_{M_{-A}^{15/2+2A}}\ll X^{-A}\|G\|_{\lfloor\frac{15}{2}+2A\rfloor+2,1}
\end{align*}
by \autoref{cor:mellinnorm}. Since the sums over $\chi$ and $\nu$ are finite, the first term is bounded by $ X^{-A}\|G\|_{\lfloor\frac{15}{2}+2A\rfloor+2,1}$.

The second term can be rewritten as
\begin{align*}
&\frac{1}{2\vartheta}\int_{\RR}\int_{\QQ_{S_\fin}}\frac{1}{\dpii}\int_{(4)} \widetilde{G}(u)|1-x|_\infty^{-u} |1-y|_q^{-u}\widehat{\Theta}_\infty(x)\widehat{\Theta}_{q}(y)\\
\times&  \widetilde{F}\legendresymbol{2u-3}{2\vartheta} \zeta^S\left(\frac{2u-3}{\vartheta}+2\right) |x|_\infty^{\frac{2u-3}{2}}|y|_q'^{\frac{2u-3}{2}}X^u\rmd u\rmd x\rmd y.
\end{align*}

Since $(2u-3)/2>-1$ when $\Re u>1/2$, the integrand the second term converges absolutely when $\Re u>1/2$ and defines a meromorphic function there. It is holomorphic except for simple poles at $u=(3-\vartheta)/2$ and $u=3/2$ with 
\[
\res_{u=\frac{3-\vartheta}{2}}\widetilde{F}\legendresymbol{2u-3}{2\vartheta} \zeta^S\left(\frac{2u-3}{\vartheta}+2\right) =\frac{\vartheta}{2}\widetilde{F}\left(-\frac12\right)\prod_{i=1}^{r}(1-q_i^{-1})
\]
and
\[
\res_{u=\frac32}\widetilde{F}\legendresymbol{2u-3}{2\vartheta} \zeta^S\left(\frac{2u-3}{\vartheta}+2\right) =\vartheta\zeta^S(2).
\]
By moving the $u$-contour from $(4)$ to $(\frac12+\varepsilon)$, the second term becomes
\begin{align*}
&\frac{1}{2\vartheta}\int_{\RR}\int_{\QQ_{S_\fin}}\frac{1}{\dpii}\int_{(\frac12+\varepsilon)} \widetilde{G}(u)|1-x|_\infty^{-u} |1-y|_q^{-u}\widehat{\Theta}_\infty(x)\widehat{\Theta}_{q}(y)\\
\times&\widetilde{F}\legendresymbol{2u-3}{2\vartheta} \zeta^S\left(\frac{2u-3}{\vartheta}+2\right) |x|_\infty^{\frac{2u-3}{2}}|y|_q'^{\frac{2u-3}{2}}X^u\rmd u\rmd x\rmd y\\
+&\frac14 \prod_{i=1}^{r}\left(1-\frac{1}{q_i}\right)\int_{\RR}\int_{\QQ_{S_\fin}} \widetilde{G}\legendresymbol{3-\vartheta}{2}|1-x|_\infty^{-\frac{3-\vartheta}{2}} |1-y|_q^{-\frac{3-\vartheta}{2}}\widehat{\Theta}_\infty(x)\widehat{\Theta}_{q}(y) \widetilde{F}\left(-\frac12\right) |x|_\infty^{-\frac{\vartheta}{2}}|y|_q'^{-\frac{\vartheta}{2}}X^{\frac{3-\vartheta}{2}}\rmd x\rmd y\\
+&\frac12 \int_{\RR}\int_{\QQ_{S_\fin}} \widetilde{G}\legendresymbol{3}{2}|1-x|_\infty^{-\frac32} |1-y|_q^{-\frac32}\widehat{\Theta}_\infty(x)\widehat{\Theta}_{q}(y) \zeta^S(2)X^{\frac32}\rmd x\rmd y.
\end{align*}
Since $\widetilde{F}$ has rapid decay vertically, the first term above is bounded by
\begin{align*}
&\int_{\RR}\int_{\QQ_{S_\fin}}\int_{(\frac12+\varepsilon)} \widetilde{G}(u)|1-x|_\infty^{-\frac12-\varepsilon} |1-y|_q^{-\frac12-\varepsilon}|\widehat{\Theta}_\infty(x)||\widehat{\Theta}_{q}(y)|  \widetilde{F}\legendresymbol{2u-3}{2\vartheta} \zeta^S\bigg(\frac{2u-3}{\vartheta}+2\bigg)\\
\times& |x|_\infty^{-1+\frac{\varepsilon}{2}}|y|_q'^{-1+\frac{\varepsilon }{2}}X^{\frac12+\varepsilon}\rmd u\rmd x\rmd y\\
\ll&\int_{(\frac12+\varepsilon)} |\widetilde{G}(u)|\left|\widetilde{F}\legendresymbol{2u-3}{2\vartheta}\right| \left|\zeta^S\bigg(\frac{2u-3}{\vartheta}+2\bigg)\right|X^{\frac12+\varepsilon}\rmd u\ll X^{\frac 12+\varepsilon}\int_{(\frac12+\varepsilon)} |\widetilde{G}(u)|(1+|u|)^{-2}\rmd u\\
\ll& X^{\frac 12+\varepsilon}\|G\|_{1},
\end{align*}
where in the last step we used \autoref{cor:mellinnorm}. Combining all the things we have proved we obtain the asymptotic formula in the lemma.
\end{insertproof}

\begin{lemma}\label{lem:asympalpha0secondterm}
For any $A>0$ and $\varepsilon>0$, \eqref{eq:alpha0secondterm} equals
\begin{align*}
&\frac14 \prod_{i=1}^{r}(1-q_i^{-1})\int_{\RR}\int_{\QQ_{S_\fin}} \widetilde{G}\legendresymbol{3-\vartheta}{2}|1-x|_\infty^{-\frac{3-\vartheta}{2}} |1-y|_q^{-\frac{3-\vartheta}{2}}\widehat{\Theta}_\infty(x)\widehat{\Theta}_{q}(y) \widetilde{F}\legendresymbol{1}{2} |x|_\infty^{-\frac{\vartheta}{2}}|y|_q'^{-\frac{\vartheta}{2}}X^{\frac{3-\vartheta}{2}}\rmd x\rmd y\\
+&\frac12\mf{M}(G)+O(X^{-A}\|G\|_{\lfloor \frac{15}{2}+2A\rfloor+2,1}+X^{\frac12+\varepsilon}\|G\|_1),
\end{align*}
where implied constant depends only on $f_\infty$, $f_{q_i}$, $A$ and $\varepsilon$.
\end{lemma}
\begin{insertproof}
First we move the $s$-contour from $(2)$ to $(7/(2\vartheta'))$. The integrand is holomorphic on $\Re s\geq 2$ and $\Re u\geq 4$ except for the function $L^S(2u-1-2\vartheta' s,\chi)$ when $\chi=\triv$, which has a pole at $s=(u-1)/\vartheta'$ with residue
\[
\res_{s=\frac{u-1}{\vartheta'}}L^S(2u-1-2\vartheta's,\mathbbm{1})=-\frac{1}{2\vartheta'} \prod_{i=1}^{r}(1-q_i^{-1}).
\]
Now by residue formula, \eqref{eq:alpha0secondterm} becomes
\begin{align*}
&\sum_{\chi}\sum_{\pm}\sum_{\nu\in \ZZ^r}\int_{\RR}\int_{\QQ_{S_\fin}}\frac{1}{\dpii}\int_{(4)} \widetilde{G}(u)|1-x|_\infty^{-u} |1-y|_q^{-u}\Theta_\infty^\pm(x)\prod_{i=1}^{r}(1-q_i^{-1})^{-1}\widehat{\Theta}_{q}^{\pm,\nu}(\chi,y)\\
\times&\frac{1}{\dpii}\int_{(\frac{7}{2\vartheta'})} \widetilde{F}(s) \uppi^{-s+\frac12}\prod_{i=1}^{r}\frac{1-\epsilon_i q_i^{s-1}}{1-\epsilon_i q_i^{-s}}\frac{\Gamma(\frac{\iota+s}{2})}{\Gamma(\frac{\iota+1-s}{2})} L^S(2u-1-2\vartheta's,\chi)\zeta^S(2s)|x|_\infty^{\vartheta's-\frac{1}{2}}|y|_q'^{\vartheta's-\frac{1}{2}}\rmd s X^u\rmd u\rmd x\rmd y\\
+&\frac{1}{2\vartheta'}\sum_{\pm}\sum_{\nu\in \ZZ^r}\int_{\RR}\int_{\QQ_{S_\fin}}\frac{1}{\dpii}\int_{(4)} \widetilde{G}(u)|1-x|_\infty^{-u} |1-y|_q^{-u}\Theta_\infty^\pm(x)\widehat{\Theta}_{q}^{\pm,\nu}(\triv,y)\\
\times&\widetilde{F}\legendresymbol{u-1}{\vartheta'} \uppi^{-\frac{u-1}{\vartheta'}+\frac12}\prod_{i=1}^{r}\frac{1-\epsilon_i q_i^{(u-1)/\vartheta'-1}}{1-\epsilon_iq_i^{-(u-1)/\vartheta'}} \frac{\Gamma(\frac{\iota}{2}+\frac{u-1}{2\vartheta'})}{\Gamma(\frac{\iota+1}{2}-\frac{u-1}{2\vartheta'})} \zeta^S\legendresymbol{2u-2}{\vartheta'}|x|_\infty^{\frac{2u-3}{2}}|y|_q'^{\frac{2u-3}{2}}X^u\rmd u\rmd x\rmd y.
\end{align*}
As in the proof of \autoref{lem:asympalpha0firstterm}, the integrand of the first term above is holomorphic on $\Re s=7/(2\vartheta')$ and $\Re u\leq 4$. Hence we may move the $u$-contour from $(4)$ to $(-A)$ for any $A>0$. 

Using \eqref{eq:boundlvertical}, we have
\begin{align*}
&\int_{\RR}\int_{\QQ_{S_\fin}}\frac{1}{\dpii}\int_{(4)} \widetilde{G}(u)|1-x|_\infty^{-u} |1-y|_q^{-u}\Theta_\infty^\pm(x)\prod_{i=1}^{r}(1-q_i^{-1})^{-1}\widehat{\Theta}_{q}^{\pm,\nu}(\chi,y)\\
\times&\frac{1}{\dpii}\int_{(\frac{7}{2\vartheta'})} \widetilde{F}(s) \uppi^{-s+\frac12}\prod_{i=1}^{r}\frac{1-\epsilon_i q_i^{s-1}}{1-\epsilon_i q_i^{-s}}\frac{\Gamma(\frac{\iota+s}{2})}{\Gamma(\frac{\iota+1-s}{2})} L^S(2u-1-2\vartheta's,\chi)\zeta^S(2s)|x|_\infty^{\vartheta's-\frac{1}{2}}|y|_q'^{\vartheta's-\frac{1}{2}}\rmd s X^u\rmd u\rmd x\rmd y\\
\ll&\int_{\RR}\int_{\QQ_{S_\fin}}\int_{(-A)} |\widetilde{G}(u)||1-x|_\infty^{A} |1-y|_q^{A}|\Theta_\infty^\pm(x)|\prod_{i=1}^{r}(1-q_i^{-1})^{-1} |\widehat{\Theta}_{q}^{\pm,\nu}(\chi,y)|\\
\times&\int_{(\frac{7}{2\vartheta'})}|\widetilde{F}(s)|\prod_{i=1}^{r}\left|\frac{1-\epsilon_i q_i^{s-1}}{1-\epsilon_i q_i^{-s}}\right|\left|\frac{\Gamma(\frac{\iota+s}{2})}{\Gamma(\frac{\iota+1-s}{2})}\right| (1+|2u-1-2\vartheta' s|)^{\frac{15}{2}+2A}|\zeta^S(2s)||x|_\infty^{2}|y|_q'^{2}\rmd s X^{-A}\rmd u\rmd x\rmd y\\
\ll&\int_{(-A)} |\widetilde{G}(u)|\int_{(\frac{7}{2\vartheta'})}|\widetilde{F}(s)|\prod_{i=1}^{r}\left|\frac{1-\epsilon_i q_i^{s-1}}{1-\epsilon_i q_i^{-s}}\right|\left|\frac{\Gamma(\frac{\iota+s}{2})}{\Gamma(\frac{\iota+1-s}{2})}\right| \\
\times&(1+|2u|)^{\frac{15}{2}+2A}(1+|2\vartheta's+1|)^{\frac{15}{2}+2A}|\zeta^S(2s)| X^{-A}\rmd |u|\rmd |s|\\
%\ll &\int_{(-A)} |\widetilde{G}(u)|\int_{(5)}|\widetilde{F}(s)|\prod_{i=1}^{r}\left|\frac{1-\epsilon_i q_i^{s-1}}{1-\epsilon_i q_i^{-s}}\right|\left|\frac{\Gamma(\frac{\iota+s}{2})}{\Gamma(\frac{\iota+1-s}{2})}\right| (1+|2u|)^{\frac{13}{2}+2A}(1+|s+1|)^{\frac{13}{2}+2A}|\zeta^S(2s)|\rmd |s| X^{-A}\rmd |u|\\
%\ll &X^{-A}\int_{(-A)} |\widetilde{G}(u)|(1+|2u|)^{\frac{15}{2}+2A}\rmd |u|\int_{(\frac{7}{2\vartheta'})}|\widetilde{F}(s)|\prod_{i=1}^{r}\left|\frac{1-\epsilon_i q_i^{s-1}}{1-\epsilon_i q_i^{-s}}\right|\left|\frac{\Gamma(\frac{\iota+s}{2})}{\Gamma(\frac{\iota+1-s}{2})}\right| (1+|2\vartheta's+1|)^{\frac{15}{2}+2A}|\zeta^S(2s)|\rmd |s|\\
\ll &X^{-A}\|G\|_{M_{-A}^{15/2+2A}}\ll X^{-A}\|G\|_{\lfloor\frac{15}{2}+2A\rfloor+2,1},
\end{align*}
where in the last step we used \autoref{cor:mellinnorm}. Since the sums over $\chi$ and $\nu$ are finite, the first term is bounded by $ X^{-A}\|G\|_{\lfloor\frac{15}{2}+2A\rfloor+2,1}$.

Next we analyze the second term. Let $\Xi(u)$ denote the integrand, which can be rewritten as
\begin{align*}
\Xi(u)=&\int_{\RR}\int_{\QQ_{S_\fin}} \widetilde{G}(u)|1-x|_\infty^{-u} |1-y|_q^{-u}\widehat{\Theta}_\infty(x)\widehat{\Theta}_{q}(y)\widetilde{F}\legendresymbol{u-1}{\vartheta'} \\
\times& \uppi^{-\frac{u-1}{\vartheta'}+\frac12}\frac{\Gamma(\frac{\iota}{2}+\frac{u-1}{2\vartheta'})}{\Gamma(\frac{\iota+1}{2}-\frac{u-1}{2\vartheta'})} \prod_{i=1}^{r}\frac{1-\epsilon_i q_i^{(u-1)/\vartheta'-1}}{1-\epsilon_iq_i^{-(u-1)/\vartheta'}} \zeta^S\legendresymbol{2u-2}{\vartheta'}|x|_\infty^{\frac{2u-3}{2}}|y|_q'^{\frac{2u-3}{2}}X^u\rmd x\rmd y.
\end{align*}

As in the proof of \autoref{lem:asympalpha0firstterm}, the integrand $\Xi(u)$ of the second term above converges absolutely when $\Re u>1/2$ and defines a meromorphic function there. It is holomorphic except for a simple pole at $u=(2+\vartheta')/2$ and a pole at $u=1$. 
The residue at $u=(2+\vartheta')/2$ can be easily computed. We have
\[
\res_{u=\frac{2+\vartheta'}{2}}\widetilde{F}\legendresymbol{u-1}{\vartheta'} \zeta^S\legendresymbol{2u-2}{\vartheta'}=\frac{\vartheta'}{2}\widetilde{F}\legendresymbol{1}{2} \prod_{i=1}^{r}(1-q_i^{-1}).
\]

Moving the $u$-contour from $(4)$ to $(\frac12+\varepsilon)$, by residue theorem the second term becomes
\begin{align*}
&\frac{1}{2\vartheta'}\int_{\RR}\int_{\QQ_{S_\fin}}\frac{1}{\dpii}\int_{(\frac12+\varepsilon)} \widetilde{G}(u)|1-x|_\infty^{-u} |1-y|_q^{-u}\Theta_\infty^\pm(x)\widehat{\Theta}_{q}^{\pm,\nu}(\triv,y)\\
\times&\widetilde{F}\legendresymbol{u-1}{\vartheta'} \uppi^{-\frac{u-1}{\vartheta'}+\frac12}\prod_{i=1}^{r}\frac{1-\epsilon_i q_i^{(u-1)/\vartheta'-1}}{1-\epsilon_iq_i^{-(u-1)/\vartheta'}} \frac{\Gamma(\frac{\iota}{2}+\frac{u-1}{2\vartheta'})}{\Gamma(\frac{\iota+1}{2}-\frac{u-1}{2\vartheta'})} \zeta^S\legendresymbol{2u-2}{\vartheta'}|x|_\infty^{\frac{2u-3}{2}}|y|_q'^{\frac{2u-3}{2}}X^u\rmd u\rmd x\rmd y\\
+& \frac14 \prod_{i=1}^{r}(1-q_i^{-1})\int_{\RR}\int_{\QQ_{S_\fin}} \widetilde{G}\legendresymbol{2+\vartheta'}{2}|1-x|_\infty^{-\frac{2+\vartheta'}{2}} |1-y|_q^{-\frac{2+\vartheta'}{2}}\widehat{\Theta}_\infty(x)\widehat{\Theta}_{q}(y) \widetilde{F}\legendresymbol{1}{2} |x|_\infty^{\frac{\vartheta'-1}{2}}|y|_q'^{\frac{\vartheta'-1}{2}}X^{\frac{2+\vartheta'}{2}}\rmd x\rmd y\\
+&\frac{1}{2\vartheta'}\res_{u=1}\Xi(u).
\end{align*}
Since $\widetilde{F}$ has rapid decay vertically, the first term above is bounded by $X^{\frac12+\varepsilon}$ as in the proof of \autoref{lem:asympalpha0firstterm}. The second term above is precisely the first term in this lemma since $\vartheta'=1-\vartheta$.
 
Our final task is to analyze the residue at $u=1$. We have 
\[
\Xi(u)=\sum_{\iota\in \{0,1\}}\sum_{\epsilon\in \{0,\pm 1\}^r}\Xi_{\iota,\epsilon}(u),
\]
where
\begin{align*}
\Xi_{\iota,\epsilon}(u)=&\int_{X_\iota}\int_{Y_\epsilon} \widetilde{G}(u)|1-x|_\infty^{-u} |1-y|_q^{-u}\widehat{\Theta}_\infty(x)\widehat{\Theta}_{q}(y)\widetilde{F}\legendresymbol{u-1}{\vartheta'} \uppi^{-\frac{u-1}{\vartheta'}+\frac12}\\
\times&\prod_{i=1}^{r}\frac{(1-\epsilon_i q_i^{(u-1)/\vartheta'-1})(1-q_i^{-2(u-1)/\vartheta'})}{1-\epsilon_iq_i^{-(u-1)/\vartheta'}} \frac{\Gamma(\frac{\iota}{2}+\frac{u-1}{2\vartheta'})} {\Gamma(\frac{\iota+1}{2}-\frac{u-1}{2\vartheta'})} \zeta\legendresymbol{2u-2}{\vartheta'}|x|_\infty^{\frac{2u-3}{2}}|y|_q'^{\frac{2u-3}{2}}X^u\rmd x\rmd y.
\end{align*}
Note that
\begin{enumerate}[itemsep=0pt,parsep=0pt,topsep=0pt,leftmargin=0pt,labelsep=2.5pt,itemindent=12pt,label=\textbullet]
  \item The function $\widetilde{F}((u-1)/\vartheta')$ has a simple pole at $u=1$ with residue $\vartheta'$ and
  \[
  \fp_{u=1}\widetilde{F}\legendresymbol{u-1}{\vartheta'}=0
  \]
  since $\widetilde{F}(s)$ is an odd function.
  \item The function
  \[
  \frac{\Gamma(\frac{\iota}{2}+\frac{u-1}{2\vartheta'})}{\Gamma(\frac{\iota+1}{2}-\frac{u-1}{2\vartheta'})} 
  \]
  is regular at $u=1$ if $\iota=1$ with value $\Gamma(1/2)/\Gamma(1)=\sqrt{\uppi}$. If $\iota=0$, it has a simple pole at $u=1$ with 
  \[
  \res_{u=1}\frac{\Gamma(\frac{u-1}{2\vartheta'})}{\Gamma(\frac{1}{2}-\frac{u-1}{2\vartheta'})} =\frac{2\vartheta'}{\Gamma(1/2)}=\frac{2\vartheta'}{\sqrt{\uppi}} 
  \]
  and 
  \begin{align*}
  \fp_{u=1}\frac{\Gamma(\frac{u-1}{2\vartheta'})}{\Gamma(\frac{1}{2}-\frac{u-1}{2\vartheta'})} &=\frac{\fp_{u=1}\Gamma(\frac{u-1}{2\vartheta'})}{\Gamma(1/2)}+\res_{u=1}\Gamma\left(\frac{u-1}{2\vartheta'}\right) \left.\frac{\rmd}{\rmd u}\right|_{u=1}\frac{1}{\Gamma(\frac{1}{2}-\frac{u-1}{2\vartheta'})}\\
  &=-\frac{\upgamma}{\Gamma(1/2)}+\frac{\Gamma'(1/2)}{\Gamma(1/2)}\frac{1}{\Gamma(1/2)}= -\frac{2\upgamma+2\log 2}{\sqrt{\uppi}},
  \end{align*}
  where $\upgamma$ is the Euler-Mascheroni constant. The above two formulas hold since we have the following Laurent expansion for $\Gamma(s)$:
  \[
  \Gamma(s)=\frac{1}{s}-\upgamma+O(s),
  \]
  which can be derived by $\Gamma(s+1)=s\Gamma(s)$ and $\Gamma'(1)=-\upgamma$. Moreover, $\Gamma'(1/2)/\Gamma(1/2)=-\upgamma-2\log 2$ \cite[\SSec 5.4]{DLMF}.
  \item  If $\epsilon_i\neq 1$, the function
  \[
  \frac{(1-\epsilon_i q_i^{(u-1)/\vartheta'-1})(1-q_i^{-2(u-1)/\vartheta'})}{1-\epsilon_iq_i^{-(u-1)/\vartheta'}}
  \]
  has a zero of order $1$ at $u=1$ and
  \[
  \left.\frac{\rmd}{\rmd u}\right|_{u=1}\frac{(1-\epsilon_i q_i^{(u-1)/\vartheta'-1})(1-q_i^{-2(u-1)/\vartheta'})}{1-\epsilon_iq_i^{-(u-1)/\vartheta'}} =\frac{2}{\vartheta'}\frac{1-\epsilon_iq_i^{-1}}{1-\epsilon_i}\log q_i.
  \]
  If $\epsilon_i=1$, the function becomes
  \[
  \frac{(1- q_i^{(u-1)/\vartheta'-1})(1-q_i^{-2(u-1)/\vartheta'})}{1-q_i^{-(u-1)/\vartheta'}} =1+q_i^{-(u-1)/\vartheta'}-q_i^{(u-1)/\vartheta' -1}-q_i^{-1}.
  \]
  Hence
  \[
  \left.\frac{(1-\epsilon_i q_i^{(u-1)/\vartheta'-1})(1-q_i^{-2(u-1)/\vartheta'})}{1-\epsilon_iq_i^{-(u-1)/\vartheta'}} \right|_{u=1}=2(1-q_i^{-1})
  \]
  and
  \[
  \left.\frac{\rmd}{\rmd u}\right|_{u=1}\frac{(1-\epsilon_i q_i^{(u-1)/\vartheta'-1})(1-q_i^{-2(u-1)/\vartheta'})}{1-\epsilon_iq_i^{-(u-1)/\vartheta'}} =-\frac{1}{\vartheta'}(1+q_i^{-1})\log q_i.
  \]
\end{enumerate} 

Let $S_{\spl}=S_{\spl,\iota,\epsilon}$ be a subset of $S$ containing the "split places" with respect to $\iota$ and $\epsilon$. More precisely, $\infty\in S_\spl$ if and only if $\iota=0$, and $q_i\in S_\spl$ if and only if $\epsilon_i=1$.

We consider the following cases:

\underline{\emph{Case 1:}} \ \ $\#(S-S_\spl)\geq 2$.

In this case, $\Xi_{\iota,\epsilon}(u)$ is regular at $u=1$ by the above analysis. Hence $\res_{u=1}\Xi_{\iota,\epsilon}(u)=0$.

\underline{\emph{Case 2:}} \ \ $\#(S-S_\spl)=1$. Let $v\in S$ be the unique element in $S-S_\spl$.

In this case, $\Xi_{\iota,\epsilon}(u)$ has a simple pole at $u=1$ by the above analysis.

\underline{\emph{Case 2.1:}}\ \ $v=\infty$.

In this case, we have
\begin{align*}
&\res_{u=1}\Xi_{\iota,\epsilon}(u)=\int_{X_\iota}\int_{Y_\epsilon} \res_{u=1}\widetilde{F}\legendresymbol{u-1}{\vartheta'}\left[\widetilde{G}(u)|1-x|_\infty^{-u} |1-y|_q^{-u}\widehat{\Theta}_\infty(x)\widehat{\Theta}_{q}(y)\uppi^{-\frac{u-1}{\vartheta'}+\frac12}\right.\\
\times&\left.\prod_{i=1}^{r}\frac{(1- q_i^{(u-1)/\vartheta'-1})(1-q_i^{-2(u-1)/\vartheta'})}{1-q_i^{-(u-1)/\vartheta'}} \frac{\Gamma(\frac{\iota}{2}+\frac{u-1}{2\vartheta'})} {\Gamma(\frac{\iota+1}{2}-\frac{u-1}{2\vartheta'})} \zeta\legendresymbol{2u-2}{\vartheta'}|x|_\infty^{\frac{2u-3}{2}}|y|_q'^{\frac{2u-3}{2}}X^u\right]_{u=1}\rmd x\rmd y\\
%=&\int_{X_\iota}\int_{Y_\epsilon} \frac12\widetilde{G}(1)|1-x|_\infty^{-1} |1-y|_q^{-1}\widehat{\Theta}_\infty(x)\widehat{\Theta}_{q}(y) \uppi^{1/2}2^r\prod_{i=1}^{r}\left(1-\frac{1}{q_i}\right)\sqrt{\uppi}\zeta(0) |x|_\infty^{-\frac{1}{2}}|y|_q'^{-\frac{1}{2}}\rmd x\rmd yX\\
=&-\vartheta'2^{r-1}\uppi\prod_{i=1}^{r}(1-q_i^{-1})\widetilde{G}(1)\sum_{\pm} \int_{X_\iota}\int_{Y_\epsilon} |1-x|_\infty^{-1} |1-y|_q^{-1}\Theta_\infty^\pm(x)\widehat{\Theta}_{q}(y) |x|_\infty^{-\frac{1}{2}}|y|_q'^{-\frac{1}{2}}\rmd x\rmd yX
\end{align*}
since $\zeta(0)=-1/2$.

\underline{\emph{Case 2.2:}}\ \ $v=q_i$ for some $i$.

In this case, we have
\begin{align*}
&\res_{u=1}\Xi_{\iota,\epsilon}(u)=\int_{X_\iota}\int_{Y_\epsilon} \res_{u=1}\widetilde{F}\legendresymbol{u-1}{\vartheta'}\res_{u=1}\frac{\Gamma(\frac{u-1}{2\vartheta'})} {\Gamma(\frac{1}{2}-\frac{u-1}{2\vartheta'})} \left[\widetilde{G}(u)|1-x|_\infty^{-u} |1-y|_q^{-u} \right.\\
\times&\left.\widehat{\Theta}_\infty(x)\widehat{\Theta}_{q}(y) \uppi^{-\frac{u-1}{\vartheta'}+\frac12}\prod_{j\neq i}\frac{(1- q_j^{(u-1)/\vartheta'-1})(1-q_j^{-2(u-1)/\vartheta'})}{1-q_j^{-(u-1)/\vartheta'}} \zeta\legendresymbol{2u-2}{\vartheta'}|x|_\infty^{\frac{2u-3}{2}}|y|_q'^{\frac{2u-3}{2}}X^u\right]_{u=1}\\
\times&\left.\frac{\rmd}{\rmd u}\right|_{u=1}\frac{(1-\epsilon_i q_i^{(u-1)/\vartheta'-1})(1-q_i^{-2(u-1)/\vartheta'})}{1-\epsilon_iq_i^{-(u-1)/\vartheta'}}\rmd x\rmd y\\
%=&\int_{X_\iota}\int_{Y_\epsilon}\frac{1}{2\sqrt{\uppi}}\widetilde{G}(1)|1-x|_\infty^{-1} |1-y|_q^{-1}\widehat{\Theta}_\infty(x)\widehat{\Theta}_{q}(y) \uppi^{1/2}2^{r-1}\prod_{j\neq i}\left(1-\frac{1}{q_j}\right)\zeta(0) |x|_\infty^{-\frac{1}{2}}|y|_q'^{-\frac{1}{2}} 4\frac{1-\epsilon_iq_i^{-1}}{1-\epsilon_i}\log q_i\rmd x\rmd yX\\
=&-\vartheta'2^{r}\prod_{i=1}^{r}(1-q_i^{-1}) \frac{1-\epsilon_iq_i^{-1}}{1-\epsilon_i}\frac{\log q_i}{1-q_i^{-1}} \int_{X_\iota}\int_{Y_\epsilon}\widetilde{G}(1)|1-x|_\infty^{-1} |1-y|_q^{-1}\widehat{\Theta}_\infty(x)\widehat{\Theta}_{q}(y) |x|_\infty^{-\frac{1}{2}}|y|_q'^{-\frac{1}{2}} \rmd x\rmd yX.
\end{align*}

\underline{\emph{Case 3:}} \ \ $S=S_\spl$, that is, $\iota=0$ and $\epsilon=\mathbf{1}$.

In this case, $\Xi(u)$ has a double pole at $u=1$ by the above analysis. Thus
\begin{align*}
&\res_{u=1}\Xi_{\iota,\epsilon}(u)=\int_{X_\iota}\int_{Y_\epsilon} \res_{u=1}\widetilde{F}\legendresymbol{u-1}{\vartheta'}\res_{u=1}\frac{\Gamma(\frac{u-1}{2\vartheta'})} {\Gamma(\frac{1}{2}-\frac{u-1}{2\vartheta'})} \left.\frac{\rmd}{\rmd u}\right|_{u=1}\left[\widetilde{G}(u)|1-x|_\infty^{-u} |1-y|_q^{-u} \right.\\
\times&\left.\widehat{\Theta}_\infty(x)\widehat{\Theta}_{q}(y) \uppi^{-\frac{u-1}{\vartheta'}+\frac12} \prod_{i=1}^{r}\frac{(1- q_i^{(u-1)/\vartheta'-1})(1-q_i^{-2(u-1)/\vartheta'})}{1-q_i^{-(u-1)/\vartheta'}} \zeta\legendresymbol{2u-2}{\vartheta'}|x|_\infty^{\frac{2u-3}{2}}|y|_q'^{\frac{2u-3}{2}}X^u\right]\rmd x\rmd y\\
+&\int_{X_\iota}\int_{Y_\epsilon} \left[\fp_{u=1}\widetilde{F}\legendresymbol{u-1}{\vartheta'}\res_{u=1} \frac{\Gamma(\frac{u-1}{2\vartheta'})} {\Gamma(\frac{1}{2}-\frac{u-1}{2\vartheta'})}+ \res_{u=1}\widetilde{F}\legendresymbol{u-1}{\vartheta'}\fp_{u=1} \frac{\Gamma(\frac{u-1}{2\vartheta'})} {\Gamma(\frac{1}{2}-\frac{u-1}{2\vartheta'})} \right]\\
\times& \left[\widetilde{G}(u)|1-x|_\infty^{-u} |1-y|_q^{-u}\widehat{\Theta}_\infty(x)\widehat{\Theta}_{q}(y) \uppi^{-\frac{u-1}{\vartheta'}+\frac12}\right.\\
\times&\left. \prod_{i=1}^{r}\frac{(1- q_i^{(u-1)/\vartheta'-1})(1-q_i^{-2(u-1)/\vartheta'})}{1-q_i^{-(u-1)/\vartheta'}} \zeta\legendresymbol{2u-2}{\vartheta'} |x|_\infty^{\frac{2u-3}{2}}|y|_q'^{\frac{2u-3}{2}}X^u\right]_{u=1}\rmd x\rmd y\\
=&-\vartheta'^22^{r}\prod_{i=1}^{r}(1-q_i^{-1})\int_{X_\iota}\int_{Y_\epsilon}\widetilde{G}'(1)|1-x|_\infty^{-1} |1-y|_q^{-1}\widehat{\Theta}_\infty(x)\widehat{\Theta}_{q}(y) |x|_\infty^{-\frac{1}{2}}|y|_q'^{-\frac{1}{2}}X\rmd x\rmd y\\
+&\vartheta'^22^{r}\prod_{i=1}^{r}(1-q_i^{-1})\int_{X_\iota}\int_{Y_\epsilon}\widetilde{G}(1)|1-x|_\infty^{-1} \log|1-x|_\infty |1-y|_q^{-1}\widehat{\Theta}_\infty(x)\widehat{\Theta}_{q}(y) |x|_\infty^{-\frac{1}{2}}|y|_q'^{-\frac{1}{2}}X\rmd x\rmd y\\
+&\vartheta'^22^{r}\prod_{i=1}^{r}(1-q_i^{-1})\int_{X_\iota}\int_{Y_\epsilon}\widetilde{G}(1)|1-x|_\infty^{-1} |1-y|_q^{-1}\log |1-y|_q\widehat{\Theta}_\infty(x)\widehat{\Theta}_{q}(y) |x|_\infty^{-\frac{1}{2}}|y|_q'^{-\frac{1}{2}}X\rmd x\rmd y\\
+&\vartheta'2^{r}\prod_{i=1}^{r}(1-q_i^{-1})\int_{X_\iota}\int_{Y_\epsilon}\widetilde{G}(1)|1-x|_\infty^{-1} |1-y|_q^{-1}\log\uppi \widehat{\Theta}_\infty(x)\widehat{\Theta}_{q}(y) |x|_\infty^{-\frac{1}{2}}|y|_q'^{-\frac{1}{2}}X\rmd x\rmd y\\
+&\vartheta'2^{r-1}\prod_{i=1}^{r}(1-q_i^{-1})\int_{X_\iota}\int_{Y_\epsilon}\widetilde{G}(1)|1-x|_\infty^{-1} |1-y|_q^{-1}\widehat{\Theta}_\infty(x)\widehat{\Theta}_{q}(y) \sum_{i=1}^{r}(1+q_i^{-1})\frac{\log q_i}{1-q_i^{-1}}|x|_\infty^{-\frac{1}{2}}|y|_q'^{-\frac{1}{2}}X\rmd x\rmd y\\
-&\vartheta'2^{r+1}\prod_{i=1}^{r}(1-q_i^{-1})\int_{X_\iota}\int_{Y_\epsilon}\widetilde{G}(1)|1-x|_\infty^{-1} |1-y|_q^{-1}\widehat{\Theta}_\infty(x)\widehat{\Theta}_{q}(y) \log(2\uppi)|x|_\infty^{-\frac{1}{2}}|y|_q'^{-\frac{1}{2}}X\rmd x\rmd y\\
-&\vartheta'^22^{r}\prod_{i=1}^{r}(1-q_i^{-1})\int_{X_\iota}\int_{Y_\epsilon}\widetilde{G}(1)|1-x|_\infty^{-1} |1-y|_q^{-1}\widehat{\Theta}_\infty(x)\widehat{\Theta}_{q}(y) |x|_\infty^{-\frac{1}{2}}\log|x|_\infty|y|_q'^{-\frac{1}{2}}X\rmd x\rmd y\\
-&\vartheta'^22^{r}\prod_{i=1}^{r}(1-q_i^{-1})\int_{X_\iota}\int_{Y_\epsilon}\widetilde{G}(1)|1-x|_\infty^{-1} |1-y|_q^{-1}\widehat{\Theta}_\infty(x)\widehat{\Theta}_{q}(y) |x|_\infty^{-\frac{1}{2}}|y|_q'^{-\frac{1}{2}}\log|y|'_q X\rmd x\rmd y\\
-&\vartheta'^22^{r}\prod_{i=1}^{r}(1-q_i^{-1})\int_{X_\iota}\int_{Y_\epsilon}\widetilde{G}(1)|1-x|_\infty^{-1} |1-y|_q^{-1}\widehat{\Theta}_\infty(x)\widehat{\Theta}_{q}(y) |x|_\infty^{-\frac{1}{2}}|y|_q'^{-\frac{1}{2}}X\log X\rmd x\rmd y\\
+&\vartheta'2^{r}(\upgamma+\log 2)\prod_{i=1}^{r}(1-q_i^{-1})\int_{X_\iota}\int_{Y_\epsilon}\widetilde{G}(1)|1-x|_\infty^{-1} |1-y|_q^{-1}\widehat{\Theta}_\infty(x)\widehat{\Theta}_{q}(y) |x|_\infty^{-\frac{1}{2}}|y|_q'^{-\frac{1}{2}}X\rmd x\rmd y,
\end{align*}
since $\zeta'(0)=-\log(2\uppi)/2$ \cite[\SSec 25.6]{DLMF}.
Combining all the things we have proved we get the asymptotic formula in the lemma.
\end{insertproof}
Finally, by \autoref{lem:asympalpha0firstterm} and \autoref{lem:asympalpha0secondterm}, and noting that the first terms of them cancel, we obtain \autoref{thm:contributealpha0}.
\end{proof}

\section{Analysis of the error term}
This section, except for \autoref{subsec:poissontheta}, is very technical. One may assume that $\vartheta=1/2$ for the first reading.

In this section, we will analyze the $\alpha\neq 0$ term, namely
\[
S_G^{\alpha\neq 0}(X)=\sum_{k,f\in \ZZ_{(S)}^{>0}}\frac{1}{k^3f^5}\sum_{\alpha\in \ZZ^S-\{0\}}\sum_{\xi\in \ZZ^S}\widehat{J}_{k,f}(X,\xi,0)\widehat{\Kl}_{k,f}^S(\xi,\alpha).
\]
By \autoref{cor:kloostermanlocalcompute} (2), we obtain
\begin{equation}\label{eq:alphanot0}
S_G^{\alpha\neq 0}(X)=\sum_{k,f\in \ZZ_{(S)}^{>0}}\frac{1}{k^3f^5}\sum_{\alpha\in \ZZ^S-\{0\}}c_{k,f}(\alpha)\sum_{\xi\in \gcd(\alpha,kf^2)\ZZ^S}\widehat{J}_{k,f}(X,\xi,\alpha)\rme\legendresymbol{-\xi^2}{\alpha kf^2}\rme_q\legendresymbol{-\xi^2}{\alpha kf^2}.
\end{equation}

\subsection{Poisson summation of the theta-type function}\label{subsec:poissontheta}
 In this section we will use Poisson summation on the term
\begin{equation}\label{eq:alphanot0poisson}
\sum_{\xi\in \gcd(\alpha,kf^2)\ZZ^S}\widehat{J}_{k,f}(X,\xi,\alpha)\rme\legendresymbol{-\xi^2}{\alpha kf^2}\rme_q\legendresymbol{-\xi^2}{\alpha kf^2},
\end{equation}
which, we will see, is a theta-type function.

By the expression \eqref{eq:transformedj} of $\widehat{J}_{k,f}(X,\xi,\alpha)$ after the transformation \eqref{eq:harishchandratransform}, \eqref{eq:alphanot0poisson} equals
\begin{equation}\label{eq:alphanot0poissontheta}
\begin{split}
&2\sum_{\xi\in \gcd(\alpha,kf^2)\ZZ^S}\int_{(a,b)\in\QQ_S}\!\left\{\!\int_{(x,y)\in\QQ_S} G\!\left(\frac{1}{X}\left|\frac{a^2}{4}-x\right|_\infty\left|\frac{b^2}{4}-y\right|_q \right) \theta_\infty\!\left(a,\frac{a^2}{4}-x\right)\theta_{q}\!\left(b,\frac{b^2}{4}-y\right)\right.\\
\times&\left.\left[F\legendresymbol{kf^2}{|4x|_\infty^{\vartheta}|4y|_q'^{\vartheta}}
+\frac{kf^2}{\sqrt{|4x|_\infty|4y|_q'}}V\legendresymbol{kf^2}{|4x|_\infty^{\vartheta'}|4y|_q'^{\vartheta'}}\right] \rme\legendresymbol{x\alpha}{kf^2}\rme_q\legendresymbol{y\alpha}{kf^2}\rmd x\rmd y\right\}\\
\times&\rme\legendresymbol{-a\xi-a^2\alpha/4-\xi^2/\alpha}{kf^2} \rme_q\legendresymbol{-b\xi-b^2\alpha/4-\xi^2/\alpha}{kf^2}\rmd a\rmd b
\end{split}
\end{equation}
The key observation is that
\[
\frac{-a\xi-a^2\alpha/4-\xi^2/\alpha}{kf^2}=-\frac{1}{\alpha kf^2}\left(\xi+\frac{a\alpha}{2}\right)^2 \quad\text{and}\quad \frac{-b\xi-b^2\alpha/4-\xi^2/\alpha}{kf^2}=-\frac{1}{\alpha kf^2}\left(\xi+\frac{b\alpha}{2}\right)^2.
\]

Now we want to use the Poisson summation for $\xi$. To do this, we need to compute certain Fourier transform on $\QQ_S$. We begin with local computation. First we consider the archimedean part.
\begin{lemma}
Let $\sigma,u\in \RR^\times$. Then for any $f\in \cS(\RR)$, 
\[
F(x)=\int_{\RR}f(a)\rme\legendresymbol{\sigma(x+u a)^2}{2}\rmd a
\]
is a Schwartz function with Fourier transform
\[
\widehat{F}(\xi)=\frac{\gamma}{\sqrt{|\sigma|}}\rme\legendresymbol{-\xi^2}{2\sigma}\widehat{f}(-u \xi),
\]
where $\gamma$ is an eighth root of unity depending on $\sigma$.
\end{lemma}
\begin{proof}
By making the\ change of variable $a\mapsto -a/u$, we obtain 
\[
F(x)=\frac{1}{|u|}\int_{\RR}f\left(\frac{-a}{u}\right)\rme^{\uppi\rmi\sigma(x-a)^2}\rmd a
\]
so that $F(x)$ is the convolution of the Schwartz function $f(-x/u)$ and the tempered distribution $\rme^{\uppi\rmi\sigma x^2}$. Hence $F(x)$ is a Schwartz function and its Fourier transform is the product of their Fourier transforms. The Fourier transform of $f(-x/u)$ is
\[
|u|\widehat{f}(-u x),
\]
and the Fourier transform of the distribution $\rme^{\uppi\rmi\sigma x^2}$ is
\[
\frac{1}{\sqrt{\rmi \sigma}}\rme^{\frac{\uppi x^2}{\rmi \sigma}}=\frac{\gamma}{\sqrt{|\sigma|}}\rme\legendresymbol{-x^2}{2\sigma}
\]
since the Fourier transform of $\rme^{-\delta\uppi x^2}$ is $\rme^{-\uppi x^2/\delta}/\sqrt{\delta}$, using analytic continuation.
\end{proof}

\begin{lemma}
Let $\tau,v\in \QQ_p^\times$. Then for any $f\in C_c^\infty(\QQ_p)$, 
\[
F(y)=\int_{\QQ_p}f(b)\rme_p\legendresymbol{\tau (y+v b)^2}{2}\rmd b
\]
is a smooth, compactly supported function with Fourier transform
\[
\widehat{F}(\eta)=\frac{\gamma}{\sqrt{|\tau|_p}}\rme_p\legendresymbol{-\eta^2}{2\tau}\widehat{f}(-v\eta)
\]
where $\gamma$ is an eighth root of unity depending on $\tau$.
\end{lemma}
\begin{proof}
The proof is essentially the same as the previous lemma. The only difference is that we want to compute the $p$-adic Fourier transform of the tempered distribution
\[
\rme_p\legendresymbol{\tau y^2}{2}\sphat=\frac{\gamma}{\sqrt{|\tau|_p}}\rme_p\legendresymbol{-\eta^2}{2\tau},
\]
which was computed in \cite[Théorème 2]{weil1964}.
\end{proof}

By the previous two lemmas we obtain the semilocal case.
\begin{corollary}
Let $\sigma, u\in \RR$ and $\tau=(\tau_1,\dots,\tau_r), v=(v_1,\dots,v_r)\in \QQ_{S_\fin}^\times$. Then for any $f\in \cS(\QQ_S)$,
\[
F(x,y)=\int_{\QQ_S}f(a,b)\rme\legendresymbol{\sigma(x+ua)^2}{2}\rme_q\legendresymbol{\tau(y+vb)^2}{2}\rmd a\rmd b
\]
is a Schwartz function for $(x,y)\in \QQ_S$ with Fourier transform
\[
\widehat{F}(\xi,\eta)=\frac{\gamma}{\sqrt{|\sigma|_\infty|\tau|_q}}\rme\legendresymbol{-\xi^2}{2\sigma}\rme_q \legendresymbol{-\eta^2}{2\tau}\widehat{f}(-u\xi,-v\eta),
\]
where $\gamma$ is an eighth root of unity depending on $\sigma$ and $\tau$.
\end{corollary}

For $m\in \ZZ_{(S)}^{>0}$ we have the Poisson summation formula (cf. Corollary B.4 of \cite{cheng2025})
\[
\sum_{\xi\in m\ZZ^S}F(\xi,\xi)=\frac1m\sum_{\kappa\in \ZZ^S}\widehat{F}\left(\frac{\kappa}{m},\frac{\kappa}{m}\right).
\]
Hence we obtain
\begin{equation}\label{eq:poissontheta}
\begin{split}
    & \sum_{\xi\in m\ZZ^S}\int_{\QQ_S}f(a,b)\rme\legendresymbol{\sigma(\xi+ua)^2}{2}\rme_q\legendresymbol{\tau(\xi+vb)^2}{2}\rmd a\rmd b \\
    = & \sum_{\kappa\in \ZZ^S}\frac{\gamma}{m\sqrt{|\sigma|_\infty|\tau|_q}}\rme\legendresymbol{-\kappa^2}{2m^2\sigma}\rme_q \legendresymbol{-\kappa^2}{2m^2\tau}\int_{\QQ_S}f(a,b)\rme\legendresymbol{au\kappa}{m} \rme_q\legendresymbol{bv\kappa}{m}\rmd a\rmd b
\end{split}
\end{equation}
for $f\in \cS(\QQ_S)$.

Now we set $f(a,b)=\Phi_\alpha(a,b)$ as in \autoref{lem:poisson2smooth}, where we proved the function is smooth and compactly supported, and
\[
(\sigma,\tau)=-\frac{2}{\alpha kf^2}\in \QQ\subseteq \QQ_S,\quad\text{and}\quad (u,v)=\frac{\alpha}{2}\in \QQ\subseteq \QQ_S.
\]
By \eqref{eq:alphanot0poissontheta}, and noting that $|4x|_\infty|4y|_q'=|4|_\infty|x|_\infty|4|_q|y|_q'=|x|_\infty|y|_q'$  we obtain
\begin{theorem}\label{thm:poissontheta}
We have
\begin{align*}
\eqref{eq:alphanot0poisson}
=&\frac{2\gamma_{k,f}(\alpha)}{\gcd(\alpha,kf^2)}\left|\legendresymbol{2}{\alpha kf^2}^{(q)}\right|^{-\frac12}\sum_{\kappa\in \ZZ^S}\rme\legendresymbol{\kappa^2\alpha kf^2}{4\gcd(\alpha,kf^2)^2}\rme_q \legendresymbol{\kappa^2\alpha kf^2}{4\gcd(\alpha,kf^2)^2}\\
\times&\int_{(a,b)\in\QQ_S}\left\{\int_{(x,y)\in\QQ_S} G\left(\frac{1}{X}\left|\frac{a^2}{4}-x\right|_\infty\left|\frac{b^2}{4}-y\right|_q \right) \theta_\infty\!\left(a,\frac{a^2}{4}-x\right)\theta_{q}\left(b,\frac{b^2}{4}-y\right)\right.\\
\times&\left.\left[F\legendresymbol{kf^2}{|x|_\infty^{\vartheta}|y|_q'^{\vartheta}}
+\frac{kf^2}{\sqrt{|x|_\infty|y|_q'}}V\legendresymbol{kf^2}{|x|_\infty^{\vartheta'}|y|_q'^{\vartheta'}}\right] \rme\legendresymbol{x\alpha}{kf^2}\rme_q\legendresymbol{y\alpha}{kf^2}\rmd x\rmd y\right\}\\
\times&\rme\legendresymbol{a\alpha\kappa}{2\gcd(\alpha,kf^2)} \rme_q\legendresymbol{b\alpha\kappa}{2\gcd(\alpha,kf^2)}\rmd a\rmd b,
\end{align*}
where $\gamma_{k,f}(\alpha)$ is an eighth root of unity depending on $\alpha$, $k$ and $f$.
\end{theorem}

\subsection{Fourier transform with singularities}
We will often use the notation $s=\sigma+\rmi t$ in this subsection.
Our main goal of this subsection is to prove the following theorem:
\begin{theorem}\label{thm:mainfourierestimate}
Let $\vartheta\in \lopen 0,1\ropen$.
Suppose that $\Phi\in \cS$. Let $\widetilde{\Phi}(s)$ be its Mellin transform. Suppose that $\widetilde{\Phi}(s)$ has a meromorphic continuation on $\Re s>0$, and is of rapid decay in $t$ when $\sigma>0$ is fixed. Fix a sign $\pm$, $\nu\in \ZZ^r$ and $\iota\in \{0,1\}$, $\epsilon_i\in \{0,\pm 1\}$. Let $a_0,a_1,\dots,a_r$ be real numbers such that $a_i>-2$ for all $i$. Let 
\[
\varphi(x)=|x|^{a_0/2} \overline{\varphi}(x)\quad\text{and}\quad \psi_i(y_i)=|y_i|_{q_i}'^{a_i/2} \overline{\psi_i}(y_i),
\]
where $\overline{\varphi}(x)$ is smooth function on the open set $X_{\iota}=\{x\in \RR\ |\ \omega_\infty(x)=\iota\}$ up to the boundary, supported on a bounded set of $X_{\iota}$, and $\overline{\psi_i}(y_i)$ are smooth functions on $Y_{\epsilon_i}=\{y_i\in \QQ_{q_i}\ |\ \omega_i(y_i)=\epsilon_i\}$ and up to $0$, with bounded support. Let $G(x,y)$ be a smooth and compactly supported function on $\QQ_S$ and set
\[
H(x,y)=G(x,y)\overline{\varphi}(x) \prod_{i=1}^{r}\overline{\psi_i}(y_i).
\]
Let $Y_\epsilon=Y_{\epsilon_1}\times\dots\times Y_{\epsilon_r}$ and 
\[
I_{\iota,\epsilon}(\xi,\eta)=\int_{X_\iota}\int_{Y_\epsilon}G(x,y)\varphi(x) \prod_{i=1}^{r}\psi_i(y_i)\Phi\legendresymbol{C}{|x|^{\vartheta}|y|_q'^{\vartheta}}\rme(-x\xi)\rme_{q}(-y\eta)\rmd x\rmd y.
\]
Let
\[
\llbracket\xi,\eta\rrbracket=(1+|\xi|_\infty)\prod_{i=1}^{r}(1+|\eta_i|_{q_i}).
\]
Then for any $\rho\in \RR$ such that $\rho>-2-a_i$ for all $i$,
%\begin{align*}
%I_{\iota,\epsilon}(\xi,\eta)&\ll %\sum_{j=0}^{M-1}\left(|C|^2\llbracket\xi,\eta\rrbracket\right)^{-\vartheta\rho} (1+|\xi|_\infty)^{-j-\lceil\frac{\rho+a_0}{2}\rceil+\frac{\rho}{2}} \prod_{i=1}^{r}(1+|\eta_i|_{q_i})^{-1-\frac{a_i}{2}}\|H\|_{j,\infty}\\
%&+\left(|C|^2\llbracket\xi,\eta\rrbracket\right)^{-\frac{\rho}{2}} (1+|\xi|_\infty)^{-M-\lceil\frac{\rho+a_0}{2}\rceil+\frac{\rho}{2}}\prod_{i=1}^{r}(1+|\eta_i|_{q_i}) ^{-1-\frac{a_i}{2}}\|H\|_{M+\lceil\frac{\rho+a_0}{2}\rceil,\infty},
%\end{align*}
\begin{align*}
I_{\iota,\epsilon}(\xi,\eta)&\ll \left(|C|^{1/\vartheta}\llbracket\xi,\eta\rrbracket\right)^{-\vartheta\rho} (1+|\xi|_\infty)^{-\lceil\vartheta\rho+\frac{a_0}{2}\rceil+\vartheta\rho}\prod_{i=1}^{r}(1+|\eta_i|_{q_i}) ^{-1-\frac{a_i}{2}}\|H\|_{\lceil\vartheta\rho+\frac{a_0}{2}\rceil,1},
\end{align*}
where 
\[
\|H\|_{M,1}=\sum_{j=0}^{M}\sup_{y\in \QQ_S}\int_{x\in \RR}\left|\frac{\diff^j}{\diff x^j}H(x,y)\right|\rmd x.
\]
The implied constant only depends on $\Phi,\vartheta$, 
\[
\diam(\{0\}\cup \pi_{v}(\supp(H)))
\]
for $v\in S$ (where $\pi_v\colon \QQ_S\to \QQ_v$ is the canonical projection), and
\[
c_i=\min\{u\in \ZZ\,|\, H((x,y)+z_i)=H(x,y)\ \text{for all}\ (x,y)\in \QQ_S, z_i\in q_i^u\ZZ_{q_i}\}
\]
for $i=1,2,\dots,r$.
\end{theorem}

\begin{lemma}\label{lem:archimedeansingular}
Let $\varphi$ be a smooth function on $X_{\iota}$ up to the boundary, supported on a bounded set of $X_{\iota}$. Let $s\in \CC$ with $\Re s>-1$. Consider the integral
\[
I_\infty(s,\xi)=\int_{X_\iota}|x|^s\varphi(x)\rme(-x\xi)\rmd x
\]
For any $N\in \ZZ_{\geq 0}$ with $\Re s-N>-1$, we have
\begin{equation}\label{eq:archimedeanfourier}
I_\infty(s,\xi)\ll (1+|\xi|)^{-N}\|\varphi\|_{N,1},
\end{equation}
where the implied constant only depends on $\sigma,N$ and $d=\diam(\{0\}\cup\supp(\varphi))$.
\end{lemma}

\begin{proof}
For simplicity we assume that $X_\iota=\lopen 0,+\infty\ropen$. If $|\xi|\ll 1$, then we bound $I_\infty(s,\xi)$ trivially and obtain
\[
I_\infty(s,\xi)\ll\int_{\supp(\varphi)}|x|^\sigma|\varphi(x)|\rmd x\ll \|\varphi\|_1,
\]
which is bounded by the right hand side of \eqref{eq:archimedeanfourier}.

From now on we assume that $|\xi|\gg 1$. In this case, the term $1+|\xi|$ on the right hand side of \eqref{eq:archimedeanfourier} can be replaced by $|\xi|$.

Note that
\[
(|x|^s)^{(j)}=s(s-1)\cdots (s-j+1)|x|^{s-j}\ll x^{\sigma-j}.
\]

For any $m\in \{0,1,\dots,N\}$,
\[
\frac{\rmd^m}{\rmd x^m}(|x|^s\varphi(x))=\sum_{j=0}^{m}\binom{m}{j}(|x|^s)^{(j)}\varphi^{(m-j)}(x).
\]
%For any $j\in \{0,1,\dots,m\}$,
%\[
%(|x|^s)^{(j)}\varphi^{(m-j)}(x)\ll x^{\sigma-j}\|\varphi\|_{N,\infty}
%\]
Since $\sigma-N>-1$ by assumption, we have
\[
\left.\frac{\rmd^m}{\rmd x^m}\right|_{x=0}(|x|^s\varphi(x))=0
\]
if $m\leq N-1$. Hence by integration by parts $N$ times, we obtain
\begin{align*}
I_\infty(s,\xi)&=\frac{1}{(\dpii \xi)^{N}}\int_0^{+\infty}\frac{\rmd^{N}}{\rmd x^{N}}(|x|^s\varphi(x))\rme(-x\xi)\rmd x\\
&\ll |\xi|^{-N}\sum_{j=0}^{N}\binom{N}{j}\int_{0}^{d}x^{\sigma-m}\left|\frac{\rmd^j}{\rmd x^j}\varphi(x)\right|\rmd x\ll |\xi|^{-N}\|\varphi\|_{N,1}.\qedhere
\end{align*}
\end{proof}

\begin{lemma}\label{lem:nonarchimedeansingular}
Let $\varphi$ be a smooth function on $Y_\epsilon$ and up to $0$, with bounded support. Let $s\in \CC$ with $\Re s>1$. Consider the integral
\[
I_p(s,\eta)=\int_{Y_\epsilon}|y|'^s\varphi(y)\rme_p(-y\eta)\rmd y.
\]
Then we have
\[
I_p(s,\eta)\ll (1+|\eta|)^{-\sigma-1}\|\varphi\|_{\infty},
\]
where the implied constant only depends on $\sigma$, $d=\diam(\{0\}\cup\supp(\varphi))$, and 
\[
c=\min\{u\in \ZZ\,|\, \varphi(y+z)=\varphi(y)\ \text{for all}\ y\in \QQ_p, z\in p^u\ZZ_p\}.
\]
\end{lemma}
\begin{proof}
By assumption $\varphi(y)=\varphi(0)$ for $y\in p^c\ZZ_p$. Hence we have
\[
I_p(s,\eta)=\varphi(0)\int_{Y_\epsilon\cap p^c\ZZ_p}|y|'^s\rme(-y\eta)\rmd y+\int_{Y_\epsilon-p^c\ZZ_p}|y|'^s\varphi(y)\rme(-y\eta)\rmd y.
\]
We first consider the second term. For any $a\in p^{c+2}\ZZ_p$ and $y\notin p^c\ZZ_p$ we have
\[
\frac{y+a}{y}\in 1+p^2\ZZ_p.
\]
Hence $|y+a|'=|y|'$. Also, $\varphi(y+a)=\varphi(y)$. Hence $|y|'^s\varphi(y)$ is invariant under the map $y\mapsto y+a$ for $a\in p^{c+2}\ZZ_p$. Hence the second integral is $0$ if $|\eta|>p^{c+2}$. For $|\eta|\leq p^{c+2}$ we have the trivial bound
\[
\int_{Y_\epsilon-p^c\ZZ_p}|y|'^s\varphi(y)\rme(-y\eta)\rmd y\ll \|\varphi\|_\infty\int_{Y_\epsilon-p^c\ZZ_p}|y|'^\sigma\rmd y\ll \|\varphi\|_\infty.
\]
Hence for any $\sigma>-1$ we have
\[
\int_{Y_\epsilon-p^c\ZZ_p}|y|'^s\varphi(y)\rme(-y\eta)\rmd y\ll(1+|\eta|)^{-\sigma-1}\|\varphi\|_\infty.
\]

Next we analyze the first term. We have
\[
\int_{Y_\epsilon\cap p^c\ZZ_p}|y|'^s\rme(-y\eta)\rmd y=\sum_{u=c}^{+\infty} \int_{Y_\epsilon\cap p^u\ZZ_p^\times}|y|'^s\rme(-y\eta)\rmd y.
\]
A same argument shows that $|y|'^s$ is invariant under the map $y\mapsto y+a$ for $y\in p^u\ZZ_p^\times$ and $a\in p^{u+2}\ZZ_p$. Hence 
\[
\int_{Y_\epsilon\cap p^u\ZZ_p^\times}|y|'^s\rme(-y\eta)\rmd y=0
\]
if $|\eta|>p^{u+2}$. If $|\eta|\leq p^{u+2}$, we bound the integral trivially. Using the fact that $|\cdot|'\asymp |\cdot|$, we obtain
\[
\int_{Y_\epsilon\cap p^u\ZZ_p^\times}|y|'^s\rme(-y\eta)\rmd y\ll p^{-u}p^{-u\sigma}.
\]

Suppose that $\eta\in \QQ_p$. If $|\eta|>p^{c+2}$, then we have
\begin{align*}
\int_{Y_\epsilon\cap p^c\ZZ_p}|y|'^s\rme(-y\eta)\rmd y&=\sum_{u=-v_p(\eta)+2}^{+\infty} \int_{Y_\epsilon\cap p^u\ZZ_p^\times}|y|'^s\rme(-y\eta)\rmd y\\
\ll& \sum_{u=-v_p(\eta)+2}^{+\infty} p^{-u(\sigma+1)}\ll p^{v_p(\eta)(\sigma+1)}\ll |\eta|^{-\sigma-1}.
\end{align*}
Hence the first term is bounded by $\|\varphi\|_\infty(1+|\eta|)^{-\sigma-1}$. If $|\eta|\leq p^{c+2}$, we only need to bound the integral trivially as before and get the same bound.
\end{proof}

\begin{proof}[Proof of \autoref{thm:mainfourierestimate}]
By Mellin inversion formula and the definition of $\varphi$ and $\psi_i$, we obtain 
\[
I_{\iota,\epsilon}(\xi,\eta)=\frac{1}{\dpii}\int_{(\sigma)}\frac{\widetilde{\Phi}(u)}{C^u} \int_{X_\iota}\int_{Y_\epsilon}G(x,y)\overline{\varphi}(x) \prod_{i=1}^{r}\overline{\psi_i}(y_i)|x|_\infty^{\vartheta u+\frac{a_0}{2}}\prod_{i=1}^{r}|y_i|_{q_i}'^{\vartheta u+\frac{a_i}{2}}\rme(-x\xi)\rme_{q}(-y\eta)\rmd x\rmd y\rmd u.
\]
Hence
\[
I_{\iota,\epsilon}(\xi,\eta)\ll\int_{(\sigma)}\frac{\widetilde{\Phi}(u)}{C^\sigma} \left|\int_{X_\iota}\int_{Y_\epsilon}G(x,y)\overline{\varphi}(x) \prod_{i=1}^{r}\overline{\psi_i}(y_i)|x|_\infty^{\vartheta u+\frac{a_0}{2}}\prod_{i=1}^{r}|y_i|_{q_i}'^{\vartheta u+\frac{a_i}{2}}\rme(-x\xi)\rme_{q}(-y\eta)\rmd x\rmd y\right|\rmd |u|
\]
Now we choose $\rho=\sigma$ and $N=\lceil\vartheta\rho+a_0/2\rceil$. By \autoref{lem:nonarchimedeansingular} we can bound each nonarchimedean integral and use \autoref{lem:archimedeansingular} with $N$ above to bound the archimedean integral. Since $\widetilde{\Phi}(u)$ has rapid decay vertically, we obtain our result. 
\end{proof}

\subsection{Estimate of the integral} 
Let
\begin{align*}
\mathbf{I}_{k,f}(\alpha,\kappa)
=&2\int_{(a,b)\in\QQ_S}\left\{\int_{(x,y)\in\QQ_S} G\left(\frac{1}{X}\left|\frac{a^2}{4}-x\right|_\infty\left|\frac{b^2}{4}-y\right|_q \right) \theta_\infty\!\left(a,\frac{a^2}{4}-x\right)\theta_{q}\left(b,\frac{b^2}{4}-y\right)\right.\\
\times&\left.\left[F\legendresymbol{kf^2}{|x|_\infty^{\vartheta}|y|_q'^{\vartheta}}
+\frac{kf^2}{\sqrt{|x|_\infty|y|_q'}}V\legendresymbol{kf^2}{|x|_\infty^{\vartheta'}|y|_q'^{\vartheta'}}\right] \rme\legendresymbol{x\alpha}{kf^2}\rme_q\legendresymbol{y\alpha}{kf^2}\rmd x\rmd y\right\}\\
\times&\rme\legendresymbol{a\alpha\kappa}{2\gcd(\alpha,kf^2)} \rme_q\legendresymbol{b\alpha\kappa}{2\gcd(\alpha,kf^2)}\rmd a\rmd b,
\end{align*}
so that by \eqref{eq:alphanot0} and \autoref{thm:poissontheta},
\[
S_G^{\alpha\neq 0}(X)=\sum_{k,f\in \ZZ_{(S)}^{>0}}\sum_{\alpha\in \ZZ^S-\{0\}}\frac{c_{k,f}(\alpha)\gamma_{k,f}(\alpha)}{k^3f^5\gcd(\alpha,kf^2)}|\alpha^{(q)}|^{\frac12}\!\sum_{\kappa\in \ZZ^S}\rme\legendresymbol{\kappa^2\alpha kf^2}{4\gcd(\alpha,kf^2)^2}\rme_q \legendresymbol{\kappa^2\alpha kf^2}{4\gcd(\alpha,kf^2)^2}\mathbf{I}_{k,f}(\alpha,\kappa).
\]
\begin{lemma}\label{lem:integralvanish}
There exist $L_1,\dots,L_r\in \ZZ_{\geq 0}$ independent of $G$ such that $\mathbf{I}_{k,f}(\alpha,\kappa)=0$ if 
\[
|\alpha\kappa|_{q_i}> q_i^{L_i}
\]
for some $i\in \{1,\dots,r\}$.
\end{lemma}
\begin{proof}
By the proof of \autoref{lem:poisson2smooth} we know that there exist $u_i\in \ZZ_{\geq 0}$ independent of $G$ such that for any $(a,b)\in \QQ_S$ and $z_i\in q_i^{u_i}\ZZ_{q_i}$, we have
\[
\Phi_\alpha(a,b)=\Phi_\alpha((a,b)+z_i).
\]
Since $\bI_{k,f}(\alpha,\kappa)$ is the Fourier transform of $\Phi_\alpha$, there exist $L_1,\dots,L_r\in \ZZ_{\geq 0}$ independent of $G$ such that $\mathbf{I}_{k,f}(\alpha,\kappa)=0$ if 
\[
\left|\frac{\alpha\kappa}{\gcd(\alpha,kf^2)}\right|_{q_i}=|\alpha\kappa|_{q_i}\geq q_i^{L_i}
\]
for some $i\in \{1,\dots,r\}$.
\end{proof}

By making the change of variable $a\mapsto \sqrt{X}a$ and $x\mapsto Xx$ and note that $\theta_\infty$ is $Z_+$-invariant, we obtain
\begin{align*}
\mathbf{I}_{k,f}(\alpha,\kappa)
=&2X^{\frac32}\int_{(a,b)\in\QQ_S}\left\{\int_{(x,y)\in\QQ_S} G\left(\left|\frac{a^2}{4}-x\right|_\infty\left|\frac{b^2}{4}-y\right|_q \right) \theta_\infty\left(a,\frac{a^2}{4}-x\right)\theta_{q}\left(b,\frac{b^2}{4}-y\right)\right.\\
\times&\left.\left[F\legendresymbol{kf^2}{X^\vartheta|x|_\infty^{\vartheta}|y|_q'^{\vartheta}}
+\frac{kf^2}{\sqrt{X}|x|_\infty^{1/2}|y|_q'^{1/2}}V\legendresymbol{kf^2}{X^{\vartheta'}|x|_\infty^{\vartheta'} |y|_q'^{\vartheta'}}\right] \rme\legendresymbol{xX\alpha}{kf^2}\rme_q\legendresymbol{y\alpha}{kf^2}\rmd x\rmd y\right\}\\
\times&\rme\legendresymbol{a\sqrt{X}\alpha\kappa}{2\gcd(\alpha,kf^2)} \rme_q\legendresymbol{b\alpha\kappa}{2\gcd(\alpha,kf^2)}\rmd a\rmd b=\mathbf{I}^{1}_{k,f}(\alpha,\kappa)+ \mathbf{I}^{2}_{k,f}(\alpha,\kappa),
\end{align*}
where
\begin{align*}
\mathbf{I}^1_{k,f}(\alpha,\kappa)
=&2X^{\frac32}\int_{(a,b)\in\QQ_S}\left\{\int_{(x,y)\in\QQ_S} G\left(\left|\frac{a^2}{4}-x\right|_\infty\left|\frac{b^2}{4}-y\right|_q \right) \theta_\infty\left(a,\frac{a^2}{4}-x\right)\theta_{q}\left(b,\frac{b^2}{4}-y\right)\right.\\
\times&\left.F\legendresymbol{kf^2}{X^\vartheta|x|_\infty^{\vartheta}|y|_q'^{\vartheta}} \rme\legendresymbol{xX\alpha}{kf^2}\rme_q\legendresymbol{y\alpha}{kf^2}\rmd x\rmd y\right\}\rme\legendresymbol{a\sqrt{X}\alpha\kappa}{2\gcd(\alpha,kf^2)} \rme_q\legendresymbol{b\alpha\kappa}{2\gcd(\alpha,kf^2)}\rmd a\rmd b
\end{align*}
and
\begin{align*}
&\mathbf{I}^2_{k,f}(\alpha,\kappa)
=2X\int_{(a,b)\in\QQ_S}\left\{\int_{(x,y)\in\QQ_S} G\left(\left|\frac{a^2}{4}-x\right|_\infty\left|\frac{b^2}{4}-y\right|_q \right) \theta_\infty\left(a,\frac{a^2}{4}-x\right)\theta_{q}\left(b,\frac{b^2}{4}-y\right)\right.\\
\times&\left.\frac{kf^2}{|x|_\infty^{1/2}|y|_q'^{1/2}}V\legendresymbol{kf^2}{X^{\vartheta'} |x|_\infty^{\vartheta'}|y|_q'^{\vartheta'}} \rme\legendresymbol{xX\alpha}{kf^2}\rme_q\legendresymbol{y\alpha}{kf^2}\rmd x\rmd y\right\}\rme\legendresymbol{a\sqrt{X}\alpha\kappa}{2\gcd(\alpha,kf^2)} \rme_q\legendresymbol{b\alpha\kappa}{2\gcd(\alpha,kf^2)}\rmd a\rmd b.
\end{align*}

\begin{lemma}\label{lem:smoothxy}
For any $a\in \RR$, $b\in \QQ_{S_\fin}$, $u,v\in \ZZ_{\geq 0}$, $\sigma\in \{0,1\}$ and $\tau\in \{0,1\}^r$, let
\[
\Omega_{a,b}^{\sigma,\tau,u,v}(x,y)=\frac{\diff^u}{\diff a^u}G\left(\left|\frac{a^2}{4}-x\right|_\infty\left|\frac{b^2}{4}-y\right|_q \right) |x|^{-\frac{\sigma}{2}}\frac{\diff^v}{\diff a^v}\theta_{\infty,\sigma}\left(a,\frac{a^2}{4}-x\right)\prod_{i=1}^{r}|y_i|_{q_i}'^{-\frac{\tau_i}{2}} \theta_{q_i,\tau_i}\left(b_i,\frac{b_i^2}{4}-y_i\right).
\]
Then there exist $u_i\in \ZZ_{\geq 0}$ only depending on $\theta_q$ such that for any $(a,b)\in \QQ_S$ and $z_i\in q_i^{u_i}\ZZ_{q_i}$, we have
\[
\Omega_{a,b}^{\sigma,\tau,u,v}(x,y)=\Omega_{a,b}^{\sigma,\tau,u,v}((x,y)+z_i).
\]
Also, the support of $\Omega_{a,b}^{\sigma,\tau,u,v}$ is bounded, uniform in $G$ and $(a,b)\in \QQ_S$.
\end{lemma}
\begin{proof}
If $\Omega_{a,b}^{\sigma,\tau,u,v}(x,y)\neq 0$, then we must have $|b_i^2/4-y_i|_{q_i}\asymp 1$ for all $i$. Hence there exists $u_i$ such that $|b_i^2/4-y_i-z_i|_{q_i}=|b_i^2/4-y_i|_{q_i}$ for $y_i$ in the support and $z_i\in q_i^{u_i}\ZZ_{q_i}$. Also, since 
\[
|y_i|_{q_i}'^{-\frac{\tau_i}{2}}\theta_{q_i,\tau_i}\left(b_i,\frac{b_i^2}{4}-y_i\right)
\]
is smooth and compactly supported, we may choose $u_i$ above such that this function is invariant under $y_i\mapsto y_i+z_i$. Hence the first assertion follows.

Now we prove the second assertion. If $\Omega_{a,b}^{\sigma,\tau,u,v}(x,y)\neq 0$, then we must have $|b_i^2/4-y_i|_{q_i}\asymp 1$ for all $i$. Also, we must have $|b_i|\ll 1$. Hence $|y_i|\ll 1$.

Since $G\in C_c^\infty(\lopen 1/4,5/4\ropen)$, we also have $|a^2/4-x|\asymp 1$. Since $\theta_\infty$ is compactly supported modulo center, we must have $|a|\ll 1$. Therefore $|x|\ll 1$. Thus the second assertion is proved.
\end{proof}
\begin{proposition}\label{prop:estimatei1kf}
We have the following estimate for $\mathbf{I}^1_{k,f}(\alpha,\kappa)$:
\begin{enumerate}[itemsep=0pt,parsep=0pt,topsep=0pt, leftmargin=0pt,labelsep=2.5pt,itemindent=15pt,label=\upshape{(\arabic*)}]
  \item If $\kappa=0$, then for any $N\in\ZZ_{\geq 0}$ and $\varepsilon>0$,
\begin{comment}
\begin{align*}
\mathbf{I}^1_{k,f}(\alpha,\kappa)\ll&X^{\frac32}\sum_{j=0}^{M-1}\left(\frac{(kf^2)^2}{X}\left\llbracket \frac{X}{kf^2}\star\alpha\right\rrbracket\right)^{-N+\frac12+\varepsilon} \left\llbracket \frac{X}{kf^2}\star\alpha\right\rrbracket^{-1} \left(1+\left|\frac{X\alpha}{kf^2}\right|_\infty\right)^{-j+\varepsilon} \|G\|_{j,\infty}\\
+&X^{\frac32}\llbracket kf^2\star\alpha\rrbracket^{-N+\frac12+\varepsilon} 
\left\llbracket \frac{X}{kf^2}\star\alpha\right\rrbracket^{-1} \left(1+\left|\frac{X\alpha}{kf^2}\right|_\infty\right)^{-M+\varepsilon}\|G\|_{M+N+1,\infty}.
\end{align*}
\end{comment}
\begin{align*}
\mathbf{I}^1_{k,f}(\alpha,\kappa)%\ll&X^{\frac32}\left(\frac{(kf^2)^2}{X}\left\llbracket \frac{X}{kf^2}\star\alpha\right\rrbracket\right)^{-N+\varepsilon} \left\llbracket \frac{X}{kf^2}\star\alpha\right\rrbracket^{-1+\varepsilon}\|G\|_{N+1,\infty}\\
%+&X^{\frac32}\left(\frac{(kf^2)^2}{X}\left\llbracket \frac{X}{kf^2}\star\alpha\right\rrbracket\right)^{-N+\frac12+\varepsilon} \left\llbracket \frac{X}{kf^2}\star\alpha\right\rrbracket^{-1+\varepsilon}\left(1+\left|\frac{X\alpha}{kf^2}\right|\right)^{-\frac12}\|G\|_{N+1,\infty}\\
\ll&X^{\frac32}\left(\frac{(kf^2)^{1/\vartheta}}{X}\left\llbracket \frac{X}{kf^2}\star\alpha\right\rrbracket\right)^{-N+\frac12+\varepsilon} \left\llbracket \frac{X}{kf^2}\star\alpha\right\rrbracket^{-1+\varepsilon}\|G\|_{N+1,1}
\end{align*}
\item If $\kappa\neq 0$, then for any $N,K\in \ZZ_{\geq 0}$ and $\varepsilon>0$,
\[
\mathbf{I}^1_{k,f}(\alpha,\kappa)\ll X^{\frac32}\legendresymbol{\gcd(\alpha,kf^2)}{\sqrt{X}|\alpha\kappa|}^{K} \left(\frac{(kf^2)^{1/\vartheta}}{X}\left\llbracket \frac{X}{kf^2}\star\alpha\right\rrbracket\right)^{-N+\frac12+\varepsilon} \left\llbracket\frac{X}{kf^2}\star\alpha\right\rrbracket^{-1+\varepsilon} \|G\|_{N+K+1,1}.
\]
\end{enumerate}
The implied constants only depend on $\theta_\infty$, $\theta_q$, $N,K$ and $\varepsilon$.
\end{proposition}
\begin{proof}
(1) If $\kappa=0$,
\begin{align*}
\mathbf{I}^1_{k,f}(\alpha,0)
=&2X^{\frac32}\int_{(a,b)\in\QQ_S}\left\{\int_{(x,y)\in\QQ_S} G\left(\left|\frac{a^2}{4}-x\right|_\infty\left|\frac{b^2}{4}-y\right|_q \right) \theta_\infty\left(a,\frac{a^2}{4}-x\right)\theta_{q}\left(b,\frac{b^2}{4}-y\right)\right.\\
\times&\left.F\legendresymbol{kf^2}{X^\vartheta|x|_\infty^{\vartheta}|y|_q'^{\vartheta}} \rme\legendresymbol{xX\alpha}{kf^2}\rme_q\legendresymbol{y\alpha}{kf^2}\rmd x\rmd y\right\}\rmd a\rmd b.
\end{align*}
For $\iota\in \{0,1\}$, $\epsilon\in \{0,\pm 1\}^r$, $\sigma\in \{0,1\}$, $\tau\in \{0,1\}^r$, we consider the integral
\begin{equation}\label{eq:integralf}
\begin{split}
&X^{\frac32}\int_{(a,b)\in\QQ_S}\left\{\int_{X_\iota\times Y_\epsilon} G\left(\left|\frac{a^2}{4}-x\right|_\infty\left|\frac{b^2}{4}-y\right|_q \right) \theta_{\infty,\sigma}\left(a,\frac{a^2}{4}-x\right)\theta_{q,\tau}\left(b,\frac{b^2}{4}-y\right)\right.\\
\times&\left.F\legendresymbol{kf^2}{X^\vartheta|x|_\infty^{\vartheta}|y|_q'^{\vartheta}} \rme\legendresymbol{xX\alpha}{kf^2}\rme_q\legendresymbol{y\alpha}{kf^2}\rmd x\rmd y\right\}\rmd a\rmd b.
\end{split}
\end{equation}

The conditions of \autoref{thm:mainfourierestimate} are satisfied by \autoref{lem:smoothxy}. Hence by \autoref{thm:mainfourierestimate} we obtain
\[
\eqref{eq:integralf}\ll X^{\frac32}\left(\legendresymbol{kf^2}{X^\vartheta}^{1/\vartheta}\left\llbracket \frac{X}{kf^2}\star\alpha\right\rrbracket\right)^{-\vartheta\rho} \left(1+\left|\frac{X\alpha}{kf^2}\right|_\infty\right)^{-\lceil\vartheta\rho+\frac{\sigma}{2} \rceil+ \vartheta\rho} \prod_{i=1}^{r}(1+|\alpha|_{q_i})^{-1-\frac{\tau_i}{2}}\|\Phi\|_{\lceil\vartheta\rho+ \frac{\sigma}{2}\rceil,1}
\]
for any $\rho>-1/\vartheta$, where
\[ 
\|\Phi\|_{i,1}=\int_{(a,b)\in\QQ_S}\left\|G\left(\left|\frac{a^2}{4}-x\right|_\infty\left|\frac{b^2}{4}-y\right|_q \right)\theta_{\infty,\sigma}\left(a,\frac{a^2}{4}-x\right)\theta_{q,\tau}\left(b,\frac{b^2}{4}-y\right) \right\|_{i,1}\rmd a\rmd b.
\]
Since the support of the integrand is compact and independent of $G$ (by the proof of \autoref{lem:poisson2smooth}), we obtain $\|\Phi\|_{i,\infty}\ll\|G\|_{i,\infty}$. Hence
\begin{align*}
\eqref{eq:integralf}&\ll X^{\frac32}\left(\frac{(kf^2)^{1/\vartheta}}{X}\left\llbracket \frac{X}{kf^2}\star\alpha\right\rrbracket\right)^{-\vartheta\rho} \left(1+\left|\frac{X\alpha}{kf^2}\right|_\infty\right)^{-\lceil\vartheta\rho+\frac{\sigma}{2}\rceil+ \vartheta\rho} \prod_{i=1}^{r}(1+|\alpha|_{q_i})^{-1}\|G\|_{\lceil\vartheta\rho+\frac{\sigma}{2}\rceil,1}\\
&\ll X^{\frac32}\left(\frac{(kf^2)^{1/\vartheta}}{X}\left\llbracket \frac{X}{kf^2}\star\alpha\right\rrbracket\right)^{-\vartheta\rho} 
\left\llbracket \frac{X}{kf^2}\star\alpha\right\rrbracket^{-1} \left(1+\left|\frac{X\alpha}{kf^2}\right|_\infty\right)^{-\lceil\vartheta\rho+\frac{\sigma}{2}\rceil+ \vartheta\rho+1}\|G\|_{\lceil\vartheta\rho+\frac{\sigma}{2}\rceil,1}.
\end{align*}
Now we choose $\rho=(N-\sigma/2+\varepsilon)/\vartheta$, and then sum over $\iota,\epsilon,\sigma,\tau$, we obtain our result. 

(2) If $\kappa\neq 0$, then
\begin{align*}
\mathbf{I}^1_{k,f}(\alpha,\kappa)
=&2X^{\frac32}\int_{(a,b)\in\QQ_S}\left\{\int_{(x,y)\in\QQ_S} G\left(\left|\frac{a^2}{4}-x\right|_\infty\left|\frac{b^2}{4}-y\right|_q \right) \theta_\infty\left(a,\frac{a^2}{4}-x\right)\theta_{q}\left(b,\frac{b^2}{4}-y\right)\right.\\
\times&\left.F\legendresymbol{kf^2}{X^\vartheta|x|_\infty^{\vartheta}|y|_q'^{\vartheta}} \rme\legendresymbol{xX\alpha}{kf^2}\rme_q\legendresymbol{y\alpha}{kf^2}\rmd x\rmd y\right\}\rme\legendresymbol{a\sqrt{X}\alpha\kappa}{2\gcd(\alpha,kf^2)} \rme_q\legendresymbol{b\alpha\kappa}{2\gcd(\alpha,kf^2)}\rmd a\rmd b.
\end{align*}
Using integration by parts $K$ times and Leibniz rule, we obtain
\begin{align*}
\mathbf{I}^1_{k,f}(\alpha,\kappa)
=&2X^{\frac32}\legendresymbol{\gcd(\alpha,kf^2)}{\sqrt{X}\alpha\kappa}^{K}\sum_{j=1}^{K} \binom{K}{j} \int_{(a,b)\in\QQ_S}\left\{\int_{(x,y)\in\QQ_S} \frac{\diff^j}{\diff a^j}G\left(\left|\frac{a^2}{4}-x\right|_\infty\left|\frac{b^2}{4}-y\right|_q \right)\right. \\
\times&\frac{\diff^{K-j}}{\diff a^{K-j}}\theta_\infty\left(a,\frac{a^2}{4}-x\right)\theta_{q}\left(b,\frac{b^2}{4}-y\right) \left.F\legendresymbol{kf^2}{X^\vartheta|x|_\infty^{\vartheta}|y|_q'^{\vartheta}} \rme\legendresymbol{xX\alpha}{kf^2}\rme_q\legendresymbol{y\alpha}{kf^2}\rmd x\rmd y\right\}\\
\times&\rme\legendresymbol{a\sqrt{X}\alpha\kappa}{2\gcd(\alpha,kf^2)} \rme_q\legendresymbol{b\alpha\kappa}{2\gcd(\alpha,kf^2)}\rmd a\rmd b.
\end{align*}
For $\iota\in \{0,1\}$, $\epsilon\in \{0,\pm 1\}^r$, $\sigma\in \{0,1\}$, $\tau\in \{0,1\}^r$, we consider the integral
\begin{equation}\label{eq:integralf2}
\begin{split}
&X^{\frac32}\legendresymbol{\gcd(\alpha,kf^2)}{\sqrt{X}\alpha\kappa}^{K}\sum_{j=1}^{K} \binom{K}{j} \int_{(a,b)\in\QQ_S}\left\{\int_{X_\iota\times Y_\epsilon} \frac{\diff^j}{\diff a^j}G\left(\left|\frac{a^2}{4}-x\right|_\infty\left|\frac{b^2}{4}-y\right|_q \right)\right. \\
\times&\frac{\diff^{K-j}}{\diff a^{K-j}}\theta_{\infty,\sigma}\left(a,\frac{a^2}{4}-x\right)\theta_{q,\tau}\left(b,\frac{b^2}{4}-y\right) \left.F\legendresymbol{kf^2}{X^\vartheta|x|_\infty^{\vartheta}|y|_q'^{\vartheta}} \rme\legendresymbol{xX\alpha}{kf^2}\rme_q\legendresymbol{y\alpha}{kf^2}\rmd x\rmd y\right\}\\
\times&\rme\legendresymbol{a\sqrt{X}\alpha\kappa}{2\gcd(\alpha,kf^2)} \rme_q\legendresymbol{b\alpha\kappa}{2\gcd(\alpha,kf^2)}\rmd a\rmd b.
\end{split}
\end{equation}

The conditions of \autoref{thm:mainfourierestimate} are satisfied by \autoref{lem:smoothxy}. Hence by \autoref{thm:mainfourierestimate}, we obtain
\begin{align*}
\eqref{eq:integralf2}\ll& X^{\frac32}\legendresymbol{\gcd(\alpha,kf^2)}{\sqrt{X}|\alpha\kappa|}^{K}\sum_{j=0}^{K} \left(\legendresymbol{kf^2}{X^\vartheta}^{1/\vartheta}\left\llbracket \frac{X}{kf^2}\star\alpha\right\rrbracket\right)^{-\vartheta\rho} \left(1+\left|\frac{X\alpha}{kf^2}\right|_\infty\right)^{-\lceil\vartheta\rho+\frac{\sigma}{2}\rceil+\vartheta\rho}\\
\times& \prod_{i=1}^{r}(1+|\alpha|_{q_i})^{-1-\frac{\tau_i}{2}} \|\Phi^{j,K-j}\|_{\lceil\vartheta\rho+\frac{\sigma}{2}\rceil,1},
\end{align*}
where
\[ 
\|\Phi^{j,K-j}\|_{m,1}=\int_{(a,b)\in\QQ_S}\left\|\frac{\diff^j}{\diff a^j}G\left(\left|\frac{a^2}{4}-x\right|_\infty\left|\frac{b^2}{4}-y\right|_q \right)\frac{\diff^{K-j}}{\diff a^{K-j}}\theta_{\infty,\sigma}\left(a,\frac{a^2}{4}-x\right)\theta_{q,\tau} \left(b,\frac{b^2}{4}-y\right) \right\|_{m,1}\rmd a\rmd b.
\]
Also, we find that $\|\Phi^{j,K-j}\|_{m,1}\ll \|G\|_{j+m,1}$.

Therefore \eqref{eq:integralf2} is bounded by
\begin{align*}
&X^{\frac32}\legendresymbol{\gcd(\alpha,kf^2)}{\sqrt{X}|\alpha\kappa|}^{K}\left(\legendresymbol {kf^2}{X^\vartheta}^{1/\vartheta}\left\llbracket \frac{X}{kf^2}\star\alpha\right\rrbracket\right)^{-\vartheta\rho} \\
\times&\left(1+\left|\frac{X\alpha}{kf^2}\right|_\infty\right)^{-\lceil\vartheta\rho+\frac{\sigma}{2}\rceil+ \vartheta\rho} \prod_{i=1}^{r}(1+|\alpha|_{q_i})^{-1} \|G\|_{\lceil\vartheta\rho+\frac{\sigma}{2}\rceil+K,1}\\
\ll & X^{\frac32}\legendresymbol{\gcd(\alpha,kf^2)}{\sqrt{X}|\alpha\kappa|}^{K}\left(\frac{(kf^2)^{1/\vartheta}}{X}\left\llbracket \frac{X}{kf^2}\star\alpha\right\rrbracket\right)^{-\vartheta\rho} \left\llbracket\frac{X}{kf^2}\star\alpha\right\rrbracket^{-1} \\
\times&\left(1+\left|\frac{X\alpha}{kf^2}\right|_\infty\right)^{-\lceil\vartheta\rho+\frac{\sigma}{2}\rceil+ \vartheta\rho+1} \|G\|_{\lceil\vartheta\rho+\frac{\sigma}{2}\rceil+K,1}.
\end{align*}
Now we choose $\rho=(N-\sigma/2+\varepsilon)/\vartheta$, and then sum over $\iota,\epsilon,\sigma,\tau$, we obtain our result. 
\end{proof}

\begin{proposition}\label{prop:estimatei2kf}
 We have the following estimate for $\mathbf{I}^2_{k,f}(\alpha,\kappa)$:
\begin{enumerate}[itemsep=0pt,parsep=0pt,topsep=0pt, leftmargin=0pt,labelsep=2.5pt,itemindent=15pt,label=\upshape{(\arabic*)}]
  \item If $\kappa=0$, then for any $N\in \ZZ_{\geq 0}$ and $\varepsilon>0$,
\begin{comment}
\begin{align*}
\mathbf{I}^2_{k,f}(\alpha,\kappa)&\ll Xkf^2\sum_{j=0}^{M-1}\left(\frac{(kf^2)^2}{X}\left\llbracket \frac{X}{kf^2}\star\alpha\right\rrbracket\right)^{-N+\varepsilon} \left\llbracket \frac{X}{kf^2}\star\alpha\right\rrbracket^{-\frac12} \left(1+\left|\frac{X\alpha}{kf^2}\right|_\infty\right)^{-j+\varepsilon} \|G\|_{j,\infty}\\
&+Xkf^2\left(\frac{(kf^2)^2}{X}\left\llbracket \frac{X}{kf^2}\star\alpha\right\rrbracket\right)^{-N+\varepsilon} 
\left\llbracket \frac{X}{kf^2}\star\alpha\right\rrbracket^{-\frac12} \left(1+\left|\frac{X\alpha}{kf^2}\right|_\infty\right)^{-M+\varepsilon}\|G\|_{M+N+1,\infty}.
\end{align*}
\end{comment}
\begin{align*}
\mathbf{I}^2_{k,f}(\alpha,\kappa)&\ll Xkf^2\left(\frac{(kf^2)^{1/\vartheta'}}{X}\left\llbracket \frac{X}{kf^2}\star\alpha\right\rrbracket\right)^{-N+\varepsilon} \left\llbracket \frac{X}{kf^2}\star\alpha\right\rrbracket^{-\frac12+\varepsilon}\|G\|_{N+1,1}
\end{align*}
\item If $\kappa\neq 0$, then for any $N,K\in \ZZ_{\geq 0}$ and $\varepsilon>0$,
\[
\mathbf{I}^2_{k,f}(\alpha,\kappa)\ll Xkf^2\legendresymbol{\gcd(\alpha,kf^2)}{\sqrt{X}|\alpha\kappa|}^{K}\left(\frac{(kf^2)^{1/\vartheta'}}{X}\left\llbracket \frac{X}{kf^2}\star\alpha\right\rrbracket\right)^{-N+\varepsilon} \left\llbracket\frac{X}{kf^2}\star\alpha\right\rrbracket^{-\frac12+\varepsilon} \|G\|_{N+K+1,1}.
\]
\end{enumerate}
The implied constants only depend on $\theta_\infty$, $\theta_q$, $N,K$ and $\varepsilon$.
\end{proposition}
\begin{proof}
(1) If $\kappa=0$,
\begin{align*}
&\mathbf{I}^2_{k,f}(\alpha,\kappa)
=2Xkf^2\int_{(a,b)\in\QQ_S}\left\{\int_{(x,y)\in\QQ_S} G\left(\left|\frac{a^2}{4}-x\right|_\infty\left|\frac{b^2}{4}-y\right|_q \right) |x|_\infty^{-1/2}\theta_\infty\left(a,\frac{a^2}{4}-x\right)\right.\\
\times&\left.|y|_q'^{-1/2}\theta_{q}\left(b,\frac{b^2}{4}-y\right)V\legendresymbol{kf^2}{X^{\vartheta'} |x|_\infty^{\vartheta'}|y|_q'^{\vartheta'}} \rme\legendresymbol{xX\alpha}{kf^2}\rme_q\legendresymbol{y\alpha}{kf^2}\rmd x\rmd y\right\}\rmd a\rmd b.
\end{align*}
For $\iota\in \{0,1\}$, $\epsilon\in \{0,\pm 1\}^r$, $\sigma\in \{0,1\}$, $\tau\in \{0,1\}^r$, we consider the integral
\begin{equation}\label{eq:integralk}
\begin{split}
&Xkf^2\int_{(a,b)\in\QQ_S}\left\{\int_{(x,y)\in\QQ_S} G\left(\left|\frac{a^2}{4}-x\right|_\infty\left|\frac{b^2}{4}-y\right|_q \right) |x|_\infty^{-1/2}\theta_{\infty,\sigma}\left(a,\frac{a^2}{4}-x\right)\right.\\
\times&\left.|y|_q'^{-1/2}\theta_{q,\tau}\left(b,\frac{b^2}{4}-y\right)V\legendresymbol{kf^2} {X^{\vartheta'}|x|_\infty^{\vartheta'}|y|_q'^{\vartheta'}} \rme\legendresymbol{xX\alpha}{kf^2}\rme_q\legendresymbol{y\alpha}{kf^2}\rmd x\rmd y\right\}\rmd a\rmd b.
\end{split}
\end{equation}

The conditions of \autoref{thm:mainfourierestimate} are satisfied by \autoref{lem:smoothxy}. Hence by \autoref{thm:mainfourierestimate}, \eqref{eq:integralk} is  bounded by
\begin{align*}
& Xkf^2\left(\legendresymbol{kf^2}{X^{\vartheta'}}^{1/\vartheta'}\left\llbracket \frac{X}{kf^2}\star\alpha\right\rrbracket\right)^{-\vartheta'\rho} \left(1+\left|\frac{X\alpha}{kf^2}\right|_\infty\right)^{-\lceil\vartheta'\rho+\frac{\sigma-1}{2}\rceil+ \vartheta'\rho} \prod_{i=1}^{r}(1+|\alpha|_{q_i})^{-\frac12-\frac{\tau_i}{2}}\|\Psi\|_{\lceil\vartheta'\rho+\frac{\sigma-1}{2}\rceil,\infty}
\end{align*}
for any $\rho>-2$, where
\[ 
\|\Psi\|_{m,1}=\int_{(a,b)\in\QQ_S}\left\|G\left(\left|\frac{a^2}{4}-x\right|_\infty\left|\frac{b^2}{4}-y\right|_q \right)|x|^{-\frac12}\theta_{\infty,\sigma}\left(a,\frac{a^2}{4}-x\right)|y|_q'^{-\frac12}\theta_{q,\tau}\left(b,\frac{b^2}{4}-y\right) \right\|_{m,1}\rmd a\rmd b
\]
As in the proof of above lemma, we have $\|\Psi\|_{m,1}\ll\|G\|_{m,1}$. Hence
\begin{align*}
\eqref{eq:integralk}&\ll Xkf^2\left(\frac{(kf^2)^{1/\vartheta'}}{X}\left\llbracket \frac{X}{kf^2}\star\alpha\right\rrbracket\right)^{-\vartheta'\rho} \left(1+\left|\frac{X\alpha}{kf^2}\right|_\infty\right)^{-\lceil\vartheta'\rho+\frac{\sigma-1}{2}\rceil+\vartheta'\rho} \prod_{i=1}^{r}(1+|\alpha|_{q_i})^{-\frac12}\|G\|_{\lceil\vartheta'\rho+\frac{\sigma-1}{2}\rceil,1}\\
&\ll Xkf^2\left(\frac{(kf^2)^{1/\vartheta'}}{X}\left\llbracket \frac{X}{kf^2}\star\alpha\right\rrbracket\right)^{-\vartheta'\rho} 
\left\llbracket \frac{X}{kf^2}\star\alpha\right\rrbracket^{-\frac12} \left(1+\left|\frac{X\alpha}{kf^2}\right|_\infty\right)^{-\lceil\vartheta'\rho+\frac{\sigma-1}{2}\rceil+ \vartheta'\rho+\frac12}\|G\|_{\lceil\vartheta'\rho+\frac{\sigma-1}{2}\rceil,1}.
\end{align*}
Now we choose $\rho=(N-\sigma/2+\varepsilon+1/2)/\vartheta'$, and then sum over $\iota,\epsilon,\sigma,\tau$, we obtain our result. 

(2) If $\kappa\neq 0$, then
\begin{align*}
\mathbf{I}^2_{k,f}(\alpha,\kappa)
=&2Xkf^2\int_{(a,b)\in\QQ_S}\left\{\int_{(x,y)\in\QQ_S} G\left(\left|\frac{a^2}{4}-x\right|_\infty\left|\frac{b^2}{4}-y\right|_q \right) |x|^{-1/2}\theta_\infty\left(a,\frac{a^2}{4}-x\right) \right.\\
\times&|y|_q'^{-1/2}\theta_{q}\left(b,\frac{b^2}{4}-y\right)\left.V\legendresymbol{kf^2} {X^{\vartheta'}|x|_\infty^{\vartheta'}|y|_q'^{\vartheta'}} \rme\legendresymbol{xX\alpha}{kf^2}\rme_q\legendresymbol{y\alpha}{kf^2}\rmd x\rmd y\right\}\\
\times&\rme\legendresymbol{a\sqrt{X}\alpha\kappa}{2\gcd(\alpha,kf^2)} \rme_q\legendresymbol{b\alpha\kappa}{2\gcd(\alpha,kf^2)}\rmd a\rmd b.
\end{align*}
Using integration by parts $K$ times and Leibniz rule, we obtain
\begin{align*}
\mathbf{I}^2_{k,f}(\alpha,\kappa)
=&2Xkf^2\legendresymbol{\gcd(\alpha,kf^2)}{\uppi\rmi\sqrt{X}\alpha\kappa}^{K}\sum_{j=1}^{K} \binom{K}{j} \int_{(a,b)\in\QQ_S}\left\{\int_{(x,y)\in\QQ_S} \frac{\diff^j}{\diff a^j}G\left(\left|\frac{a^2}{4}-x\right|_\infty\left|\frac{b^2}{4}-y\right|_q \right)|x|^{-\frac12}\right. \\
\times&\frac{\diff^{K-j}}{\diff a^{K-j}}\theta_\infty\left(a,\frac{a^2}{4}-x\right)|y|_q'^{-\frac12}\theta_{q}\left(b,\frac{b^2}{4}-y\right) \left.V\legendresymbol{kf^2}{X^{\vartheta'}|x|_\infty^{\vartheta'}|y|_q'^{\vartheta'}} \rme\legendresymbol{xX\alpha}{kf^2}\rme_q\legendresymbol{y\alpha}{kf^2}\rmd x\rmd y\right\}\\
\times&\rme\legendresymbol{a\sqrt{X}\alpha\kappa}{2\gcd(\alpha,kf^2)} \rme_q\legendresymbol{b\alpha\kappa}{2\gcd(\alpha,kf^2)}\rmd a\rmd b.
\end{align*}
For $\iota\in \{0,1\}$, $\epsilon\in \{0,\pm 1\}^r$, $\sigma\in \{0,1\}$, $\tau\in \{0,1\}^r$, we consider the integral
\begin{equation}\label{eq:integralk2}
\begin{split}
&Xkf^2\legendresymbol{\gcd(\alpha,kf^2)}{\uppi\rmi\sqrt{X}\alpha\kappa}^{K}\sum_{j=1}^{K} \binom{K}{j} \int_{(a,b)\in\QQ_S}\left\{\int_{X_\iota\times Y_\epsilon} \frac{\diff^j}{\diff a^j}G\left(\left|\frac{a^2}{4}-x\right|_\infty\left|\frac{b^2}{4}-y\right|_q \right)|x|^{-\frac12}\right. \\
\times&\frac{\diff^{K-j}}{\diff a^{K-j}}\theta_{\infty,\sigma}\left(a,\frac{a^2}{4}-x\right)|y|_q'^{-\frac12}\theta_{q,\tau}\left(b,\frac{b^2} {4}-y\right)\left.V\legendresymbol{kf^2}{X^{\vartheta'}|x|_\infty^{\vartheta'}|y|_q'^{\vartheta'}} \rme\legendresymbol{xX\alpha}{kf^2}\rme_q\legendresymbol{y\alpha}{kf^2}\rmd x\rmd y\right\}\\
\times&\rme\legendresymbol{a\sqrt{X}\alpha\kappa}{2\gcd(\alpha,kf^2)} \rme_q\legendresymbol{b\alpha\kappa}{2\gcd(\alpha,kf^2)}\rmd a\rmd b.
\end{split}
\end{equation}

The conditions of \autoref{thm:mainfourierestimate} are satisfied by \autoref{lem:smoothxy}. Hence by \autoref{thm:mainfourierestimate}, we obtain
\begin{align*}
\eqref{eq:integralk2}\ll & Xkf^2\legendresymbol{\gcd(\alpha,kf^2)}{\sqrt{X}|\alpha\kappa|}^{K}\sum_{j=0}^{K}\left(\legendresymbol {kf^2}{X^{\vartheta'}}^{1/\vartheta'} \left\llbracket \frac{X}{kf^2}\star\alpha\right\rrbracket\right)^{-\vartheta'\rho} \left(1+\left|\frac{X\alpha}{kf^2}\right|_\infty\right)^{-\lceil\vartheta'\rho+\frac{\sigma-1}{2}\rceil+ \vartheta'\rho} \\
\times&\prod_{i=1}^{r}(1+|\alpha|_{q_i})^{-\frac12-\frac{\tau_i}{2}} \|\Psi^{j,K-j}\|_{\lceil\vartheta'\rho+\frac{\sigma}{2}\rceil,1},
\end{align*}
where $\|\Psi^{j,K-j}\|_{m,1}$ is
\[ 
\int_{(a,b)\in\QQ_S}\left\|\frac{\diff^j}{\diff a^j}G\left(\left|\frac{a^2}{4}-x\right|_\infty\left|\frac{b^2}{4}-y\right|_q \right)|x|^{-\frac12}\frac{\diff^{K-j}}{\diff a^{K-j}}\theta_{\infty,\sigma}\left(a,\frac{a^2}{4}-x\right)|y|_q'^{-\frac12}\theta_{q,\tau}\left(b,\frac{b^2}{4}-y\right) \right\|_{m,1}\rmd a\rmd b.
\]
Also, we find that $\|\Psi^{j,K-j}\|_{m,1}\ll \|G\|_{j+m,1}$.

Therefore \eqref{eq:integralk2} is bounded by
\begin{align*}
& Xkf^2\legendresymbol{\gcd(\alpha,kf^2)}{\sqrt{X}|\alpha\kappa|}^{K}\left(\frac{(kf^2)^{1/\vartheta'}}{X} \left\llbracket \frac{X}{kf^2}\star\alpha\right\rrbracket\right)^{-\vartheta'\rho} \\
\times&\left(1+\left|\frac{X\alpha}{kf^2}\right|_\infty\right)^{-\lceil\vartheta'\rho +\frac{\sigma-1}{2}\rceil+ \vartheta'\rho}\prod_{i=1}^{r}(1+|\alpha|_{q_i})^{-\frac12} \|G\|_{\lceil\vartheta'\rho+\frac{\sigma-1}{2}\rceil+K,1}\\
\ll & Xkf^2\legendresymbol{\gcd(\alpha,kf^2)}{\sqrt{X}|\alpha\kappa|}^{K}\left(\frac{(kf^2)^{1/\vartheta'}}{X}\left\llbracket \frac{X}{kf^2}\star\alpha\right\rrbracket\right)^{-\vartheta'\rho} \left\llbracket\frac{X}{kf^2}\star\alpha\right\rrbracket^{-\frac12}\\
\times& \left(1+\left|\frac{X\alpha}{kf^2}\right|_\infty\right)^{-\lceil\vartheta'\rho+\frac{\sigma-1}{2}\rceil +\vartheta'\rho+\frac12} \|G\|_{\lceil\vartheta'\rho+\frac{\sigma-1}{2}\rceil+K,1}.
\end{align*}
Now we choose $\rho=(N-\sigma/2+\varepsilon+1/2)/\vartheta'$, and then sum over $\iota,\epsilon,\sigma,\tau$, we obtain our result. 
\end{proof}

\begin{comment}
Combining \autoref{prop:estimatei1kf} and \autoref{prop:estimatei2kf} we obtain
\begin{corollary}\label{cor:estimateikf}
We have the following estimate for $\mathbf{I}_{k,f}(\alpha,\kappa)$:
\begin{enumerate}[itemsep=0pt,parsep=0pt,topsep=0pt, leftmargin=0pt,labelsep=2.5pt,itemindent=15pt,label=\upshape{(\arabic*)}]
  \item If $\kappa=0$, then for any $M,N\in \ZZ_{\geq 0}$ and $\varepsilon>0$,
\begin{align*}
\mathbf{I}_{k,f}(\alpha,\kappa)&\ll X^\frac32\left(\frac{(kf^2)^2}{X}\left\llbracket \frac{X}{kf^2}\star\alpha\right\rrbracket\right)^{-N+\frac12+\varepsilon} \left\llbracket \frac{X}{kf^2}\star\alpha\right\rrbracket^{-1+\varepsilon}\|G\|_{N,\infty}
\end{align*}
\item If $\kappa\neq 0$, then for any $M,N\in \ZZ_{\geq 0}$ and $\varepsilon>0$,
\[
\mathbf{I}_{k,f}(\alpha,\kappa)\ll X^\frac32\legendresymbol{\gcd(\alpha,kf^2)}{\sqrt{X}|\alpha\kappa|}^{K}\left(\frac{(kf^2)^2}{X}\left\llbracket \frac{X}{kf^2}\star\alpha\right\rrbracket\right)^{-N+\frac12+\varepsilon} \left\llbracket\frac{X}{kf^2}\star\alpha\right\rrbracket^{-1+\varepsilon} \|G\|_{N+K+1,\infty}.
\]
\end{enumerate}
\end{corollary}
\begin{proof}
Clearly the right hand sides of (1) and (2) are equal in \autoref{prop:estimatei1kf} and \autoref{prop:estimatei2kf}.
\end{proof}
\end{comment}

\subsection{Final estimate}
In this subsection we will estimate of $S_G^{\alpha\neq 0}(X)$. We have
\begin{align*}
S_G^{\alpha\neq 0}(X)&=\sum_{k,f\in \ZZ_{(S)}^{>0}}\sum_{\alpha\in \ZZ^S-\{0\}}\frac{c_{k,f}(\alpha)\gamma_{k,f}(\alpha)}{k^3f^5\gcd(\alpha,kf^2)}|\alpha^{(q)}|^{\frac12}\!\sum_{\kappa\in \ZZ^S}\rme\legendresymbol{\kappa^2\alpha kf^2}{4\gcd(\alpha,kf^2)^2}\rme_q \legendresymbol{\kappa^2\alpha kf^2}{4\gcd(\alpha,kf^2)^2}\mathbf{I}_{k,f}(\alpha,\kappa)\\
&\ll \sum_{k,f\in \ZZ_{(S)}^{>0}}\sum_{\alpha\in \ZZ^S-\{0\}}\frac{|c_{k,f}(\alpha)|}{k^3f^5\gcd(\alpha,kf^2)}|\alpha^{(q)}|^{\frac12}\sum_{\kappa\in \ZZ^S}|\mathbf{I}_{k,f}(\alpha,\kappa)|.
\end{align*}
By \autoref{prop:estimateckf}, we have
$|c_{k,f}(\alpha)|\leq k^{3/2}f\sqrt{\gcd(kf^2,\alpha)}$. Hence $S_G^{\alpha\neq 0}(X)$ is bounded by
\[
\sum_{k,f\in \ZZ_{(S)}^{>0}}\sum_{\alpha\in \ZZ^S-\{0\}}\frac{|\alpha^{(q)}|^{\frac12}}{k^{3/2}f^4\sqrt{\gcd(\alpha,kf^2)}}\sum_{\kappa\in \ZZ^S}|\mathbf{I}_{k,f}(\alpha,\kappa)|\ll \sum_{k,f\in \ZZ_{(S)}^{>0}}\sum_{\alpha\in \ZZ^S-\{0\}}\frac{|\alpha^{(q)}|^{\frac12}}{k^{3/2}f^4}\sum_{\kappa\in \ZZ^S}|\mathbf{I}_{k,f}(\alpha,\kappa)|.
\]

\subsubsection{Bound of $\kappa=0$} 
We first bound the $\kappa=0$ term. A same argument above shows that
\begin{equation}\label{eq:estimatekappa0}
\begin{split}
   S_G^{\alpha\neq 0,\kappa=0}(X) & \ll \sum_{k,f\in \ZZ_{(S)}^{>0}}\sum_{\alpha\in \ZZ^S-\{0\}}\frac{|\alpha^{(q)}|^{\frac12}}{k^{3/2}f^4}|\mathbf{I}_{k,f}(\alpha,0)| \\
     & \ll \sum_{k,f\in \ZZ_{(S)}^{>0}}\sum_{\alpha\in \ZZ^S-\{0\}}\frac{|\alpha^{(q)}|^{\frac12}}{k^{3/2}f^4}(|\mathbf{I}^1_{k,f}(\alpha,0)|+|\mathbf{I}^2_{k,f}(\alpha,0)|).
\end{split}
\end{equation}

\begin{proposition}\label{prop:finalestimateikfkappa0}
For any $\varepsilon>0$, $c\in \lopen 0,1\ropen$, and $N\in \ZZ_{\geq 0}$, there exists $\varepsilon'>0$ such that if $|\vartheta-1/2|<\varepsilon'$, then
\[
 \sum_{k,f\in \ZZ_{(S)}^{>0}}\sum_{\alpha\in \ZZ^S-\{0\}}\frac{|\alpha^{(q)}|^{\frac12}}{k^{3/2}f^4}|\mathbf{I}^1_{k,f}(\alpha,0)|\ll X^{\frac12-cN}\|G\|_{N+3,1}+X^{\frac12+c+\varepsilon}\|G\|_{1,1},
\]
where the implied constant only depends on $f_\infty$, $f_{q_i}$, $\varepsilon$, $c$, and $N$.
\end{proposition}
\begin{proof}
We split the sum into the following two terms
\begin{equation}\label{eq:finalestimatei1kf1}
\sum_{\substack{k,f\in \ZZ_{(S)}^{>0},\,\alpha\in \ZZ^S-\{0\}\\ \frac{(kf^2)^{1/\vartheta}}{X}\left\llbracket\frac{X}{kf^2}\star\alpha\right\rrbracket\gg X^c}}\frac{|\alpha^{(q)}|^{\frac12}}{k^{3/2}f^4}|\mathbf{I}^1_{k,f}(\alpha,0)|
\end{equation}
and
\begin{equation}\label{eq:finalestimatei1kf2}
\sum_{\substack{k,f\in \ZZ_{(S)}^{>0},\,\alpha\in \ZZ^S-\{0\}\\ \frac{(kf^2)^{1/\vartheta}}{X}\left\llbracket\frac{X}{kf^2}\star\alpha\right\rrbracket\ll X^c}}\frac{|\alpha^{(q)}|^{\frac12}}{k^{3/2}f^4}|\mathbf{I}_{k,f}^1(\alpha,0)|.
\end{equation}

We first consider \eqref{eq:finalestimatei1kf1}. By \autoref{prop:estimatei1kf} (1) with $N$ replaced by $N+2$ we have
\begin{align*}
\eqref{eq:finalestimatei1kf1}
\ll&X^{\frac32}\sum_{\substack{k,f\in \ZZ_{(S)}^{>0},\,\alpha\in \ZZ^S-\{0\}\\ \frac{(kf^2)^{1/\vartheta}}{X}\left\llbracket\frac{X}{kf^2}\star\alpha\right\rrbracket\gg X^c}}\frac{|\alpha^{(q)}|^{\frac12}}{k^{3/2}f^4}\left(\frac{(kf^2)^{1/\vartheta}}{X}\left\llbracket \frac{X}{kf^2}\star\alpha\right\rrbracket\right)^{-N-\frac32+\varepsilon} \left\llbracket \frac{X}{kf^2}\star\alpha\right\rrbracket^{-1+\varepsilon}\|G\|_{N+3,1}\\
\ll&X^{\frac32}X^{-cN}\sum_{\substack{k,f\in \ZZ_{(S)}^{>0},\,\alpha\in \ZZ^S-\{0\}\\ \frac{(kf^2)^{1/\vartheta}}{X}\left\llbracket\frac{X}{kf^2}\star\alpha\right\rrbracket\gg X^c}}\frac{|\alpha^{(q)}|^{\frac12}}{k^{3/2}f^4}\left(\frac{(kf^2)^{1/\vartheta}}{X}\left\llbracket \frac{X}{kf^2}\star\alpha\right\rrbracket\right)^{-\frac32+\varepsilon} \left\llbracket \frac{X}{kf^2}\star\alpha\right\rrbracket^{-1+\varepsilon}\|G\|_{N+3,1}\\
\ll&X^{3-cN}\sum_{\substack{k,f\in \ZZ_{(S)}^{>0},\,\alpha\in \ZZ^S-\{0\}\\ \frac{(kf^2)^{1/\vartheta}}{X}\left\llbracket\frac{X}{kf^2}\star\alpha\right\rrbracket\gg X^c}}\frac{|\alpha^{(q)}|^{\frac12}}{k^{3/2+3/(2\vartheta)-2\varepsilon}f^{4+3/\vartheta-2\varepsilon}}\left\llbracket \frac{X}{kf^2}\star\alpha\right\rrbracket^{-\frac52+2\varepsilon}\|G\|_{N+3,1}.
\end{align*}
We have
\begin{equation}\label{eq:estimatealphaq}
|\alpha^{(q)}|=|\alpha|_\infty\prod_{i=1}^{r}|\alpha|_{q_i}\leq\frac{kf^2}{X}\left(1+ \left|\frac{X\alpha}{kf^2}\right|\right)\prod_{i=1}^{r}(1+|\alpha|_{q_i})=\frac{kf^2}{X}\left\llbracket \frac{X}{kf^2}\star\alpha\right\rrbracket.
\end{equation}
Hence by the above equation and \autoref{lem:estimatexilargesum} (1) we have
\begin{align*}
\eqref{eq:finalestimatei1kf1}
\ll&X^{\frac52-cN}\sum_{\substack{k,f\in \ZZ_{(S)}^{>0},\,\alpha\in \ZZ^S-\{0\}\\ \frac{(kf^2)^2}{X}\left\llbracket\frac{X}{kf^2}\star\alpha\right\rrbracket\gg X^c}}\frac{1}{k^{1+3/(2\vartheta)-2\varepsilon}f^{3+3/\vartheta-2\varepsilon}}\left\llbracket \frac{X}{kf^2}\star\alpha\right\rrbracket^{-2+2\varepsilon}\|G\|_{N+3,1}\\
\ll&X^{\frac52-cN}\sum_{k,f\in \ZZ_{(S)}^{>0}}\frac{1}{k^{1+3/(2\vartheta)-2\varepsilon}f^{3+3/\vartheta-2\varepsilon}} \frac{(kf^2)^{2+3\varepsilon}}{X^{2+3\varepsilon}}\|G\|_{N+3,1}\ll X^{\frac12-cN}\|G\|_{N+3,1}
\end{align*}
if $\vartheta$ is sufficiently close to $1/2$.

Next we consider \eqref{eq:finalestimatei1kf2}. By \autoref{prop:estimatei1kf} (1) with $N=0$ we obtain
\begin{align*}
\eqref{eq:finalestimatei1kf2}
\ll&X^{\frac32}\sum_{\substack{k,f\in \ZZ_{(S)}^{>0},\,\alpha\in \ZZ^S-\{0\}\\ \frac{(kf^2)^{1/\vartheta}}{X}\left\llbracket\frac{X}{kf^2}\star\alpha\right\rrbracket\ll X^c}}\frac{|\alpha^{(q)}|^{\frac12}}{k^{3/2}f^4}\left(\frac{(kf^2)^{1/\vartheta}}{X}\left\llbracket \frac{X}{kf^2}\star\alpha\right\rrbracket\right)^{\frac12+\varepsilon} \left\llbracket \frac{X}{kf^2}\star\alpha\right\rrbracket^{-1+\varepsilon}\|G\|_{1,1}\\
\ll&X\sum_{\substack{k,f\in \ZZ_{(S)}^{>0},\,\alpha\in \ZZ^S-\{0\}\\ \frac{(kf^2)^{1/\vartheta}}{X}\left\llbracket\frac{X}{kf^2}\star\alpha\right\rrbracket\ll X^c}}\frac{|\alpha^{(q)}|^{\frac12}}{k^{3/2-1/(2\vartheta)-\varepsilon}f^{4-1/\vartheta-2\varepsilon}}\left\llbracket \frac{X}{kf^2}\star\alpha\right\rrbracket^{-\frac12+\varepsilon}\|G\|_{1,1}.
\end{align*}
By \eqref{eq:estimatealphaq} and the condition that 
\[
\frac{(kf^2)^{1/\vartheta}}{X}\left\llbracket\frac{X}{kf^2}\star\alpha\right\rrbracket\ll X^c
\]
we have 
\[
|\alpha^{(q)}|\ll \frac{X^{1+c}}{(kf^2)^{\frac{1}{\vartheta}-1}}.
\]
Also, we have $k\ll X^{\vartheta(1+c)}/f^2$.
Hence by \autoref{lem:estimatexismallsum} we obtain
\begin{align*}
\eqref{eq:finalestimatei1kf2}
\ll&X^{1+\frac{c}{2}}\sum_{\substack{k,f\in \ZZ_{(S)}^{>0},\,\alpha\in \ZZ^S-\{0\}\\ \frac{(kf^2)^{1/\vartheta}}{X}\left\llbracket\frac{X}{kf^2}\star\alpha\right\rrbracket\ll X^c}}\frac{1}{k^{1-\varepsilon}f^{3-2\varepsilon}}\left\llbracket \frac{X}{kf^2}\star\alpha\right\rrbracket^{-\frac12+\varepsilon}\|G\|_{1,1}\\
\ll&X^{1+\frac{c}{2}}\sum_{\substack{k,f\in \ZZ_{(S)}^{>0}\\ k\ll X^{\vartheta(1+c)}/f^2}}\frac{1}{k^{1-\varepsilon}f^{3-2\varepsilon}}\frac{kf^2}{X} \left(\frac{X^{1+c}}{(kf^2)^{1/\vartheta}}\right)^{\frac12+2\vartheta\varepsilon}\|G\|_{1,1}\\
\ll&X^{\frac{c}{2}+\frac{c+1}{2}}\sum_{\substack{k,f\in \ZZ_{(S)}^{>0}\\ k\ll X^{\vartheta(1+c)}/f^2}}\frac{1}{k^{1/2+1/(2\vartheta)+\varepsilon}f^{1+1/\vartheta+2\varepsilon}} \|G\|_{1,1}.
\end{align*}
If $\vartheta\leq 1/2$, then $1/2+1/(2\vartheta)\geq 1$. Hence we directly get
\[
\eqref{eq:finalestimatei1kf2}\ll X^{\frac{1}{2}+c}\|G\|_{1,1}.
\]
Now we assume that $\vartheta>1/2$. In this case we have
\[
\eqref{eq:finalestimatei1kf2}\ll \sum_{f\in \ZZ_{(S)}^{>0}}\frac{1}{f^{1+1/\vartheta+2\varepsilon}} \legendresymbol{X^{\vartheta(1+c)}}{f^2}^{\frac{1}{2\vartheta}-\frac12} X^{\frac{1}{2}+c}\|G\|_{1,\infty}\ll X^{\frac{1}{2}+c+\varepsilon}\|G\|_{1,1}
\]
if $\vartheta$ is sufficiently close to $1/2$.
The conclusion now holds by combining the bounds of \eqref{eq:finalestimatei1kf1} and \eqref{eq:finalestimatei1kf2} together.
\end{proof}

Similarly, by using \autoref{prop:estimatei2kf} (1), one can prove that
\begin{proposition}\label{prop:finalestimateikfkappa02}
For any $\varepsilon>0$, $c\in \lopen 0,1\ropen$, and $N\in \ZZ_{\geq 0}$, there exists $\varepsilon'>0$ such that if $|\vartheta-1/2|<\varepsilon'$, then
\[
 \sum_{k,f\in \ZZ_{(S)}^{>0}}\sum_{\alpha\in \ZZ^S-\{0\}}\frac{|\alpha^{(q)}|^{\frac12}}{k^{3/2}f^4}|\mathbf{I}^2_{k,f}(\alpha,0)|\ll X^{\frac12-cN}\|G\|_{N+3,\infty}+X^{\frac12+c+\varepsilon}\|G\|_{1,1},
\]
where the implied constant only depends on $f_\infty$, $f_{q_i}$, $\varepsilon$, $c$, and $N$.
\end{proposition}
Hence we obtain
\begin{corollary}\label{cor:finalestimateikfkappa0}
For any $\varepsilon>0$, $c\in \lopen 0,1\ropen$, and $N\in \ZZ_{\geq 0}$, there exists $\varepsilon'>0$ such that if $|\vartheta-1/2|<\varepsilon'$, then
\[
S_G^{\alpha\neq 0,\kappa=0}(X)\ll X^{\frac12-cN}\|G\|_{N+3,\infty}+X^{\frac12+c+\varepsilon}\|G\|_{1,1},
\]
where the implied constant only depends on $f_\infty$, $f_{q_i}$, $\varepsilon$, $c$, and $N$.
\end{corollary}

\subsubsection{Bound of $\kappa\neq 0$} 
Finally, we bound the $\kappa\neq 0$ term. We have
\begin{equation}\label{eq:estimatekappanot0}
S_G^{\alpha\neq 0,\kappa\neq 0}(X)\ll \sum_{k,f\in \ZZ_{(S)}^{>0}}\sum_{\alpha\in \ZZ^S-\{0\}}\frac{|\alpha^{(q)}|^{\frac12}}{k^{3/2}f^4}\sum_{\kappa\in \ZZ^S-\{0\}}|\mathbf{I}_{k,f}(\alpha,\kappa)|.
\end{equation}
By \autoref{lem:integralvanish}, $\mathbf{I}_{k,f}(\alpha,\kappa)=0$ if $v_{q_i}(\kappa)< -L_i-v_{q_i}(\alpha)$ for some $i$. Hence we obtain
\[
  S_G^{\alpha\neq 0,\kappa\neq 0}(X)\ll  \sum_{k,f\in \ZZ_{(S)}^{>0}}\sum_{\alpha\in \ZZ^S-\{0\}}\frac{|\alpha^{(q)}|^{\frac12}}{k^{3/2}f^4}\sum_{\substack{\kappa\in \ZZ^S-\{0\}\\ v_{q_i}(\kappa)\geq-L_i-v_{q_i}(\alpha)\ \forall i}}(|\mathbf{I}^1_{k,f}(\alpha,\kappa)|+|\mathbf{I}^2_{k,f}(\alpha,\kappa)|) 
\]
\begin{lemma}\label{lem:kappanot0sum}
For any $N\in \ZZ_{\geq 2}$, we have
\[
\sum_{\substack{\kappa\in \ZZ^S-\{0\}\\ v_{q_i}(\kappa)\geq-L_i-v_{q_i}(\alpha)\ \forall i}}\frac{1}{|\kappa|^N}\ll |\alpha_{(q)}|^N.
\]
where the implied constant only depends on $L_1,\dots,L_r$.
\end{lemma}
\begin{proof}
This is because
\[
 \sum_{\substack{\kappa\in \ZZ^S-\{0\}\\ v_{q_i}(\kappa)\geq-L_i-v_{q_i}(\alpha)\ \forall i}}\frac{1}{|\kappa|^N}=\sum_{\kappa\in c\ZZ-\{0\}}\frac{1}{|\kappa|^N}\ll c^{-N},
\]
where
\[
c=\prod_{i=1}^{r}q_i^{-L_i-v_{q_i}(\alpha)}\asymp |\alpha_{(q)}|^{-1}.\qedhere
\]
\end{proof}
\begin{proposition}\label{prop:finalestimateikfkappanot0}
There exists $\varepsilon'>0$ such that
for any $K\in \ZZ_{\geq 0}$ and $|\vartheta-1/2|<\varepsilon'$, we have
\[
\sum_{k,f\in \ZZ_{(S)}^{>0}}\sum_{\alpha\in \ZZ^S-\{0\}}\frac{|\alpha^{(q)}|^{\frac12}}{k^{3/2}f^4}\sum_{\substack{\kappa\in \ZZ^S-\{0\}\\ v_{q_i}(\kappa)\geq-L_i-v_{q_i}(\alpha)\ \forall i}}|\mathbf{I}^1_{k,f}(\alpha,\kappa)|\ll X^{\frac12-\frac{K}{2}}\|G\|_{K+3,1},
\]
where the implied constant only depends on $f_\infty$, $f_{q_i}$, $\varepsilon$ and $K$.
\end{proposition}
\begin{proof}
Denote the left hand side by $S_G^{\alpha\neq 0,\kappa\neq 0,1}(X)$.

By \autoref{prop:estimatei1kf} (2) with $N=2$, we have
\begin{align*}
S_G^{\alpha\neq 0,\kappa\neq 0,1}(X)\ll&\sum_{k,f\in \ZZ_{(S)}^{>0}}\sum_{\alpha\in \ZZ^S-\{0\}}\frac{|\alpha^{(q)}|^{\frac12}}{k^{3/2}f^4}\sum_{\substack{\kappa\in \ZZ^S-\{0\}\\ v_{q_i}(\kappa)\geq-L_i-v_{q_i}(\alpha)\ \forall i}}X^{\frac32}\legendresymbol{\gcd(\alpha,kf^2)}{\sqrt{X}|\alpha\kappa|}^{K}\\
\times&\left(\frac{(kf^2)^{1/\vartheta}}{X}\left\llbracket \frac{X}{kf^2}\star\alpha\right\rrbracket\right)^{-\frac32+\varepsilon} \left\llbracket\frac{X}{kf^2}\star\alpha\right\rrbracket^{-1+\varepsilon} \|G\|_{K+3,1}.
\end{align*}
By \autoref{lem:kappanot0sum}, we obtain
\begin{align*}
&S_G^{\alpha\neq 0,\kappa\neq 0,1}(X)\ll X^{\frac32}\sum_{k,f\in \ZZ_{(S)}^{>0}}\sum_{\alpha\in \ZZ^S-\{0\}}\frac{|\alpha^{(q)}|^{\frac12}}{k^{3/2}f^4}\legendresymbol{\gcd(\alpha,kf^2) |\alpha_{(q)}|}{\sqrt{X}|\alpha|}^{K}\\
\times&\left(\frac{(kf^2)^{1/\vartheta}}{X}\left\llbracket \frac{X}{kf^2}\star\alpha\right\rrbracket\right)^{-\frac32+\varepsilon} \left\llbracket\frac{X}{kf^2}\star\alpha\right\rrbracket^{-1+\varepsilon} \|G\|_{K+3,1}.
\end{align*}
Since $\gcd(\alpha,kf^2)\leq |\alpha^{(q)}|$, we obtain
\[
\frac{\gcd(\alpha,kf^2) |\alpha_{(q)}|}{\sqrt{X}|\alpha|}= \frac{\gcd(\alpha,kf^2)}{\sqrt{X}|\alpha^{(q)}|} \leq \frac{1}{\sqrt{X}}.
\]
Hence 
\begin{align*}
&S_G^{\alpha\neq 0,\kappa\neq 0,1}(X)\ll X^{3-\frac{K}{2}}\sum_{k,f\in \ZZ_{(S)}^{>0}}\sum_{\alpha\in \ZZ^S-\{0\}}\frac{|\alpha^{(q)}|^{\frac12}}{k^{3/2+3/(2\vartheta)-2\varepsilon}f^{3+3/\vartheta-2\varepsilon}}\left\llbracket \frac{X}{kf^2}\star\alpha\right\rrbracket^{-\frac52+\varepsilon} \|G\|_{K+3,1}.
\end{align*}
By \eqref{eq:estimatealphaq} and \autoref{lem:estimatexilargesum} (1) we have
\begin{align*}
S_G^{\alpha\neq 0,\kappa\neq 0,1}(X)
\ll&X^{\frac52-\frac{K}{2}}\sum_{k,f\in \ZZ_{(S)}^{>0}}\sum_{\alpha\in \ZZ^S-\{0\}}\frac{1}{k^{1+3/(2\vartheta)-2\varepsilon}f^{3+3/\vartheta-2\varepsilon}}\left\llbracket \frac{X}{kf^2}\star\alpha\right\rrbracket^{-2+2\varepsilon}\|G\|_{K+3,1}\\
\ll&X^{\frac52-\frac{K}{2}}\sum_{k,f\in \ZZ_{(S)}^{>0}}\frac{1}{k^{1+3/(2\vartheta)-2\varepsilon}f^{3+3/\vartheta-2\varepsilon}} \frac{(kf^2)^{2+3\varepsilon}}{X^{2+3\varepsilon}}\|G\|_{K+3,\infty}\ll X^{\frac12-\frac{K}{2}}\|G\|_{K+3,1}
\end{align*}
if $\vartheta$ is sufficiently close to $1/2$.
\end{proof}
Similarly, by using \autoref{prop:estimatei2kf} (2), one can prove that
\begin{proposition}\label{prop:finalestimateikfkappanot02}
There exists $\varepsilon'>0$ such that
for any $K\in \ZZ_{\geq 0}$ and $|\vartheta-1/2|<\varepsilon'$, we have
\[
\sum_{k,f\in \ZZ_{(S)}^{>0}}\sum_{\alpha\in \ZZ^S-\{0\}}\frac{|\alpha^{(q)}|^{\frac12}}{k^{3/2}f^4}\sum_{\substack{\kappa\in \ZZ^S-\{0\}\\ v_{q_i}(\kappa)\geq-L_i-v_{q_i}(\alpha)\ \forall i}}|\mathbf{I}^2_{k,f}(\alpha,\kappa)|\ll X^{\frac12-\frac{K}{2}}\|G\|_{K+3,1},
\]
where the implied constant only depends on $f_\infty$, $f_{q_i}$, $\varepsilon$ and $K$.
\end{proposition}
Hence we obtain
\begin{corollary}\label{cor:finalestimateikfkappanot0}
There exists $\varepsilon'>0$ such that
for any $K\in \ZZ_{\geq 0}$ and $|\vartheta-1/2|<\varepsilon'$, we have
\[
S_G^{\alpha\neq 0,\kappa\neq 0}(X)\ll X^{\frac12-\frac{K}{2}}\|G\|_{K+3,1},
\]
where the implied constant only depends on $f_\infty$, $f_{q_i}$, $\varepsilon$ and $K$.
\end{corollary}

Combining \autoref{cor:finalestimateikfkappa0} and \autoref{cor:finalestimateikfkappanot0}, we obtain the final result in this section:
\begin{theorem}\label{thm:contributealphanot0}
For any $\varepsilon>0$, $c\in \lopen 0,1\ropen$, and $N,K\in \ZZ_{\geq 0}$, there exists $\varepsilon'>0$ such that we have
\[
S_G^{\alpha\neq 0}(X)\ll X^{\frac12-cN}\|G\|_{N+3,1}+ X^{\frac12-\frac{K}{2}}\|G\|_{K+3,1}+X^{\frac12+c+\varepsilon}\|G\|_{1,1}
\]
for any $|\vartheta-1/2|<\varepsilon'$,
where the implied constant only depends on $f_\infty$, $f_{q_i}$, $\varepsilon$, $c$, $N$ and $K$.
\end{theorem}

\section{Comparison of $S(X)$ and $S_G(X)$}\label{sec:sgxandsx}
Combining \autoref{thm:contributealpha0} and \autoref{thm:contributealphanot0}, we obtain
\begin{theorem}\label{thm:asymptoticsg}
For any $\varepsilon>0$, $c\in \lopen 0,1\ropen$, $A>0$, and $N,K\in \ZZ_{\geq 0}$, there exists $\varepsilon'>0$ such that if $|\vartheta-1/2|<\varepsilon'$, then
\begin{align*}
  S_G(X)=&\int_{\RR}\int_{\QQ_{S_\fin}} \widetilde{G}\legendresymbol{3}{2}|1-x|_\infty^{-\frac32} |1-y|_q^{-\frac32}\widehat{\Theta}_\infty(x)\widehat{\Theta}_{q}(y)\zeta^S(2)\rmd x\rmd yX^{\frac32}+\mf{M}(G)\\
+&O(X^{-A}\|G\|_{\lfloor \frac{15}{2}+2A\rfloor+2,1}+X^{\frac12+\varepsilon}\|G\|_1+ X^{\frac12-cN}\|G\|_{N+3,\infty}+ X^{\frac12-\frac{K}{2}}\|G\|_{K+3,\infty}+X^{\frac12+c+\varepsilon}\|G\|_{1,\infty}),
\end{align*}
where the implied constant only depends on $f_\infty$, $f_{q_i}$, $\varepsilon$, $\vartheta$, $c$, $N$ and $K$.
\end{theorem}
We will see that the first term in the asymptotic formula comes from the nontempered part. Recall that in \cite{cheng2025} we have proved that
\[
\Sigma^n(\xi)=\sum_{\mu}\Tr(\mu(f^n))-\frac{1}{2}\sum_{\mu}\Tr((\xi_0\otimes\mu)(f^n))+\Sigma^n(0) +\Sigma^n(\xi\neq 0).
\]
Now we write
\[
\Sigma_1^n(\xi)=\sum_{\mu}\Tr(\mu(f^n))\quad\text{and}\quad \Sigma_\temp^n(\xi)=-\frac{1}{2}\sum_{\mu}\Tr((\xi_0\otimes\mu)(f^n))+\Sigma^n(0)+\Sigma^n(\xi\neq 0),
\]
so that $\Sigma^n(\xi)=\Sigma_1^n(\xi)+\Sigma_\temp^n(\xi)$. Moreover, we have
\[
\Sigma_\temp^n(\xi)\ll n^{1/4+\varepsilon}
\]
by the main result Theorem 1.1 of \cite{cheng2025b}. 

Correspondingly we write
\[
S_1(X)=\sum_{\substack{n<X\\\gcd(n,S)=1}}\Sigma_1^n(\xi),\quad S_\temp(X)=\sum_{\substack{n<X\\\gcd(n,S)=1}}\Sigma_\temp^n(\xi),
\]
\[
S_{G,1}(X)=\sum_{\substack{n=1\\\gcd(n,S)=1}}^{+\infty}G\legendresymbol{n}{X}\Sigma_1^n(\xi),\quad\text{and} \quad S_{G,\temp}(X)=\sum_{\substack{n=1\\\gcd(n,S)=1}}^{+\infty}G\legendresymbol{n}{X}\Sigma_\temp^n(\xi).
\]

\subsection{Computation of the nontempered part}
In this subsection we will give asymptotic formulas for $S_{G,1}(X)$ and show that the main term is precisely the first term in the asymptotic formula of $S_G(X)$. 

By Proposition 7.2 and Corollary 7.5 of \cite{cheng2025}, the sum over $\mu$ is finite and is independent of $n$, and each $\mu$ is unramified outside $S$. Now we only need to consider such $\mu$.

\begin{proposition}\label{prop:unramifiedtrace}
Suppose that $\mu$ is unramified outside $S$. Then we have
\[
\Tr(\mu(f^n))=\prod_{v\in S}\Tr(\mu_v(f_v))\frac{\bm{\sigma}(n)\mu(n)}{\sqrt{n}},
\]
where $\bm{\sigma}(n)$ denotes the sum of the divisors of $n$. 
\end{proposition}
\begin{proof}
By Proposition 7.2 of \cite{cheng2025} we have
\begin{align*}
  \Tr(\mu(f^n)) & =\prod_{v\in S}\Tr(\mu_v(f_v))\prod_{p\notin S}p^{-n_p/2}\mu_p(p^{n_p})p^{n_p}\frac{1-p^{-n_p-1}}{1-p^{-1}}\\
  &=\prod_{v\in S}\Tr(\mu_v(f_v))\prod_{p\notin S}p^{-n_p/2}\mu_p(p^{n_p})\sum_{m=0}^{n_p}p^m=\prod_{v\in S}\Tr(\mu_v(f_v))\frac{\bm{\sigma}(n)\mu(n)}{\sqrt{n}}.\qedhere
\end{align*}
\end{proof}

\begin{proposition}\label{prop:sigmanestimateg}
Let $\chi$ be a Dirichlet character such that $\chi(m)=0$ if and only if $\gcd(m,S)\neq 1$. Let $G\in C_c^\infty(\lopen 1/4,5/4\ropen)$. Then for any $A>0$,
\[
\sum_{n=1}^{+\infty}G\legendresymbol{n}{X}\frac{\bm{\sigma}(n)}{\sqrt{n}}\chi(n)= \delta(\chi)\prod_{i=1}^{r}(1-q_i^{-1})\widetilde{G}\legendresymbol{3}{2}\zeta^S(2)X^{3/2}+O(X^{1/2}\|G\|_\infty +X^{-A}\|G\|_{\lfloor 2A+1\rfloor,1}),
\]
where the implied constant only depends on $A$ and $\chi$.
\end{proposition}
\begin{proof}
By Mellin inversion formula we have
\[
\sum_{n=1}^{+\infty}G\legendresymbol{n}{X}\frac{\bm{\sigma}(n)}{\sqrt{n}}\chi(n)=\frac{1}{\dpii} \int_{(\tau)}\widetilde{G}(u)\sum_{n=1}^{+\infty} \frac{\bm{\sigma}(n)\chi(n)}{n^{u+1/2}}X^u\rmd u
\]
for $\tau$ sufficiently large. 

The Dirichlet series inside can be computed explicitly. We have
\[
\sum_{n=1}^{+\infty}\frac{\bm{\sigma}(n)\chi(n)}{n^{u+1/2}}=\prod_{p\notin S}\sum_{m=0}^{+\infty}  \frac{\bm\sigma(p^m)\chi(p^m)}{p^{m(u+1/2)}}
\]
with
\begin{align*}
\sum_{m=0}^{+\infty}  \frac{\sigma(p^m)\chi(p^m)}{p^{m(u+1/2)}}&= \frac{1}{p-1}\sum_{m=0}^{+\infty}\frac{(p^{m+1}-1)\chi(p)^m}{p^{m(u+1/2)}}=\frac{1}{p-1}\sum_{m=0}^{+\infty}\left(p\legendresymbol{\chi(p)}{p^{u-1/2}}^m -\legendresymbol{\chi(p)}{p^{u+1/2}}^m\right) \\
&=\frac{1}{p-1}\left(\frac{p}{1-\chi(p)p^{-u+1/2}}-\frac{1}{1-\chi(p)p^{-u-1/2}}\right)\\
&= \frac{1}{(1-\chi(p)p^{-u+1/2})(1-\chi(p)p^{-u-1/2})}.
\end{align*}
Hence we obtain
\[
\sum_{n=1}^{+\infty}\frac{\bm{\sigma}(n)}{n^{u+1/2}}\chi(n)=L^S(u-1/2,\chi) L^S(u+1/2,\chi).
\]
Therefore
\[
\sum_{n=1}^{+\infty}G\legendresymbol{n}{X}\frac{\bm{\sigma}(n)}{\sqrt{n}}\chi(n)=\frac{1}{\dpii} \int_{(\tau)}\widetilde{G}(u)L^S(u-1/2,\chi) L^S(u+1/2,\chi)X^u\rmd u.
\]
Now we move the contour from $(\tau)$ to $(-A)$. The only poles come from $\chi=\triv$ and $u=1/2$ and $u=3/2$ with
\[
\res_{u=3/2}\widetilde{G}(u)L^S(u-1/2,\chi) L^S(u+1/2,\chi)X^u= \prod_{i=1}^{r}(1-q_i^{-1})\widetilde{G}\legendresymbol{3}{2}\zeta^S(2)X^{\frac32}
\]
and
\[
\res_{u=1/2}\widetilde{G}(u)L^S(u-1/2,\chi) L^S(u+1/2,\chi)X^u=\prod_{i=1}^{r}(1-q_i^{-1}) \widetilde{G}\legendresymbol{1}{2}\zeta^S(0)X^{\frac12}.
\]
Hence we obtain
\begin{align*}
  \sum_{n=1}^{+\infty}G\legendresymbol{n}{X}\frac{\bm{\sigma}(n)}{\sqrt{n}}\chi(n) & =\frac{1}{\dpii} \int_{(-A)}\widetilde{G}(u)L^S(u-1/2,\chi) L^S(u+1/2,\chi)X^u\rmd u \\
   & +\delta(\chi)\prod_{i=1}^{r}(1-q_i^{-1})\widetilde{G}\legendresymbol{3}{2}\zeta^S(2)X^{\frac32} +\delta(\chi)\prod_{i=1}^{r}(1-q_i^{-1})\widetilde{G}\legendresymbol{1}{2}\zeta^S(0)X^{\frac12}.
\end{align*}

By \eqref{eq:boundlvertical}, the first term above can be bounded by
\begin{align*}
&\int_{(-A)}|\widetilde{G}(u)||L^S(u-1/2,\chi)||L^S(u+1/2,\chi)|X^{-A}\rmd |u|\\
\ll_A& X^{-A}\int_{(-A)}|\widetilde{G}(u)|(1+|t|)^{\frac{1}{2}+A+\frac{1}{2}}|(1+|t|)^{\frac{1}{2}+A-\frac{1}{2}}\rmd |u|\\
\ll_A& X^{-A}\|G\|_{M_{-A}^{2A+1}}\ll_A X^{-A}\|G\|_{\lfloor 2A+1\rfloor+2,1},
\end{align*}
where in the last step we used \autoref{cor:mellinnorm}. Since
\[
\widetilde{G}\legendresymbol{1}{2}=\int_{1/4}^{5/4}G(x)x^{-1/2}\rmd x\ll \|G\|_\infty,
\]
we obtain our result.
\end{proof}

By \autoref{prop:unramifiedtrace} and \autoref{prop:sigmanestimateg} we obtain
\begin{corollary}\label{cor:sgnontemper}
For any $A>0$,
\[
S_{G,1}(X)= \prod_{v\in S}\Tr(\triv(f_v))\prod_{i=1}^{r}(1-q_i^{-1})\widetilde{G}\legendresymbol{3}{2}\zeta^S(2)X^{3/2}+O(X^{1/2}\|G\|_\infty +X^{-A}\|G\|_{\lfloor 2A+1\rfloor,1}),
\]
where the implied constant only depends on $f_\infty$, $f_{q_i}$ and $A$.
\end{corollary}
Hence we only need to compute traces of the trivial representation.
\begin{proposition}\label{prop:archimedeantrace}
We have
\[
\Tr(\triv(f_\infty))=2\int_{\RR}\frac{\widehat{\Theta}_\infty(x)}{|1-x|^{3/2}}\rmd x.
\]
\end{proposition}
\begin{proof}
We have
\begin{equation}\label{eq:thetarelation}
\Theta_\infty^\pm(x)=\theta_\infty\left(\pm 1,\frac{1-x}{4}\right)=\begin{dcases}
                                                                     \theta_\infty^+\left(\pm \frac{1}{\sqrt{1-x}}\right), & x<1, \\
                                                                     \theta_\infty^-\left(\pm \frac{1}{\sqrt{x-1}}\right), & x>1.
                                                                   \end{dcases}
\end{equation}
Also, recall that $\widehat{\Theta}_\infty(x)=\Theta_\infty^+(x)+\Theta_\infty^-(x)$. Hence
\[
2\int_{\RR}\frac{\widehat{\Theta}_\infty(x)}{|1-x|^{3/2}}\rmd x=2\int_{-\infty}^{1}\sum_{\pm}\frac{\theta_\infty^+\left(\pm \frac{1}{\sqrt{1-x}}\right)}{(1-x)^{3/2}}\rmd x+2\int_{1}^{+\infty}\sum_{\pm}\frac{\theta_\infty^-\left(\pm \frac{1}{\sqrt{x-1}}\right)}{(x-1)^{3/2}}\rmd x.
\]
For the first term, we make the change of variable $1/\sqrt{1-x}\mapsto t$ so that $x=1-1/t^2$ and $\rmd x=2\rmd t/t^3$. Hence
\[
2\int_{-\infty}^{1}\sum_{\pm}\frac{\theta_\infty^+\left(\pm \frac{1}{\sqrt{1-x}}\right)}{(1-x)^{3/2}}\rmd x =2\sum_{\pm}\int_{0}^{+\infty}\theta_\infty^+(\pm t)t^3\frac{2\rmd t}{t^3}=4\int_{\RR}\theta_\infty^+(x)\rmd x.
\]
Similarly,
\[
2\int_{1}^{+\infty}\sum_{\pm}\frac{\theta_\infty^-\left(\pm \frac{1}{\sqrt{x-1}}\right)}{(x-1)^{3/2}}\rmd x= 4\int_{\RR}\theta_\infty^-(x)\rmd x.
\]
Hence the conclusion holds by Proposition 7.1 of \cite{cheng2025}.
\end{proof}
\begin{proposition}\label{prop:ramifiedtrace}
We have
\[
\Tr(\triv(f_p))=|2|_p(1-p^{-1})^{-1}\int_{\QQ_p}\frac{\widehat{\Theta}_p(y)}{|1-y|_p^{3/2}}\rmd y.
\]
\end{proposition}
\begin{proof}
In the proof we omit the subscript $p$ in the norms.
By the definition of $\widehat{\Theta}_p(y)$ we have
\[
\int_{\QQ_p}\frac{\widehat{\Theta}_p(y)}{|1-y|^{3/2}}\rmd y=\int_{\QQ_p^2}\theta_p\left(z,\frac{z^2(1-y)}{4}\right)\frac{\rmd y\rmd z}{|1-y|^{3/2}|z|}.
\]
Now we make change of variable $z\mapsto T$ and $z^2(1-y)/4\mapsto N$ so that
\[
y\mapsto 1-\frac{4N}{T^2}\quad\text{and}\quad z\mapsto T.
\]
We have
\[
\rmd y\wedge\rmd z=\left(\frac{8N}{T^3}\rmd T-\frac{4}{T^2}\rmd N\right)\wedge \rmd T=\frac{4}{T^2}\rmd T\wedge \rmd N.
\]
Hence we obtain
\[
\int_{\QQ_p}\frac{\widehat{\Theta}_p(y)}{|1-y|^{3/2}}\rmd y=\int_{\QQ_p^2}\theta_p(T,N) \frac{\rmd T\rmd N}{|4N/T^2|^{3/2}|T|}\left|\frac{4}{T^2}\right|=|2|^{-1}\int_{\QQ_p^2} \frac{1}{|N|^{3/2}}\theta_{p}(T,N)\rmd T\rmd N.
\]
By Proposition 7.3 of \cite{cheng2025} we obtain the desired result.
\end{proof}

Combining \autoref{cor:sgnontemper}, \autoref{prop:archimedeantrace}, \autoref{prop:ramifiedtrace} and recall that $2\in S$, we obtain
\begin{theorem}\label{thm:asymptoticsgnontemper}
For any $A>0$,
\begin{align*}
  S_{G,1}(X) & = \int_{\RR}\int_{\QQ_{S_\fin}} \widetilde{G}\legendresymbol{3}{2}|1-x|_\infty^{-\frac32} |1-y|_q^{-\frac32}\widehat{\Theta}_\infty(x)\widehat{\Theta}_{q}(y)\zeta^S(2)\rmd x\rmd yX^{\frac32} \\
   & +O(X^{1/2}\|G\|_\infty +X^{-A}\|G\|_{\lfloor 2A+1\rfloor+2,1}),
\end{align*}
where the implied constant only depends on $f_\infty$, $f_{q_i}$ and $A$.
\end{theorem}
Note that the main term is precisely the first term in \autoref{thm:asymptoticsg}.

Since $S_G(X)=S_{G,1}(X)+S_{G,\temp}(X)$, we obtain an asymptotic formula for $S_{G,\temp}(X)$:
\begin{theorem}\label{thm:asymptoticsgtemper}
For any $\varepsilon>0$, $c\in \lopen 0,1\ropen$, $A>0$, and $N,K\in \ZZ_{\geq 0}$,  there exists $\varepsilon'>0$ such that if $|\vartheta-1/2|<\varepsilon'$, then
\[
S_{G,\temp}(X)=\mf{M}(G)+\mf{R}(G),
\]
where the error term $\mf{R}(G)$ is bounded by
\[
X^{\frac12}\|G\|_\infty+X^{-A}\|G\|_{\lfloor \frac{15}{2}+2A\rfloor+2,1}+X^{\frac12+\varepsilon}\|G\|_1+ X^{\frac12-cN}\|G\|_{N+3,1}+ X^{\frac12-\frac{K}{2}}\|G\|_{K+3,1}+X^{\frac12+c+\varepsilon}\|G\|_{1,1},
\]
and the implied constant only depends on $\vartheta$, $f_\infty$, $f_{q_i}$, $\varepsilon$, $\varepsilon'$, $c$, $N$ and $K$.
\end{theorem}

At the end of the subsection we give an asymptotic formula for $S_1(X)$.

\begin{proposition}\label{prop:sigmanestimate}
Let $\chi$ be a Dirichlet character such that $\chi(m)=0$ if and only if $\gcd(m,S)\neq 1$. Then
\[
\sum_{n<X}\frac{\bm{\sigma}(n)}{\sqrt{n}}\chi(n)= \delta(\chi)\prod_{i=1}^{r}(1-q_i^{-1})\frac{2}{3}\zeta^S(2)X^{3/2}+O(X^{1/2}\log X),
\]
where $\delta(\chi)=1$ if $\chi$ is trivial, and $\delta(\chi)=0$ otherwise. The implied constant only depends on $\chi$.
\end{proposition}
\begin{proof}
We have
\[
\sum_{n<X}\frac{\bm{\sigma}(n)}{\sqrt{n}}\chi(n)=\sum_{n<X}\sum_{d\mid n}\frac{d}{\sqrt{n}}\chi(n)= \sum_{ab<X}\frac{a}{\sqrt{ab}}\chi(ab)=\sum_{b<X}\frac{\chi(b)}{\sqrt{b}} \sum_{a<X/b}\chi(a)\sqrt{a}.
\]
By Abel summation formula,
\[
\sum_{a<X/b}\chi(a)\sqrt{a}= \frac{2}{3}\prod_{i=1}^{r}\left(1-\frac{1}{q_i}\right)\legendresymbol{X}{b}^{3/2}\delta(\chi)+O\left(\legendresymbol{X}{b}^{1/2}\right)
\]
Hence
\begin{align*}
\sum_{n<X}\frac{\bm{\sigma}(n)}{\sqrt{n}}\chi(n)&= \frac{2}{3}\prod_{i=1}^{r}\left(1-\frac{1}{q_i}\right)\sum_{b<X}\frac{\chi(b)}{b^2}X^{3/2}\delta(\chi) +\sum_{b<X}O\legendresymbol{1}{b}X^{1/2}\\
&=\delta(\chi)\prod_{i=1}^{r}\left(1-\frac{1}{q_i}\right)\frac{2}{3}\zeta^S(2)X^{3/2}+O(X^{1/2}\log X).\qedhere
\end{align*}
\end{proof}

\begin{theorem}\label{thm:contributenontempered}
We have the following asymptotic formula for $S_1(X)$:
\[
S_1(X)=\mf{A}X^{\frac32}+O(X^{1/2}\log X),
\]
where 
\[
\mf{A}=\frac{2}{3}\zeta^S(2)\int_{\QQ_S}\frac{\widehat{\Theta}_\infty(x) \widehat{\Theta}_q(y)}{|1-x|_\infty^{3/2}|1-y|_q^{3/2}}\rmd x\rmd y.
\]
\end{theorem}
\begin{proof}
Recall that
\[
S_1(X)=\sum_{\substack{n<X\\\gcd(n,S)=1}}\sum_{\mu}\Tr(\mu(f^n))
\]
and the sum over $\mu$ is finite.

By \autoref{prop:unramifiedtrace} and \autoref{prop:sigmanestimate} we obtain
\begin{align*}
\sum_{\substack{n<X\\\gcd(n,S)=1}}\Tr(\mu(f^n))=&\prod_{v\in S}\Tr(\mu_v(f_v))\sum_{\substack{n<X\\\gcd(n,S)=1}}\frac{\bm{\sigma}(n)\mu(n)}{\sqrt{n}}\\
=&\prod_{v\in S}\Tr(\mu_v(f_v))\delta(\mu)\prod_{i=1}^{r}(1-q_i^{-1})\frac{2}{3}\zeta^S(2)X^{3/2}+O(X^{1/2}\log X).
\end{align*}
Hence
\begin{align*}
S_1(X)
=&\prod_{v\in S}\Tr(\triv(f_v))\prod_{i=1}^{r}(1-q_i^{-1})\frac{2}{3}\zeta^S(2)X^{3/2}+O(X^{1/2}\log X).
\end{align*}
Now the conclusion follows from \autoref{prop:archimedeantrace}, \autoref{prop:ramifiedtrace} and that $2\in S$.
\end{proof}

\subsection{Comparison of tempered parts} In this subsection we will give an asymptotic formula for $S_\temp(X)$. To do this, we need to choose $G$ more specific.

Fix $\delta\in \lopen 0,1\ropen$. Let $\phi\in C_c^\infty(\lopen -1,1\ropen)$ such that $\int\phi=1$. For any $Y>1$, let $\phi_{Y^\delta}(x)=Y^{1-\delta}\phi(xY^{1-\delta})$ be the kernel function. For $\beta\in [1/2,1\ropen$, we define
\[
G_{Y^\delta,\beta}(x)=\triv_{[\beta,1]}* \phi_{Y^\delta}(x)=Y^{1-\delta}\int_{\RR}\triv_{[\beta,1]}(x-y) \phi(yY^{1-\delta})\rmd y.
\]
Thus we have
\begin{equation}\label{eq:propertyconvolution}
G_{Y^\delta,\beta}\legendresymbol{x}{Y}=\begin{cases}
                                          1, & \beta Y+Y^\delta\leq x\leq Y-Y^\delta \\
                                          O(1), & |x-\beta Y|<Y^\delta\ \text{or}\ |x-Y|<Y^\delta\\
                                          0,&\text{otherwise}.
                                        \end{cases}
\end{equation}
Since $G_{Y^\delta,\beta}^{(j)}=\triv_{[\beta,1]}* \phi_{Y^\delta}^{(j)}$, we have
\[
G_{Y^\delta,\beta}(x)=\int_{x-1}^{x-\beta}\phi_{Y^\delta}(y)\rmd y,
\]
we get
\[
  G'_{Y^\delta,\beta}(x)
  =\phi_{Y^\delta}(x-\beta)-\phi_{Y^\delta}(x-1).
\]
More generally, for \(j\ge 1\),
\[
  G_{Y^\delta,\beta}^{(j)}(x)
  =\phi_{Y^\delta}^{(j-1)}(x-\beta)
   -\phi_{Y^\delta}^{(j-1)}(x-1).
\]
Therefore
\[
  \|G_{Y^\delta,\beta}\|_\infty\ll 1,\qquad
  \|G_{Y^\delta,\beta}^{(j)}\|_\infty
  \ll_j Y^{j(1-\delta)}\quad (j\geq1),
\]
and
\[
  \|G_{Y^\delta,\beta}\|_1\ll 1,\qquad
  \|G_{Y^\delta,\beta}^{(j)}\|_1
  \ll_j Y^{(j-1)(1-\delta)}\quad (j\geq1).
\]
Hence for $M\geq 1$ we have $\|G_{Y^\delta,\beta}\|_{M,1}\ll Y^{(M-1)(1-\delta)}$ and $\|G_{Y^\delta,\beta}\|_{M,\infty}\ll Y^{M(1-\delta)}$.
The implied constants in the above two estimates only depend on $\phi$ and $\delta$.

Now we give an asymptotic formula for $\mf{M}(G_{X^\delta,\beta})$ and $\mf{R}(G_{X^\delta,\beta})$ which is independent of $G$.
\begin{lemma}\label{lem:valueg1}
We have $\widetilde{G}_{X^\delta,\beta}(1)=1-\beta$ and 
\[
\widetilde{G}_{X^\delta,\beta}'(1)=-1-\beta\log\beta+\beta+O(X^{\delta-1}),
\]
where the implied constant only depends on $\phi$ and $\delta$.
\end{lemma}
\begin{proof}
By definition of the Mellin transform we obtain
\[
\widetilde{G}(1)=\int_{0}^{+\infty}G(x)\rmd x\quad\text{and}\quad \widetilde{G}'(1)=\int_{0}^{+\infty}G(x)\log x\rmd x.
\]
Hence
\[
\widetilde{G}_{X^\delta,\beta}(1)=\int_{\RR}X^{1-\delta}\int_{\RR}\triv_{[\beta,1]}(x-y) \phi(yX^{1-\delta})\rmd y\rmd x=\int_{\RR}\triv_{[\beta,1]}\rmd x X^{1-\delta}\int_{\RR}\phi(yX^{1-\delta})\rmd y=1-\beta
\]
since $\int\phi=1$. For the derivative at $1$, by \eqref{eq:propertyconvolution} we have
\[
\int_{0}^{+\infty}G_{X^\delta,\beta}(x)\log x\rmd x-\int_{\beta}^{1}\log x\rmd x\ll \int_{\beta-X^{\delta-1}}^{\beta+X^{\delta-1}}1+\int_{1-X^{\delta-1}}^{1+X^{\delta-1}}1\ll X^{\delta-1}.
\]
The conclusion now follows since
\[
\int_{\beta}^{1}\log x\rmd x=\left.(x\log x-x)\right|_\beta^1=-1-\beta\log\beta+\beta.\qedhere
\]
\end{proof}

By the expression of $\mf{M}(G)$ in \autoref{thm:contributealpha0}, we obtain
\begin{corollary}\label{cor:estimatea}
We have
\[
\mf{M}(G_{X^\delta,\beta})=\mf{B}\vartheta'(-1-\beta\log\beta+\beta)X+\mf{B}\vartheta'(1-\beta) X\log X+(\mf{C}+\mf{D}\vartheta')(1-\beta) X+O(X^\delta),
\]
where the implied constant only depends on $f_\infty$, $f_{q_i}$, $\phi$ and $\delta$.
\end{corollary}

\begin{proposition}\label{cor:estimater}
Suppose that $\delta\in \lopen 1/2,1\ropen$ and $c>1-\delta$. Then for any $\varepsilon>0$, there exists $\varepsilon'>0$ such that if $|\vartheta-1/2|<\varepsilon'$, then
\[
\mf{R}(G_{X^\delta,\beta})\ll 
X^{\frac32-\delta+\varepsilon}+X^{\frac12+c+\varepsilon},
\]
where the implied constant only depends on $\vartheta$, $\varepsilon$, $f_\infty$, $f_{q_i}$, $c$ and $\delta$.
\end{proposition}
\begin{proof}
By the bounds of $\|G\|_{M,1}$ and $\|G\|_{M,\infty}$ and the bound of $\mf{R}$ in \autoref{thm:asymptoticsgtemper}, we know that $\mf{R}(G_{X^\delta,\beta})$ is bounded by
\[
X^{\frac12+1-\delta}+X^{-A+(\lfloor \frac{15}{2}+2A\rfloor+1)(1-\delta)}+ X^{\frac12-cN+(N+2)(1-\delta)}+ X^{\frac12-\frac{K}{2}+(K+2)(1-\delta)}+X^{\frac12+c+\varepsilon}.
\]
Since $\delta>1/2$ and $c>1-\delta$, we can take $A,N,K$ sufficiently large to obtain the result.
\end{proof}

By \autoref{cor:estimatea} and \autoref{cor:estimater} we obtain

\begin{theorem}\label{thm:asymptoticconvolutiong}
For any $\varepsilon>0$, $\delta\in \lopen 1/2,1\ropen$ and $c\in \lopen 0,1\ropen$ such that $c>1-\delta$, there exists $\varepsilon'>0$ such that if $|\vartheta-1/2|<\varepsilon'$, then
\begin{align*}
  S_{G_{X^\delta,\beta},\temp}(X) & =\mf{B}\vartheta'(-1-\beta\log\beta+\beta)X+\mf{B}\vartheta'(1-\beta) X\log X+(\mf{C}+\mf{D}\vartheta')(1-\beta) X \\
   & +O(X^\delta+X^{\frac32-\delta+\varepsilon}+X^{\frac12+c+\varepsilon}),
\end{align*}
where the implied constant only depends on $\vartheta$, $\varepsilon$, $f_\infty$, $f_{q_i}$, $c$ and $\delta$.
\end{theorem}

We end this section by giving an asymptotic formula for $S_\temp(X)$.
\begin{proposition}\label{prop:stempxdiffersgx}
Let $K=\lfloor \log_2 X\rfloor$. Then we have
\[
S_\temp(X)-\sum_{j=1}^{K}S_{G_{2^{j\delta},1/2},\temp}(2^j)-S_{G_{X^{\delta},2^K/X},\temp}(X)\ll X^{1/4+\delta+\varepsilon},
\]
where the implied constant only depends on $f_\infty$ and $f_{q_i}$.
\end{proposition}
\begin{proof}
We have
\[
S_\temp(X)=\sum_{j=1}^{K}S_\temp(2^{j-1},2^j)+S_\temp(2^K,X),
\]
where
\[
S_\temp(2^{j-1},2^j)=\sum_{\substack{2^{j-1}\leq n<2^j\\n\in \ZZ_{(S)}^{>0}}}\Sigma^n_\temp(\xi)
\]
and
\[
S_\temp(2^K,X)=\sum_{\substack{2^K\leq n<X\\n\in \ZZ_{(S)}^{>0}}}\Sigma^n_\temp(\xi).
\]
By the definition of $S_{G,\temp}(X)$ and \eqref{eq:propertyconvolution} we obtain
\[
S_{G_{2^{j\delta},1/2},\temp}(2^j)-S_\temp(2^{j-1},2^j)\ll \sum_{2^j-2^{j\delta}<n<2^j+2^{j\delta}}\Sigma^n_\temp(\xi)+\sum_{2^{j-1}-2^{(j-1)\delta}<n<2^{j-1}+2^{(j-1)\delta}} \Sigma^n_\temp(\xi).
\]
We have
\[
\Sigma^n_\temp(\xi)\ll_{f_\infty,f_{q_i},\varepsilon} n^{1/4+\varepsilon}
\]
since we have proved in \cite{cheng2025b} that all parts except for the elliptic part and the $1$-dimensional part are $\ll n^\varepsilon$ and $I_\cusp(f^n)\ll n^{1/4+\varepsilon}$.
Hence
\[
S_{G_{2^{j\delta},1/2},\temp}(2^j)-S_\temp(2^{j-1},2^j)\ll_{f_\infty,f_{q_i},\varepsilon} 2^{j(\delta+1/4+\varepsilon)}.
\]
Similarly we have
\[
S_{G_{X^{\delta},2^K/X},\temp}(X)-S_\temp(2^K,X)\ll X^{\delta+1/4+\varepsilon}.
\]
Therefore
\begin{align*}
S_\temp(X)-\sum_{j=1}^{K}S_{G_{2^{j\delta},1/2},\temp}(2^j)-S_{G_{X^{\delta},2^K/X},\temp}(X)\ll& \sum_{j=1}^{K}2^{j(\delta+1/4+\varepsilon)}+X^{\delta+1/4+\varepsilon}\\
\ll& (2^K)^{\delta+1/4+\varepsilon}+X^{\delta+1/4+\varepsilon}\ll X^{\delta+1/4+\varepsilon}.\qedhere
\end{align*}
\end{proof}

\begin{proposition}
Let $K=\lfloor \log_2 X\rfloor$. Then we have
\begin{align*}
&\sum_{j=1}^{K}S_{G_{2^{j\delta},1/2},\temp}(2^j) +S_{G_{X^{\delta},2^K/X},\temp}(X)\\
=&\mf{B}\vartheta'(X\log X-X)+(\mf{C}+\mf{D}\vartheta')X+O(X^{\delta}+X^{\frac32-\delta+\varepsilon}+ X^{\frac12+c+\varepsilon}).
\end{align*}
\end{proposition}
\begin{proof}
By \autoref{thm:asymptoticconvolutiong} we have
\begin{align*}
&\sum_{j=1}^{K}S_{G_{2^{j\delta},1/2},\temp}(2^j)+S_{G_{X^{\delta},2^K/X},\temp}(X)\\
=&\mf{B}\vartheta'\left[ \left(-\frac{1}{2}+\frac{1}{2}\log 2\right)\sum_{j=1}^{K}2^j+\left(-1-\frac{2^K}{X}\log\frac{2^K}{X}+\frac{2^K}{X}\right)X+ \frac{1}{2}\sum_{j=1}^{K}2^j\log(2^j)+\left(1-\frac{2^K}{X}\right)X\log X\right]\\
+&(\mf{C}+\mf{D}\vartheta')\left[\frac{1}{2}\sum_{j=1}^{K}2^j+\left(1-\frac{2^K}{X}\right)X\right]\\
+&O\left(\sum_{j=1}^{K}\left((2^j)^\delta+(2^j)^{\frac32-\delta+\varepsilon}+(2^j)^{\frac12+c+\varepsilon}\right) +X^{\delta}+X^{\frac32-\delta+\varepsilon}+X^{\frac12+c+\varepsilon}\right).
\end{align*}
Clearly we have
\[
\frac{1}{2}\sum_{j=1}^{K}2^j+\left(1-\frac{2^K}{X}\right)X=X+O(1).
\]
The coefficient of $\mf{B}\vartheta'$ can be computed as follows:
\begin{align*}
&\left(-\frac{1}{2}+\frac{1}{2}\log 2\right)\sum_{j=1}^{K}2^j+\left(-1-\frac{2^K}{X}\log\frac{2^K}{X}+\frac{2^K}{X}\right)X+\frac{1}{2}\sum_{j=1}^{K}2^j\log(2^j)+\left(1-\frac{2^K}{X}\right)X\log X\\
=&(-1+\log 2)(2^K-1)-X+2^K-2^K(K\log 2-\log X)+\log 2(1-2^K+K2^K)+X\log X-2^K\log X\\
=&X\log X-X+O(1).
\end{align*}
For the error term we have
\begin{align*}
&\sum_{j=1}^{K}\left((2^j)^\delta+(2^j)^{\frac32-\delta+\varepsilon}+ (2^j)^{\frac12+c+\varepsilon}\right) +X^{\delta}+X^{\frac32-\delta+\varepsilon}+X^{\frac12+c+2(1-\delta)}\\
\ll& (2^K)^\delta+(2^K)^{\frac32-\delta+\varepsilon}+(2^K)^{\frac12+c+\varepsilon}+ X^{\delta}+X^{\frac32-\delta+\varepsilon}+X^{\frac12+c+\varepsilon}\\
\ll&X^{\delta}+X^{\frac32-\delta+\varepsilon}+X^{\frac12+c+\varepsilon}
\end{align*}
since $2^K\ll X$.
\end{proof}

Finally, by choosing $\delta=5/8>1/2$, $\varepsilon$ replaced by $\varepsilon/2$, and $c=3/8+\varepsilon/2>1-\delta$ and use the above two propositions, we obtain
\begin{theorem}\label{thm:contributetempered}
For any $\varepsilon>0$, there exists $\varepsilon'>0$ such that if $|\vartheta-1/2|<\varepsilon'$, then
\[
S_\temp(X)= \mf{B}\vartheta'(X\log X-X)+(\mf{C}+\mf{D}\vartheta')X+O(X^{\frac78+\varepsilon}),
\] 
where the implied constant only depends on $\vartheta$, $\varepsilon$, $f_\infty$ and $f_{q_i}$.
\end{theorem}

Combining \autoref{thm:contributenontempered} and \autoref{thm:contributetempered} we obtain
\begin{theorem}\label{thm:contributeellipticpoisson}
For any $\varepsilon>0$, there exists $\varepsilon'>0$ such that if $|\vartheta-1/2|<\varepsilon'$, then
\[
S(X)=\mf{A}X^{\frac32}+\mf{B}\vartheta'(X\log X-X)+(\mf{C}+\mf{D}\vartheta')X+O(X^{\frac78+\varepsilon}),
\]
where the implied constant only depends on $\vartheta$, $\varepsilon$, $f_\infty$ and $f_{q_i}$.
\end{theorem}

\section{Contribution of the square term}
In this section we will derive a formula of $\Sigma^n(\square)$ that splits the parameter $\vartheta$, and then derive the main result. Recall the definition in \cite{cheng2025}:
\begin{equation}\label{eq:defsigmasquare}
\begin{split}
\Sigma^n(\square)=&2\sum_{\pm}\sum_{\nu\in \ZZ^r}\sum_{\substack{T\in \ZZ^S\\ T^2\mp 4nq^\nu= \square}}\sum_{f^2\mid T^2\mp 4nq^\nu}  \sum_{k\in \ZZ_{(S)}^{>0}}\frac{1}{kf}\legendresymbol{(T^2\mp 4nq^\nu)/f^2}{k}\theta_\infty^\pm\legendresymbol{T}{2n^{1/2}q^{\nu/2}}\\
    \times &\prod_{i=1}^{r}\theta_{q_i}^{\pm,\nu}(T)\left[F\legendresymbol{kf^2}{|T^2\mp 4nq^\nu|_{\infty,q}'^{\vartheta}}+\frac{kf^2}{\sqrt{|T^2\mp 4nq^\nu|_{\infty,q}'}}V\legendresymbol{kf^2}{|T^2\mp 4nq^\nu|_{\infty,q}'^{\vartheta'}}\right].
\end{split}
\end{equation}
Note that when $T^2\mp 4nq^\nu=0$, the contribution to $\Sigma^n(\square)$ is $0$. Hence we only need to consider the summation over $T\in \ZZ^S$ with $T^2\mp 4nq^\nu=\square$ and is not zero.

By \cite[Theorem 4.9]{cheng2025b} with $A=|T^2\mp 4nq^\nu|_{\infty,q}'^{\vartheta}$, we obtain
\begin{proposition}\label{cor:sigmasquarel}
We have
\begin{equation}\label{eq:sigmasquarel}
\begin{split}
   \Sigma^n(\square)= &2\sum_{\pm}\sum_{\nu\in \ZZ^r}\sum_{\substack{T\in \ZZ^S\\ T^2\mp 4nq^\nu= \square\\ T^2\mp 4nq^\nu\neq 0}}\theta_\infty^\pm\legendresymbol{T}{2n^{1/2}q^{\nu/2}}
    \prod_{i=1}^{r}\theta_{q_i}(T,\pm nq^\nu) \\
      \times& \res_{s=0}\widetilde{F}(s)L^{S}(1+s,T^2\mp 4nq^\nu)|T^2\mp 4nq^\nu|_{\infty,q}'^{\vartheta s}.
\end{split}
\end{equation}
\end{proposition}

\begin{lemma}\label{lem:sigmasquare}
For any sign $\pm$ and $\nu\in \ZZ^r$, we have a $2:1$ map
\begin{align*}
  \{(a,b)\in \ZZ^S\times \ZZ^S\,|\, a\neq b,\,ab=\pm nq^\nu\} & \to\{T\in \ZZ^S\,|\, T^2\mp 4nq^\nu=\square,\,T^2\mp 4nq^\nu\neq 0\} \\
  (a,b) &\mapsto a+b.
\end{align*}
\end{lemma}

\begin{proof}
This map is well defined since $(a+b)^2\mp 4nq^\nu=(a-b)^2$. Moreover, the inverse image of $T$ is precisely
\[
\left\{\left(\frac{T+\sqrt{T^2\mp 4nq^\nu}}{2},\frac{T-\sqrt{T^2\mp 4nq^\nu}}{2}\right), \left(\frac{T-\sqrt{T^2\mp 4nq^\nu}}{2},\frac{T+\sqrt{T^2\mp 4nq^\nu}}{2}\right)\right\}.\qedhere
\]
\end{proof}

Now suppose that $a,b\in \ZZ^S$ such that $a+b=T$ and $ab=\pm nq^\nu$.
By the definition of $\theta_\infty^\pm(x)$, we have
\[
\theta_\infty^\pm\legendresymbol{T}{2n^{1/2}q^{\nu/2}}=\theta_\infty\left(\frac{T}{2n^{1/2}q^{\nu/2}},\pm \frac14\right)=\theta_\infty(T,\pm nq^\nu)=\theta_\infty\begin{pmatrix}
                                                          a & 0 \\
                                                          0 & b 
                                                        \end{pmatrix}.
\]
Also we have
\[
\theta_{q_i}(T,\pm nq^\nu)=\theta_{q_i}\begin{pmatrix}
                                                          a & 0 \\
                                                          0 & b 
                                                        \end{pmatrix}.
\]
Hence by \eqref{eq:sigmasquarel} and \autoref{lem:sigmasquare} we obtain
\begin{corollary}\label{cor:sigmasquareexplicit}
We have
\[
\Sigma^n(\square)=\sum_{\pm}\sum_{\nu\in \ZZ^r}\sum_{\substack{a\neq b\in \ZZ^S\\ ab=\pm nq^\nu}}\theta_\infty\begin{pmatrix} a & 0 \\ 0 & b \end{pmatrix}\theta_q\begin{pmatrix}
 a & 0 \\ 0 & b\end{pmatrix} \res_{s=0}\widetilde{F}(s)L^{S}(1+s,(a-b)^2)|(a-b)^2|_{\infty,q}'^{\vartheta s}.
\]
\end{corollary}

Now we compute the residue at $s=0$ of the function
\begin{equation}\label{eq:residuesquare}
\widetilde{F}(s)L^{S}(1+s,\sigma^2)|\sigma^2|_{\infty,q}'^{\vartheta s}
\end{equation}
for $\sigma\in \ZZ^S$.

Clearly we have
\[
|\sigma^2|_{\infty,q}'^{\vartheta s}=|\sigma^2|_\infty^{\vartheta s}\prod_{i=1}^{r}|\sigma^2|_{q_i}'^{\vartheta s}= |\sigma|_\infty^{2\vartheta s}\prod_{i=1}^{r}|\sigma|_{q_i}'^{2\vartheta s}=|\sigma^{(q)}|^{2\vartheta s}.
\]

Next we consider the partial Zagier $L$-function. We have (cf. (4.4) in \cite{cheng2025b})
\begin{equation}\label{eq:zagierzeta}
   L^{S}(s,\sigma^2) =\prod_{i=1}^{r}\left(1-q_i^{-s}\right)\sum_{f\mid {\sigma^{(q)}}}\frac{1}{f^{2s-1}}\sum_{k\mid \frac{{\sigma^{(q)}}}{f}}\frac{\bm{\mu}(k)}{k^s}\zeta(s),
\end{equation}
which is an entire function times $\zeta(s)$, where $\bm{\mu}(k)$ denotes the M\"obius function. Hence it is regular away from $s=1$ and has a simple pole at $s=1$ with residue
\[
\res_{s=1}L^{S}(s,\sigma^2)=\prod_{i=1}^{r}(1-q_i^{-1})\sum_{f\mid {\sigma^{(q)}}}\frac{1}{f}\sum_{k\mid \frac{{\sigma^{(q)}}}{f}}\frac{\bm{\mu}(k)}{k}=\prod_{i=1}^{r}(1-q_i^{-1})
\]
and finite part
\begin{align*}
\fp_{s=1}L^{S}(s,\sigma^2)&=\left.\prod_{i=1}^{r}\left(1-q_i^{-s}\right)\sum_{f\mid {\sigma^{(q)}}}\frac{1}{f^{2s-1}}\sum_{k\mid \frac{{\sigma^{(q)}}}{f}}\frac{\bm{\mu}(k)}{k^s}\right|_{s=1}\fp_{s=1}\zeta(s)\\
&+\left.\frac{\rmd}{\rmd s}\right|_{s=1}\left(\prod_{i=1}^{r}\left(1-q_i^{-s}\right)\sum_{f\mid {\sigma^{(q)}}}\frac{1}{f^{2s-1}}\sum_{k\mid \frac{{\sigma^{(q)}}}{f}}\frac{\bm{\mu}(k)}{k^s}\right)\res_{s=1}\zeta(s).
\end{align*}
Since $\fp_{s=1}\zeta(s)=\upgamma$ is the Euler-Mascheroni constant \cite[25.2.4]{DLMF}, we obtain
\begin{align*}
\fp_{s=1}L^{S}(s,\sigma^2)&=\upgamma\prod_{i=1}^{r}\left(1-q_i^{-1}\right)\sum_{f\mid {\sigma^{(q)}}}\frac{1}{f}\sum_{k\mid \frac{{\sigma^{(q)}}}{f}}\frac{\bm{\mu}(k)}{k} +\prod_{i=1}^{r}\left(1-q_i^{-1}\right)\sum_{i=1}^{r}\frac{q_i^{-1}\log q_i}{1-q_i^{-1}}\sum_{f\mid {\sigma^{(q)}}}\frac{1}{f}\sum_{k\mid \frac{{\sigma^{(q)}}}{f}}\frac{\bm{\mu}(k)}{k}\\
&-2\prod_{i=1}^{r}\left(1-q_i^{-1}\right)\sum_{f\mid {\sigma^{(q)}}}\frac{\log f}{f}\sum_{k\mid \frac{{\sigma^{(q)}}}{f}}\frac{\bm{\mu}(k)}{k}
-\prod_{i=1}^{r}\left(1-q_i^{-1}\right)\sum_{f\mid {\sigma^{(q)}}}\frac{1}{f}\sum_{k\mid \frac{{\sigma^{(q)}}}{f}}\frac{\bm{\mu}(k)\log k}{k}.
\end{align*}
By Lemma 4.11 of \cite{cheng2025b}, we obtain
\[
\fp_{s=1}L^{S}(s,\sigma^2)=\prod_{i=1}^{r}\left(1-q_i^{-1}\right)\left(\upgamma+\sum_{i=1}^{r}\frac{q_i^{-1}\log q_i}{1-q_i^{-1}}\right)-\prod_{i=1}^{r}\left(1-q_i^{-1}\right)\sum_{d\mid \sigma^{(q)}}\frac{\bm{\Lambda}(d)}{d},
\]
where $\bm{\Lambda}(d)$ denotes the von Mangoldt function.

Finally, for $\widetilde{F}(s)$ we have $\res_{s=0}\widetilde{F}(s)=1$ and $\fp_{s=0}\widetilde{F}(s)=0$ since $\widetilde{F}(s)$ is an odd function.

Therefore, we conclude that
\begin{align*}
&\res_{s=0}\widetilde{F}(s)L^{S}(1+s,\sigma^2)|\sigma^2|_{\infty,q}'^{\vartheta s}= \res_{s=0}\widetilde{F}(s)L^{S}(1+s,\sigma^2)|\sigma^{(q)}|^{2\vartheta s}\\
=&\res_{s=0}\widetilde{F}(s)\res_{s=0}L^{S}(1+s,\sigma^2)\left.\frac{\rmd}{\rmd s}\right|_{s=0}|\sigma^{(q)}|^{2\vartheta s}+\res_{s=0}\widetilde{F}(s)\fp_{s=0}L^{S}(1+s,\sigma^2) |\sigma^{(q)}|^{2\vartheta s}|_{s=0}\\
+&\fp_{s=0}\widetilde{F}(s)\res_{s=0}L^{S}(1+s,\sigma^2)|\sigma^{(q)}|^{2\vartheta s}|_{s=0}\\
=&\prod_{i=1}^{r}(1-q_i^{-1})\left[2\vartheta\log|\sigma^{(q)}|+\upgamma+\sum_{i=1}^{r}\frac{q_i^{-1}\log q_i}{1-q_i^{-1}}-\sum_{d\mid \sigma^{(q)}}\frac{\bm{\Lambda}(d)}{d}\right].
\end{align*}

Thus by \autoref{cor:sigmasquareexplicit} we obtain
\begin{proposition}\label{prop:sigmasquareexplicit}
The square term $\Sigma^n(\square)$ can be expressed as
\begin{align*}
&\prod_{i=1}^{r}(1-q_i^{-1})\sum_{\pm}\sum_{\nu\in \ZZ^r}\sum_{\substack{a\neq b\in \ZZ^S\\ ab=\pm nq^\nu}}\theta_\infty\begin{pmatrix} a & 0 \\ 0 & b \end{pmatrix}\theta_q\begin{pmatrix}
 a & 0 \\ 0 & b\end{pmatrix}\\
 \times&\left[2\vartheta\log|(a-b)^{(q)}|+\upgamma+\sum_{i=1}^{r}\frac{q_i^{-1}\log q_i}{1-q_i^{-1}}-\sum_{d\mid (a-b)^{(q)}}\frac{\bm{\Lambda}(d)}{d}\right].
\end{align*}
\end{proposition}

We now define the \emph{degree $1$ part of the hyperbolic part} for $f^n$ to be
\begin{equation}\label{eq:hyperbolicpartdeg1}
I_{\hyp}^{\deg=1}(f^n)=\prod_{i=1}^{r}(1-q_i^{-1})\sum_{\pm}\sum_{\nu\in \ZZ^r}\sum_{\substack{a\neq b\in \ZZ^S\\ ab=\pm nq^\nu}}\theta_\infty\begin{pmatrix} a & 0 \\ 0 & b \end{pmatrix}\theta_q\begin{pmatrix}
 a & 0 \\ 0 & b\end{pmatrix}
\end{equation}
and the \emph{modified unramified hyperbolic part of the second kind} for $f^n$ to be
\begin{equation}\label{eq:hyperbolicpartmodify}
\widehat{J}_{\hyp}^{S}(f^n)=-\prod_{i=1}^{r}(1-q_i^{-1})\sum_{\pm}\sum_{\nu\in \ZZ^r}\sum_{\substack{a\neq b\in \ZZ^S\\ ab=\pm nq^\nu}}\theta_\infty\begin{pmatrix} a & 0 \\ 0 & b \end{pmatrix}\theta_q\begin{pmatrix}
 a & 0 \\ 0 & b\end{pmatrix}\sum_{d\mid (a-b)^{(q)}}\frac{\bm{\Lambda}(d)}{d}.
\end{equation}
\begin{remark}
One can define $I_{\hyp}^{\deg=1}(f)$ and $\widehat{J}_{\hyp}^{S}(f)$ for arbitrary $f=\bigotimes_{v\in \mf{S}}'f_v\in C_c^\infty(\G(\AA)^1)$ and show that such definitions agree with \eqref{eq:hyperbolicpartdeg1} and \eqref{eq:hyperbolicpartmodify}, respectively.
\end{remark}

Hence we obtain
\begin{corollary}\label{cor:sigmasquareexplicit3}
We have
\begin{align*}
  \Sigma^n(\square) & =\upgamma_SI_{\hyp}^{\deg=1}(f^n)+\widehat{J}_{\hyp}^{S}(f^n) \\
   & +2\vartheta\prod_{i=1}^{r}(1-q_i^{-1})\sum_{\pm}\sum_{\nu\in \ZZ^r}\sum_{\substack{a\neq b\in \ZZ^S\\ ab=\pm nq^\nu}}\theta_\infty\begin{pmatrix} a & 0 \\ 0 & b \end{pmatrix}\theta_q\begin{pmatrix}
 a & 0 \\ 0 & b\end{pmatrix}\log|(a-b)^{(q)}|,
\end{align*}
where
\begin{equation}\label{eq:defgammas}
\upgamma_S=\upgamma+\sum_{i=1}^{r}\frac{q_i^{-1}\log q_i}{1-q_i^{-1}}.
\end{equation}
\end{corollary}

By \autoref{thm:contributeellipticpoisson} and \autoref{cor:sigmasquareexplicit3} we obtain the contribution of the elliptic part with parameter $\vartheta$:
\begin{theorem}\label{thm:contributeellipticfull}
For any $\varepsilon>0$, there exists $\varepsilon'>0$ such that if $|\vartheta-1/2|<\varepsilon'$, then
\begin{align*}
  &\sum_{\substack{n<X\\\gcd(n,S)=1}}I_\el(f^n) =\mf{A}X^{\frac32}+\mf{B}\vartheta'(X\log X-X)+(\mf{C}+\mf{D}\vartheta')X-\upgamma_S\sum_{\substack{n<X\\\gcd(n,S)=1}}I_{\hyp}^{\deg=1}(f^n) \\
   -&\sum_{\substack{n<X\\\gcd(n,S)=1}}\widehat{J}_{\hyp}^{S}(f^n) -2\vartheta\prod_{i=1}^{r}(1-q_i^{-1})\sum_{\substack{n<X\\\gcd(n,S)=1}}\sum_{\pm}\sum_{\nu\in \ZZ^r}\sum_{\substack{a\neq b\in \ZZ^S\\ ab=\pm nq^\nu}}\theta_\infty\begin{pmatrix} a & 0 \\ 0 & b \end{pmatrix}\theta_q\begin{pmatrix}
 a & 0 \\ 0 & b\end{pmatrix}\log|(a-b)^{(q)}|\\
 +&O(X^{\frac78+\varepsilon}),
\end{align*}
where the implied constant only depends on $\vartheta$, $\varepsilon$, $f_\infty$ and $f_{q_i}$.
\end{theorem}

Since $\vartheta'=1-\vartheta$, the above formula is a linear polynomial for $\vartheta$. Hence the condition that $|\vartheta-1/2|<\varepsilon'$ is redundant and we obtain \autoref{thm:contributeellipticfinal} and \autoref{thm:contributeellipticdeg1}.

\appendix
\section{An application to modular forms}
As another application of \autoref{thm:contributeellipticfinal}, we will prove an estimate of the sum of the traces of Hecke operators of modular forms in this appendix, which generali.zes the main result of \cite{altug2020}.
Let $T_{m,N,\omega'}(n)$ denote the normalized $n^{\mathrm{th}}$ Hecke operator acting on $S_m(\Gamma_{0}(N),\omega')$, the space of holomorphic cusp forms of weight $m>2$ and level $N$ with nebentypus $\omega'$. More precisely, for $\varphi\in S_m(\Gamma_{0}(N),\omega')$, we define
\[
T_{m,N,\omega'}(n)(\varphi)(z)=n^{\frac{m-1}{2}}\sum_{\substack{ad=n\\ a>0\\ b\bmod d }}\omega'(a)^{-1}d^{-m}\varphi\legendresymbol{az+b}{d}.
\]

We set $S=\{\infty,2\}\cup\{q\text{ prime}\,|\, q\mid N\}$ in this appendix. The main theorem we will prove in this appendix is
\begin{theorem}\label{thm:modularformestimate}
Suppose that $m>2$. Then for any $\varepsilon>0$ we have
\[
\sum_{\substack{n<X\\\gcd(n,S)=1}} \Tr(T_{m,N,\omega'}(n))\ll_{m,N,\omega',\varepsilon} X^{\frac78+\varepsilon}.
\]
\end{theorem}
In this appendix, we assume that $f_\infty=f_{\infty,m}$ is the matrix coefficient of a weight $m$ discrete series with $m>2$. Although $f_\infty$ is not compactly supported modulo $Z_+$, the same arguments in \cite{cheng2025b} and in this paper still work, except that we need to verify $I_\hyp(f^n)\ll_\varepsilon n^\varepsilon$, which can be easily verified by the formula of $I_\hyp(f^n)$ at the end of Step 1 of the proof of \autoref{thm:contributehyperbolicmodular} below. 

Also, we fix the Haar measure on $\G(\RR)$ to be
  \[
  \rmd g=\frac{\rmd a\wedge\rmd b\wedge\rmd c\wedge\rmd d}{(ad-bc)^2}
  \]
  for $g=(\begin{smallmatrix}
            a & b \\
            c & d 
          \end{smallmatrix})$.
Such measure normalization on $\G(\RR)$ is taken from \cite{knightly2006traces}.

Recall that the Arthur-Selberg trace formula for $\GL_2$ and $f\in C_c^\infty(\G(\AA)^1)=C_c^\infty(Z_+\bs\G(\AA))$ (where $Z_+$ denotes the identity component of $\Z(\RR)$, the center of $\G(\RR)$) can be written in the following form:
\begin{equation}\label{eq:traceformulagl2}
J_{\mathrm{geom}}(f)=J_{\mathrm{spec}}(f),
\end{equation}
where the geometric side is
\begin{equation}\label{eq:geometric}
J_{\mathrm{geom}}(f)=I_\id(f)+I_\el(f)+J_\hyp(f)+J_\unip(f)
\end{equation}
and the spectral side is
\begin{equation}\label{eq:spectral}
J_{\mathrm{spec}}(f)=I_{\mathrm{cusp}}(f)+J_{\mathrm{cont}}(f)+\sum_{\mu}\Tr(\mu(f)) +\frac14\sum_{\mu}\Tr(M(\mu,0)(\xi_0\otimes\mu)(f)).
\end{equation}
The sum of $\mu$ is over all $1$-dimensional representations of $\G(\QQ)\bs \G(\AA)^1$. 

%The Arthur-Selberg trace formula of $\GL_2$ was occurred in lots of literature. For example, \cite[Chapter 16]{langlands1970}, \cite[Theorem 6.33]{gelbart1979} and \cite[Theorem 7.14]{knapp1997}.
%The explicit formula of these terms are contained in these references. 

By \cite[Proposition 26.7]{knightly2006traces}, $f_{\infty,m}\in C^\infty(Z_+\bs \GL_2(\RR))$ and the orbital integrals are compactly supported modulo $Z_+$. Moreover, we have $\theta_\infty^{-}(x)=0$ and
\begin{equation}\label{eq:discreteseriestheta}
\theta_\infty^+(x)=\begin{dcases}
  \frac{\rmi}{2\uppi}\left[(x+\rmi\sqrt{1-x^2})^{m-1}-(x-\rmi\sqrt{1-x^2})^{m-1}\right], & \text{if $|x|<1$},  \\
  0, & \text{otherwise}.
\end{dcases}
\end{equation}
Note that the factor $2\uppi$ comes from the different normalization of $\G_\gamma(\RR)$ (since we chose the measure of $\Z(\RR)\bs \G_\gamma(\RR)=\Z(\RR)\bs \CC^\times$ to be $2\uppi$, see \autoref{subsec:measure}), which does not occur in \cite{knightly2006traces}. Moreover, $f_{\infty,m}$ is supercuspidal as defined in the introductory section (cf. \cite[Lemma 13.9]{knightly2006traces}). Hence by \cite[Lemma 16.4.2]{getz2024} (the conclusion in loc. cit. still holds for archimedean places and the proof is the same), we obtain

\begin{proposition}\label{prop:cuspidalimage}
The right multiplication action $R$ for $f^n$ on $L^2(\G(\QQ)\bs \G(\AA)^1)$ has cuspidal image. In other words,
\[
J_\spec(f^n)=I_\cusp(f^n).
\]
\end{proposition}

\subsection{Contribution of the geometric side}
In this subsection we will prove that
\begin{equation}\label{eq:modularformgeometric}
\lim_{X\to +\infty}\frac{1}{X}\sum_{\substack{n<X\\ \gcd(n,S)=1}}J_{\mathrm{geom}}(f^n)=0,
\end{equation}
thereby proving \eqref{eq:standardrepresentation} for $f_\infty=f_{\infty,m}$ with $m>2$.

Firstly, we analyze the identity and the unipotent part.
\begin{proposition}\label{prop:contributeidentitymodular}
We have
\[
\sum_{\substack{n<X\\ \gcd(n,S)=1}} I_\id(f^n)\ll X^{1/2},
\]
where the implied constant only depends on $m$ and $f_{q_i}$.
\end{proposition}
\begin{proof}
By the proof of Proposition 3.3 of \cite{cheng2025b}, $I_\id(f^n)=0$ if $n$ is not a square, and $I_\id(f^n)\ll_{m,f_{q_i}} 1$ if $n$ is a square. Hence 
\[
\sum_{\substack{n<X\\ \gcd(n,S)=1}} I_\id(f^n)\ll \sum_{\substack{n<X\\ \gcd(n,S)=1\\ n=\square}} 1\ll X^{1/2}.\qedhere
\]
\end{proof}

\begin{proposition}\label{prop:contributeunipotentmodular}
We have
\[
\sum_{\substack{n<X\\ \gcd(n,S)=1}} J_\unip(f^n)\ll X^{1/2+\varepsilon},
\]
where the implied constant only depends on $m$, $f_{q_i}$ and $\varepsilon>0$.
\end{proposition}
\begin{proof}
By the proof of Proposition 3.4 of \cite{cheng2025b}, $J_\unip(f^n)=0$ if $n$ is not a square, and $J_\unip(f^n)\ll_{m,f_{q_i},\varepsilon} X^{\varepsilon}$ if $n$ is a square. Hence 
\[
\sum_{\substack{n<X\\ \gcd(n,S)=1}} J_\unip(f^n)\ll \sum_{\substack{n<X\\ \gcd(n,S)=1\\ n=\square}} X^\varepsilon\ll X^{1/2+\varepsilon}.\qedhere
\]
\end{proof}

Next, we consider the hyperbolic part.
\begin{theorem}\label{thm:contributehyperbolicmodular}
We have 
\[
\sum_{\substack{n<X\\ \gcd(n,S)=1}} J_\hyp(f^n)=0
\]
if $m$ is odd and
\[
\sum_{\substack{n<X\\ \gcd(n,S)=1}} J_\hyp(f^n)=2\prod_{i=1}^{r}(1-q_i^{-1})^3\int_{\A(\QQ_{q_i})}\theta_{q_i}(t)\rmd t \frac{X}{1-m}+O(X^{1/2})
\]
if $m$ is even, where $\A$ denotes the diagonal torus of $\G$. The implied constant only depends on $m$ and $f_{q_i}$.
\end{theorem}
\begin{proof}
The proof is divided into the following four steps.

\underline{\emph{Step 1:}}\ \ Computation of the hyperbolic part.

By definition,
\[
J_\hyp(f^n)=-\frac{1}{2}\sum_{\gamma\in \A(\QQ)_{\reg}}\int_{\A(\AA)\bs \G(\AA)}f^n(g^{-1}\gamma g)\alpha(H_\B(wg)+H_\B(g))\rmd g,
\]
where $w$ is the nontrivial element in the Weyl group of $(\G,\A)$, $\alpha$ denotes the simple root, and $H_\B$ denotes the Harish-Chandra map. 
We have
\begin{equation}\label{eq:weightorbitalintegral}
\begin{split}
  & \int_{\A(\AA)\bs \G(\AA)}f^n(g^{-1}\gamma g)\alpha(H_\B(wg)+H_\B(g))\rmd g \\
  = & \sum_{v\in \mf{S}}\int_{\A(\QQ_v)\bs \G(\QQ_v)}f_v^n(g_v^{-1}\gamma g_v)\alpha(H_\B(wg_v)+H_\B(g_v))\rmd g_v\prod_{w\neq v}\int_{\A(\QQ_w)\bs \G(\QQ_w)}f_w^n(g_w^{-1}\gamma g_w)\rmd g_w.
\end{split} 
\end{equation}

By \cite[Proposition 24.2]{knightly2006traces}, we have
\[
\int_{\A(\RR)\bs \G(\RR)}f_{\infty,m}(g^{-1}\gamma g)\rmd g=0,
\]
and by \cite[Proposition 24.3]{knightly2006traces}, for $\gamma=(\begin{smallmatrix}  \gamma_1 &  \\   & \gamma_2\end{smallmatrix})$ with $\gamma_1,\gamma_2>0$,
\[
\int_{\A(\RR)\bs \G(\RR)}f_{\infty,m}(g^{-1}\gamma g)\alpha(H_\B(wg)+H_\B(g))\rmd g=\frac{|\gamma_1\gamma_2|^{1-m/2}\min\{|\gamma_1|,|\gamma_2|\}^{m-1}}{|\gamma_1-\gamma_2|}.
\]
Note that $f_{\infty,m}$ is an even function if $m$ is even. Hence we obtain the same formula of the local weighted orbital integral if $\gamma_1,\gamma_2<0$. Finally, since $f_{\infty,m}$ is supported on $\GL_2^{\! +}(\RR)$, we know that the weighted orbital integral vanishes if $\gamma_1\gamma_2<0$. 

Therefore, for $\gamma_1\gamma_2<0$ we have $\eqref{eq:weightorbitalintegral}=0$, and for $\gamma_1\gamma_2>0$ we have
\[
\eqref{eq:weightorbitalintegral}= \frac{|\gamma_1\gamma_2|^{1-m/2}\min\{|\gamma_1|,|\gamma_2|\}^{m-1}}{|\gamma_1-\gamma_2|} \prod_{p}\orb(f^n_p;\gamma).
\]

Now we assume that $\det\gamma=\gamma_1\gamma_2>0$.

If $m$ is odd, then the elements $\gamma$ and $-\gamma$ give opposite contribution of $J_\hyp(f^n)$ and hence $J_\hyp(f^n)=0$. Thus the conclusion holds trivially. Now we assume that $m$ is even.

By Theorem 2.7 of \cite{cheng2025}, \eqref{eq:weightorbitalintegral} does not vanish only if $\gamma_1,\gamma_2\in \ZZ^S$ and $\mathopen{|}\det\gamma\mathclose{|}_p=|n|_p$ for all $p\notin S$. That is, $\det\gamma=nq^\nu$ for some $\nu\in \ZZ^r$. Moreover, we have
\[
\orb(f^n_p;\gamma)=p^{-n_p/2}p^{k_\gamma}=\frac{p^{-n_p/2}}{|\gamma_1-\gamma_2|_p}
\]
if $p\notin S$. By the definition of $\theta_{p}(\gamma)$ for $p\in S$ we have
\[
\orb(f^n_p;\gamma)=\left(1-\frac{1}{p}\right)p^{{k_\gamma}}\mathopen{|}\det\gamma\mathclose{|}^{1/2}\theta_{p}(\gamma)= \left(1-\frac{1}{p}\right)\frac{|\gamma_1\gamma_2|_p^{1/2}}{|\gamma_1-\gamma_2|_p}\theta_{p}(\gamma).
\]
By product formula we have $\prod_{v\in \mf{S}}1/|\gamma_1-\gamma_2|_v=1$. Hence
\[
\eqref{eq:weightorbitalintegral}=(nq^\nu)^{\frac{1-m}{2}}\min\{|\gamma_1|,|\gamma_2|\}^{m-1} \prod_{i=1}^{r}\left(1-\frac{1}{q_i}\right)\theta_{q_i}(\gamma).
\]
For fixed $n$, we must have $\gamma_1=dq^\alpha$ and $\gamma_2=nq^\beta/d$, or $\gamma_1=-dq^\alpha$ and $\gamma_2=-nq^\beta/d$, where $d$ is a divisor of $n$ and $\alpha+\beta=\nu$. By symmetry these two cases give the same contribution to $J_\hyp(f^n)$. Hence
\[
J_\hyp(f^n)=-\sum_{\alpha,\beta\in \ZZ^r}\sum_{\substack{d\mid n\\dq^\alpha\neq nq^\beta/d}}(nq^{\alpha+\beta})^{\frac{1-m}{2}}\min\{dq^\alpha,nq^\beta/d\}^{m-1} \prod_{i=1}^{r}\left(1-\frac{1}{q_i}\right)\theta_{q_i}\begin{pmatrix}
                                                         dq^\alpha & 0 \\
                                                         0 & nq^\beta/d 
                                                       \end{pmatrix}.
\]

\underline{\emph{Step 2:}}\ \ Fourier inversion.

Since $f_{q_i}$ is compactly supported, $\theta_{q_i}(\gamma)$ is compactly supported. Hence there exist $M_i>0$ such that $\theta_{q_i}(\begin{smallmatrix}    x & 0 \\   0 & y \end{smallmatrix})=0$ if $|v_{q_i}(x)|\geq M_i$ or $|v_{q_i}(y)|\geq M_i$. Hence the above sum over $\alpha$ and $\beta$ is finite (which does not depend on $n$).

Now we fix $\alpha$ and $\beta$ and consider
\begin{equation}\label{eq:hyperbolicsumn}
\sum_{\substack{n<X\\ n\in \ZZ_{(S)}^{>0}}}\sum_{\substack{d\mid n\\dq^\alpha\neq nq^\beta/d}}(nq^{\alpha+\beta})^{\frac{1-m}{2}}\min\{dq^\alpha,nq^\beta/d\}^{m-1} \prod_{i=1}^{r}\left(1-\frac{1}{q_i}\right)\theta_{q_i}\begin{pmatrix}
                                                         dq^\alpha & 0 \\
                                                         0 & nq^\beta/d 
                                                       \end{pmatrix}.
\end{equation}

By \eqref{eq:shalikalocal}, $\theta_{q_i}(\gamma)$ is smooth for $\gamma\in \A(\QQ_{q_i})$. Hence there exists $L_i>0$ such that for any $a\in 1+q_i^{L_i}\ZZ_{q_i}$, we have
\[
\theta_{q_i}\begin{pmatrix}  ax & 0 \\  0 & y \end{pmatrix}= \theta_{q_i}\begin{pmatrix}  x & 0 \\  0 & y \end{pmatrix}=\theta_{q_i}\begin{pmatrix}  x & 0 \\  0 & ay \end{pmatrix}
\]
for any $x,y\in \QQ_{q_i}^\times$.

Now we consider the following function defined on $(x,y)\in \ZZ_{q_i}^\times\times\ZZ_{q_i}^\times$:
\[
\Xi_{i,\alpha,\beta}(x,y)=\theta_{q_i}\begin{pmatrix}   xq^\alpha & 0 \\  0 & yq^\beta \end{pmatrix}.
\]
By the above argument, we know that $\Xi_{i,\alpha,\beta}(x,y)=\Xi_{i,\alpha,\beta}(ax,y)=\Xi_{i,\alpha,\beta}(x,ay)$ for any $x,y\in \ZZ_{q_i}^\times$ and $a\in 1+q_i^{L_i}\ZZ_{q_i}$. For $\chi_i,\chi_i'\in \widehat{\ZZ_{q_i}^\times}$, we consider the Fourier transform
\[
\widehat{\Xi}_{i,\alpha,\beta}(\chi_i,\chi_i')=\int_{\ZZ_{q_i}^\times\times \ZZ_{q_i}^\times}\Xi_{i,\alpha,\beta}(x,y)\overline{\chi_i(x)}~\overline{\chi_i'(y)}\rmd x\rmd y.
\]
If $\chi_i$ is not trivial on $1+q_i^{L_i}\ZZ_{q_i}$, then there exists $a\in 1+q_i^{L_i}\ZZ_{q_i}$ such that $\chi_i(a)\neq 1$. Since $\Xi_{i,\alpha,\beta}(x,y)=\Xi_{i,\alpha,\beta}(ax,y)$, we obtain
\[
\widehat{\Xi}_{i,\alpha,\beta}(\chi_i,\chi_i')=\int_{\ZZ_{q_i}^\times\times \ZZ_{q_i}^\times}\Xi_{i,\alpha,\beta}(ax,y)\overline{\chi_i(ax)}~\overline{\chi_i'(y)}\rmd x\rmd y=\overline{\chi_i(a)}\widehat{\Xi}_{i,\alpha,\beta}(\chi_i,\chi_i')
\]
and thus $\widehat{\Xi}_{i,\alpha,\beta}(\chi_i,\chi_i')=0$. Similarly, if $\chi_i'$ is not trivial on $1+q_i^{L_i}\ZZ_{q_i}$, then $\widehat{\Xi}_{i,\alpha,\beta}(\chi_i,\chi_i')=0$. Hence $\widehat{\Xi}_{i,\alpha,\beta}(\chi_i,\chi_i')\neq 0$ only if 
\[
\chi_i,\chi_i'\in (\ZZ_{q_i}/1+q_i^{L_i}\ZZ_{q_i})\sphat\,.
\]
Hence the choice of the pair $(\chi_i,\chi_i')$ is finite and thus by Fourier inversion formula
\[
\Xi_{i,\alpha,\beta}(x,y)=\left(1-\frac{1}{q_i}\right)^{-2}\sum_{\chi_i,\chi_i'}\widehat{\Xi}_{i,\alpha,\beta}(\chi_i,\chi_i') \chi_i(x)\chi_i'(y),
\]
where the factor $(1-1/q_i)^{-2}$ comes from the measure of $\ZZ_{q_i}^\times\times \ZZ_{q_i}^\times$. Therefore
\begin{align*}
\eqref{eq:hyperbolicsumn}&=
\sum_{\substack{n<X\\ n\in \ZZ_{(S)}^{>0}}}\sum_{\substack{d\mid n\\dq^\alpha\neq \frac{nq^\beta}{d}}}(nq^{\alpha+\beta})^{\frac{1-m}{2}}\min\left\{dq^\alpha,\frac{nq^\beta}{d}\right\}^{m-1} \prod_{i=1}^{r}\left(1-\frac{1}{q_i}\right)^{-1}\sum_{\chi_i,\chi_i'}\widehat{\Xi}_{i,\alpha,\beta}(\chi_i,\chi_i') \chi_i(d)\chi_i'(n/d)\\
&=\sum_{\substack{ab<X\\ aq^\alpha\neq bq^\beta}}(ab)^{\frac{1-m}{2}}q^{(\alpha+\beta)\frac {1-m}{2}}\min\{aq^\alpha,bq^\beta\}^{m-1} \sum_{\chi,\chi'}\chi(a)\chi'(b)\prod_{i=1}^{r}\left(1-\frac{1}{q_i}\right)^{-1} \widehat{\Xi}_{i,\alpha,\beta}(\chi_i,\chi_i'),
\end{align*}
where $\chi,\chi'$ run over all characters on $\ZZ_{S_\fin}^\times$ and the sum over $\chi$ and $\chi'$ is actually finite. We also consider them as Dirichlet characters so that $\chi(m)=\chi'(m)=0$ if $\gcd(m,S)\neq 1$.

\underline{\emph{Step 3:}}\ \ Estimate of character sums.

We will give an asymptotic formula for
\begin{equation}\label{eq:charactersumhyperbola}
\sum_{\substack{ab<X\\ aq^\alpha\neq bq^\beta}}(ab)^{\frac{1-m}{2}}q^{(\alpha+\beta)\frac {1-m}{2}}\min\{aq^\alpha,bq^\beta\}^{m-1} \chi(a)\chi'(b).
\end{equation}

We have
\[
\eqref{eq:charactersumhyperbola}= \sum_{\substack{ab<X\\ aq^\alpha<bq^\beta}}(ab)^{\frac{1-m}{2}}q^{(\alpha+\beta)\frac{1-m}{2}}(aq^\alpha)^{m-1} \chi(a)\chi'(b)+\sum_{\substack{ab<X\\ aq^\alpha> bq^\beta}}(ab)^{\frac{1-m}{2}}q^{(\alpha+\beta)\frac {1-m}{2}}(bq^\beta)^{m-1} \chi(a)\chi'(b)
\]

For the $aq^\alpha<bq^\beta$ part of \eqref{eq:charactersumhyperbola}, we have
\[
\sum_{\substack{ab<X\\ aq^\alpha< bq^\beta}}=\sum_{\substack{b\leq X^{1/2}q^{(\alpha-\beta)/2}\\ a< bq^{\beta-\alpha}}}+\sum_{\substack{X^{1/2}q^{(\alpha-\beta)/2}<b<X\\ a<X/b}}.
\]

For the first sum, we have
\begin{align*}
   & \sum_{\substack{b\leq X^{1/2}q^{(\alpha-\beta)/2}\\ a< bq^{\beta-\alpha}}}(ab)^{\frac{1-m}{2}}q^{(\alpha+\beta)\frac {1-m}{2}}(aq^\alpha)^{m-1} \chi(a)\chi'(b) \\
  = & q^{(\alpha-\beta)\frac{m-1}{2}}\sum_{b\leq X^{1/2}q^{(\alpha-\beta)/2}}b^{\frac{1-m}{2}}\chi'(b)\sum_{a<bq^{\beta-\alpha}}a^{\frac{m-1}{2}}\chi(a).
\end{align*}

By Abel summation formula, we have
\[
\sum_{a<bq^{\beta-\alpha}}a^{\frac{m-1}{2}}\chi(a)= \prod_{i=1}^{r}\left(1-\frac{1}{q_i}\right)\frac{2}{m+1}(bq^{\beta-\alpha})^{\frac{m+1}{2}} \delta(\chi)+O_{m,\alpha,\beta}(b^{\frac{m-1}{2}}),
\]
where $\delta(\chi)=1$ if $\chi$ is the trivial character, and $\delta(\chi)=0$ otherwise. Therefore
\begin{align*}
   &\sum_{b\leq X^{1/2}q^{(\alpha-\beta)/2}}b^{\frac{1-m}{2}}\chi'(b)\sum_{a<bq^{\beta-\alpha}}a^{\frac{m-1}{2}} \chi(a)\\
   =& \prod_{i=1}^{r}\left(1-\frac{1}{q_i}\right)\frac{2}{m+1}q^{(\beta-\alpha)\frac{m+1}{2}} \delta(\chi)\sum_{b\leq X^{1/2}q^{(\alpha-\beta)/2}}b\chi'(b)+\sum_{b\leq X^{1/2}q^{(\alpha-\beta)/2}}O_{m,\alpha,\beta}(1) \\
  = & \prod_{i=1}^{r}\left(1-\frac{1}{q_i}\right)^2\frac{Xq^{\alpha-\beta}}{m+1}q^{(\beta-\alpha)\frac{m+1}{2}} \delta(\chi)\delta(\chi')+O_{m,\alpha,\beta}(X^{\frac 12}).
\end{align*}
Hence
\[
   \sum_{\substack{b\leq X^{1/2}q^{(\alpha-\beta)/2}\\ a< bq^{\beta-\alpha}}}(ab)^{\frac{1-m}{2}}q^{(\alpha+\beta)\frac{1-m}{2}}(aq^\alpha)^{m-1} \chi(a)\chi'(b)
  = \prod_{i=1}^{r}\left(1-\frac{1}{q_i}\right)^2\frac{X}{m+1} \delta(\chi)\delta(\chi')+O_{m,\alpha,\beta}(X^{\frac 12}).
\]

For the second sum, we have
\begin{align*}
   & \sum_{\substack{X^{1/2}q^{(\alpha-\beta)/2}<b<X\\ a<X/b}}(ab)^{\frac{1-m}{2}}q^{(\alpha+\beta)\frac{1-m}{2}}(aq^\alpha)^{m-1} \chi(a)\chi'(b) \\
  = & q^{(\alpha-\beta)\frac{m-1}{2}}\sum_{ X^{1/2}q^{(\alpha-\beta)/2}<b<X}b^{\frac{1-m}{2}}\chi'(b)\sum_{a<X/b}a^{\frac{m-1}{2}}\chi(a).
\end{align*}
Since
\[
\sum_{a<X/b}a^{\frac{m-1}{2}}\chi(a)= \prod_{i=1}^{r}\left(1-\frac{1}{q_i}\right)\frac{2}{m+1}\legendresymbol{X}{b}^{\frac{m+1}{2}} \delta(\chi)+O_{m,\alpha,\beta}\left(\legendresymbol{X}{b}^{\frac{m-1}{2}}\right),
\]
we obtain
\begin{align*}
   &\sum_{ X^{1/2}q^{(\alpha-\beta)/2}<b<X}b^{\frac{1-m}{2}}\chi'(b)\sum_{a<X/b}a^{\frac{m-1}{2}}\chi(a)\\
   =& \prod_{i=1}^{r}\left(1-\frac{1}{q_i}\right)\frac{2}{m+1} \delta(\chi)X^{\frac{m+1}{2}}\sum_{ X^{1/2}q^{(\alpha-\beta)/2}<b<X}b^{-m}\chi'(b)+X^{\frac{m-1}{2}}\sum_{ X^{1/2}q^{(\alpha-\beta)/2}<b<X}O_{m,\alpha,\beta}(b^{1-m}) \\
  = & \prod_{i=1}^{r}\left(1-\frac{1}{q_i}\right)^2\frac{Xq^{(\alpha-\beta)\frac{1-m}{2}}}{(m+1)(m-1)} \delta(\chi)\delta(\chi')+O_{m,\alpha,\beta}(X^{\frac 12}).
\end{align*}
Hence
\begin{align*}
  & \sum_{\substack{q^{X^{1/2}(\alpha-\beta)/2}<b<X\\ a< bq^{\beta-\alpha}}}(ab)^{\frac{1-m}{2}}q^{(\alpha+\beta)\frac{1-m}{2}}(aq^\alpha)^{m-1} \chi(a)\chi'(b)\\
  = &\prod_{i=1}^{r}\left(1-\frac{1}{q_i}\right)^2\frac{2X}{(m+1)(m-1)} \delta(\chi)\delta(\chi')+O_{m,\alpha,\beta}(X^{\frac 12}).
\end{align*}

Therefore
\begin{align*}
   & \sum_{\substack{ab<X\\ aq^\alpha<bq^\beta}}(ab)^{\frac{1-m}{2}}q^{(\alpha+\beta)\frac{1-m}{2}}\min\{aq^\alpha,bq^\beta\}^{m-1} \chi(a)\chi'(b) \\
  = &\prod_{i=1}^{r}\left(1-\frac{1}{q_i}\right)^2\left(\frac{X}{m+1}+\frac{2X}{(m+1)(m-1)}\right) \delta(\chi)\delta(\chi')+O_{m,\alpha,\beta}(X^{\frac 12})\\
  =&\prod_{i=1}^{r}\left(1-\frac{1}{q_i}\right)^2\frac{X}{m-1} \delta(\chi)\delta(\chi')+O_{m,\alpha,\beta}(X^{\frac 12}).
\end{align*}
Similarly,
\[
   \sum_{\substack{ab<X\\ aq^\alpha>bq^\beta}}(ab)^{\frac{1-m}{2}}q^{(\alpha+\beta)\frac{1-m}{2}}\min\{aq^\alpha,bq^\beta\}^{m-1} \chi(a)\chi'(b)=\prod_{i=1}^{r}\left(1-\frac{1}{q_i}\right)^2\frac{X}{m-1} \delta(\chi)\delta(\chi')+O_{m,\alpha,\beta}(X^{\frac 12}).
\]
Thus
\[
\eqref{eq:charactersumhyperbola}=\prod_{i=1}^{r}\left(1-\frac{1}{q_i}\right)^2\frac{2X}{m-1} \delta(\chi)\delta(\chi')+O_{m,\alpha,\beta}(X^{\frac 12}).
\]

\underline{\emph{Step 4:}}\ \ Final formula.

Now we sum over all $\chi$ and $\chi'$ and use the formula at the end of Step 2, we obtain
\[
\eqref{eq:hyperbolicsumn}=\frac{2X}{m-1}
\prod_{i=1}^{r}\left(1-\frac{1}{q_i}\right) \widehat{\Xi}_{i,\alpha,\beta}(\triv,\triv)+O_{m,\alpha,\beta,f_{q_i}}(X^{\frac 12}).
\]
Summing over $\alpha$ and $\beta$, we obtain
\[
\sum_{\substack{n<X\\\gcd(n,S)=1}}J_\hyp(f^n)=-\sum_{\alpha,\beta\in \ZZ^r}\frac{2X}{m-1}
\prod_{i=1}^{r}\left(1-\frac{1}{q_i}\right) \widehat{\Xi}_{i,\alpha,\beta}(\triv,\triv)+O_{m,f_{q_i}}(X^{\frac 12}).
\]

Since
\begin{align*}
\sum_{\alpha,\beta\in \ZZ^r}\prod_{i=1}^{r}\widehat{\Xi}_{i,\alpha,\beta}(\triv,\triv)=&
\sum_{\alpha,\beta\in \ZZ^r}\int_{x,y\in \ZZ_{q_1}^\times\times\dots \ZZ_{q_r}^\times}\prod_{i=1}^{r}\theta_{q_i}\begin{pmatrix}   xq^\alpha & 0 \\  0 & yq^\beta \end{pmatrix}\rmd x\rmd y\\
=&\prod_{i=1}^{r}\sum_{\alpha_i,\beta_i\in \ZZ}
\left(1-\frac{1}{q_i}\right)^2 \int_{x,y\in \ZZ_{q_i}^\times}\theta_{q_i}\begin{pmatrix}   xq_i^{\alpha_i} & 0 \\  0 & yq_i^{\beta_i} \end{pmatrix}\rmd^\times x\rmd^\times y\\
=&\prod_{i=1}^{r}
\left(1-\frac{1}{q_i}\right)^2\int_{\A(\QQ_{q_i})}\theta_{q_i}(t)\rmd t,
\end{align*}
the conclusion follows.
\end{proof}

Finally we consider the elliptic part. Since we have computed the asymptotic formula, it suffices to compute $\mf{A}$, $\mf{B}$ and $\mf{C}$ explicitly. We first prove the following lemmas:
\begin{lemma}\label{lem:archimedeaneisenstein}
We have
\[
\int_{X_1}|1-x|^{-1}\widehat{\Theta}_\infty(x)|x|^{-\frac12}\rmd x=2\int_{-1}^{1}\frac{\theta_\infty^+(x)}{\sqrt{1-x^2}}\rmd x.
\]
and this integral is $0$ if $m$ is odd, and is 
\[
-\frac{4}{\uppi(m-1)}
\]
if $m$ is even.
\end{lemma}
\begin{proof}
By \eqref{eq:thetarelation} and recall that $\widehat{\Theta}_\infty(x)=\Theta_\infty^+(x)+\Theta_\infty^-(x)$, we have
\[
\int_{X_1}|1-x|^{-1}\widehat{\Theta}_\infty(x)|x|^{-\frac12}\rmd x=\int_{-\infty}^{0}\sum_{\pm}\frac{\theta_\infty^+\left(\pm \frac{1}{\sqrt{1-x}}\right)}{1-x}|x|^{-\frac12}\rmd x
\]
By making change of variable $1/\sqrt{1-x}\mapsto t$ so that $x=1-1/t^2$ and $\rmd x=2\rmd t/t^3$, we obtain
\[
\int_{-\infty}^{0}\sum_{\pm}\frac{\theta_\infty^+\left(\pm \frac{1}{\sqrt{1-x}}\right)}{1-x}|x|^{-\frac12}\rmd x =\sum_{\pm}\int_{0}^{1}\frac{\theta_\infty^+(\pm t)t^2}{\sqrt{1/t^2-1}}\frac{2\rmd t}{t^3}=2\int_{-1}^{1}\frac{\theta_\infty^+(x)}{\sqrt{1-x^2}}\rmd x.
\]
Recall that  $\theta_\infty^{-}(x)=0$ and
\[
\theta_\infty^+(x)=\begin{dcases}
  \frac{\rmi}{2\uppi}\left[(x+\rmi\sqrt{1-x^2})^{m-1}-(x-\rmi\sqrt{1-x^2})^{m-1}\right], & \text{if $|x|<1$},  \\
  0, & \text{otherwise}.
\end{dcases}
\]
Hence
\begin{align*}
2\int_{-1}^{1}\frac{\theta_\infty^\pm(x)}{\sqrt{1-x^2}}&= 2\int_{-1}^{1}\frac{\rmi}{2\uppi}\frac{(x+\rmi\sqrt{1-x^2})^{m-1}-(x-\rmi\sqrt{1-x^2})^{m-1}}{\sqrt{1-x^2}}\rmd x\\
&=2\int_{0}^{\uppi}\frac{\rmi}{2\uppi} \left[(\cos\theta+\rmi\sin\theta)^{m-1}-(\cos\theta-\rmi\sin\theta)^{m-1}\right]\rmd\theta\\
&=\frac{\rmi}{\uppi}\left[\int_{0}^{\uppi}\rme^{\rmi(m-1)\theta}\rmd \theta-\int_{0}^{\uppi}\rme^{-\rmi(m-1)\theta}\rmd \theta\right]\\
&=\frac{1}{\uppi(m-1)}(\rme^{\rmi(m-1)\uppi}-1+\rme^{-\rmi(m-1)\uppi}-1),
\end{align*}
which equals $0$ if $m$ is odd and 
equals $-\frac{4}{\uppi(m-1)}$ if $m$ is even.
\end{proof}

\begin{lemma}\label{lem:nonarchimedeaneisenstein}
We have
\[
\int_{Y_1}|1-y|^{-1}\widehat{\Theta}_p(y)|y|'^{-\frac{1}{2}}\rmd y=\frac12(1-p^{-1})^2\int_{ \A(\QQ_p)}\theta_p(t)\rmd t.
\]
\end{lemma}
\begin{proof}
Since $y\in Y_1$, we have $|y|'=|y|$. Also, by the definition of $\widehat{\Theta}_p(y)$ we have
\[
\int_{Y_1}|1-y|^{-1}\widehat{\Theta}_p(y)|y|'^{-\frac{1}{2}}\rmd y=\int_{Y_1}\int_{\QQ_p}\theta_p\left(z,\frac{z^2(1-y)}{4}\right)\frac{1}{|z||1-y||y|^{\frac12}}\rmd y\rmd z.
\]
Now we make change of variable $z\mapsto T$ and $z^2(1-y)/4\mapsto N$ so that
\[
y\mapsto 1-\frac{4N}{T^2}\quad\text{and}\quad z\mapsto T
\]
and
\[
\rmd y\wedge\rmd z=\frac{4}{T^2}\rmd T\wedge \rmd N.
\]
For $T,N\in \QQ_p^2$ with $T^2-4N\neq 0$, we denote $\gamma_{T,N}$ to be the regular element in $\G(\QQ_p)$ with trace $T$ and determinant $N$, up to conjugacy. $\gamma_{T,N}\in \A(\QQ_p)$ is equivalent to $y=(T^2-4N)/T^2\in Y_1$, we obtain
\begin{align*}
\int_{Y_1}|1-y|^{-1}\widehat{\Theta}_p(y)|y|'^{-\frac{1}{2}}\rmd y &=\int_{\gamma_{T,N}\in \A(\QQ_p)}\theta_p(T,N) \frac{\rmd T\rmd N}{|4N/T^2||T||1-4N/T^2|^{1/2}}\left|\frac{4}{T^2}\right|\\
&=\int_{\gamma_{T,N}\in \A(\QQ_p)}\frac{\theta_p(T,N)}{|T^2-4N|^{1/2}}\frac{1}{|N|}\rmd T\rmd N.
\end{align*}
Finally, we make change of variable $T=x+y$ and $N=xy$. The map $(x,y)\mapsto (T,N)$ is $2:1$. Also we have
\[
\rmd T\wedge\rmd N=(\rmd x+\rmd y)\wedge (x\rmd y+y\rmd x)=(x-y)\rmd x\wedge\rmd y.
\]
Hence
\begin{align*}
\int_{Y_1}|1-y|^{-1}\widehat{\Theta}_p(y)|y|'^{-\frac{1}{2}}\rmd y &=\frac12\int_{ \A(\QQ_p)}\frac{\theta_p(t)}{|x-y|}\frac{1}{|xy|}|x-y|\rmd x\rmd y=\frac12(1-p^{-1})^2\int_{ \A(\QQ_p)}\theta_p(t)\rmd t,
\end{align*}
where we used the fact that $\rmd t$ on $\A(\QQ_p)=\QQ_p^\times\times\QQ_p^\times$ is 
\[
\left(1-\frac{1}{p}\right)^{-2}\frac{\rmd x\rmd y}{|xy|}.\qedhere
\]
\end{proof}

\begin{proposition}
If $f_\infty=f_{\infty,m}$, then the constants $\mf{A},\mf{B},\mf{C}$ and $\mf{D}$ can be determined as follows:
\begin{enumerate}[itemsep=0pt,parsep=0pt,topsep=0pt, leftmargin=0pt,labelsep=2.5pt,itemindent=15pt,label=\upshape{(\arabic*)}]
  \item $\mf{A}=\mf{B}=\mf{D}=0$.
  \item If $m$ is odd, then $\mf{C}=0$. If $m$ is even, then
  \[
  \mf{C}=\frac{2}{m-1}\prod_{i=1}^{r}(1-q_i^{-1})^3\int_{\A(\QQ_{q_i})}\theta_{q_i}(t)\rmd t.
  \]
\end{enumerate}
\end{proposition}
\begin{proof}
By \autoref{prop:archimedeantrace} we have
\[
\int_{\RR}\frac{\widehat{\Theta}_\infty(x)}{|1-x|^{3/2}}\rmd x=\frac12\Tr(\triv(f_\infty))=\frac12\int_{Z_+\bs \G(\RR)}f_\infty(g)\rmd g.
\]
By Iwasawa decomposition, we have
\[
\int_{Z_+\bs \G(\RR)}f_\infty(g)\rmd g=\int_{Z_+\bs \A(\RR)}\int_{K}\int_{\N(\RR)}f(tnk)\rmd n\rmd k\rmd t,
\]
which vanishes by supercuspidality. Hence $\mf{A}=0$.

By \eqref{eq:discreteseriestheta} and \eqref{eq:thetarelation} we know that $\widehat{\Theta}_\infty(x)$ vanishes unless $x<0$. Hence $\mf{B}=\mf{D}=0$ and 
\[
\mf{C}=-2^{r-1}\prod_{i=1}^{r}(1-q_i^{-1})\uppi\int_{X_1}\int_{Y_{\mathbf{1}}} |1-x|_\infty^{-1} |1-y|_q^{-1}\widehat{\Theta}_\infty(x)\widehat{\Theta}_{q}(y) |x|_\infty^{-\frac{1}{2}}|y|_q'^{-\frac{1}{2}}\rmd x\rmd y.
\]  
By \autoref{lem:archimedeaneisenstein} and \autoref{lem:nonarchimedeaneisenstein} we obtain $\mf{C}=0$ if $m$ is odd, and
\[
\mf{C}=\frac{2}{m-1}\prod_{i=1}^{r}(1-q_i^{-1})^3\int_{\A(\QQ_{q_i})}\theta_{q_i}(t)\rmd t
\]
if $m$ is even.
\end{proof}

\begin{proposition}\label{prop:sigmasquarevanish}
Suppose that $f_\infty=f_{\infty,m}$. Then
\[
\sum_{\substack{n<X\\ \gcd(n,S)=1}}\Sigma^n(\square)=0.
\]
\end{proposition}

\begin{proof}
Recall that
\begin{align*}
\Sigma^n(\square)=&\sum_{\pm}\sum_{\nu\in \ZZ^r}q^{\nu/2}\sum_{\substack{T\in \ZZ^S\\ T^2\mp 4nq^\nu= \square}}\sum_{\substack{f\in \ZZ_{(S)}\\ f^2\mid T^2\mp 4nq^\nu}}  \sum_{k\in \ZZ_{(S)}^{>0}}\frac{1}{kf}\legendresymbol{(T^2\mp 4nq^\nu)/f^2}{k}\theta_\infty^\pm\legendresymbol{T}{2n^{1/2}q^{\nu/2}}\\
    \times &\prod_{i=1}^{r}\theta_{q_i}(T,\pm nq^\nu)\left[F\legendresymbol{kf^2}{|T^2\mp 4nq^\nu|_{\infty,q}'^{\vartheta}}+\frac{kf^2}{\sqrt{|T^2\mp 4nq^\nu|_{\infty,q}'}}V\legendresymbol{kf^2}{|T^2\mp 4nq^\nu|_{\infty,q}'^{\vartheta'}}\right].
\end{align*}

Suppose that $T^2\mp 4nq^\nu$ is a square. If $T^2\mp 4nq^\nu\neq 0$, then the element $\gamma$ with $\Tr\gamma=T$ and $\det\gamma=\pm nq^\nu$ is regular hyperbolic. Hence by \eqref{eq:discreteseriestheta} and recall that $\theta_\infty^-(x)=0$, we obtain
\[
\theta_\infty^\pm\legendresymbol{T}{2n^{1/2}q^{\nu/2}}=0.
\]
Hence the term for $T$ has no contribution. If $T^2\mp 4nq^\nu=0$, the above two equations still hold by continuity. Hence we obtain our result.
\end{proof}

Recall that $\Sigma^n(\xi)=I_\el(f^n)+\Sigma^n(\square)$. Hence by \autoref{thm:contributeellipticpoisson} we obtain
\begin{theorem}\label{thm:contributeellipticmodular}
For any $\varepsilon>0$, we have
\[
\sum_{\substack{n<X\\ \gcd(n,S)=1}} I_\el(f^n)\ll X^{\frac78+\varepsilon}
\]
if $m$ is odd and
\[
\sum_{\substack{n<X\\ \gcd(n,S)=1}} I_\el(f^n)=2\prod_{i=1}^{r}(1-q_i^{-1})^3\int_{\A(\QQ_{q_i})}\theta_{q_i}(t)\rmd t \frac{X}{m-1}+O(X^{\frac78+\varepsilon})
\]
if $m$ is even, where $\A$ denotes the diagonal torus of $\G$. The implied constants only depend on $m$, $f_{q_i}$ and $\varepsilon$.
\end{theorem}

Combining \autoref{prop:contributeidentitymodular}, \autoref{prop:contributeunipotentmodular}, \autoref{thm:contributehyperbolicmodular} and \autoref{thm:contributeellipticmodular} together, and note that the main terms of the elliptic term and the hyperbolic term cancel, we obtain
\begin{theorem}\label{thm:contributemodular}
For any $\varepsilon>0$, we have
\[
\sum_{\substack{n<X\\ \gcd(n,S)=1}} I_{\mathrm{cusp}}(f^n)=\sum_{\substack{n<X\\ \gcd(n,S)=1}} J_{\mathrm{geom}}(f^n)\ll X^{\frac78+\varepsilon},
\]
where the implied constants only depend on $m$, $f_{q_i}$ and $\varepsilon$. In particular, \eqref{eq:modularformgeometric} holds.
\end{theorem}

\subsection{Traces of Hecke operators}
In this section we prove \autoref{thm:modularformestimate}.
Let $m>2$ be an integer and let $N\in \ZZ_{>0}$, $n\in \ZZ_{>0}$ such that $\gcd(n,N)=1$. Let $\omega'$ be a Dirichlet character modulo $N$, and let $\omega=\bigotimes'_{v\in \mf{S}} \omega_v$ be the corresponding Hecke character on $\QQ^\times\bs \AA^\times$. Let $\overline{\bff}^n=\bigotimes'_{v\in \mf{S}} \overline{\bff}_v^n\in C_c^\infty(\G(\AA),\omega^{-1})$, that is, $\overline{\bff}^n$ is a smooth function on $\G(\AA)$, compactly supported modulo center, and that $f(zg)=\omega^{-1}(z)f(g)$ for all $z\in \Z(\AA)$ and $g\in \G(\AA)$. 

Recall that for any prime number $p$ and $r\in \ZZ_{\geq 0}$,
\[
\cX_p^{r}=\{X\in \M_2(\ZZ_p)\,|\, |\det X|_p = p^{-r}\}.
\]
Moreover, for any prime number $p$ and $r\in \ZZ_{\geq 1}$, we define
\[
\cI_p^{r}=\left\{X=\begin{pmatrix}
                        a & b \\
                        c & d 
                      \end{pmatrix}\in \GL_2(\ZZ_p)\,\middle|\, c\in p^r\ZZ_p\right\}.
\]
For example, if $r=1$, $\cI_p^{r}$ is the standard Iwahori subgroup.

Now $\overline{\bff}^n=\bigotimes_{v\in \mf{S}} \overline{\bff}_v^n$ is defined as follows:
\begin{enumerate}[itemsep=0pt,parsep=0pt,topsep=2pt,leftmargin=0pt,labelsep=3pt,itemindent=9pt,label=\textbullet]
  \item If $v=p\nmid N$, we define $\overline{\bff}_p^{n}$ to be a function supported on $\Z(\QQ_p)\cX_p^{n_p}$ such that $\overline{\bff}_p^{n}(zk)=p^{-n_p/2}\omega_p(z)^{-1}$, where $n_p=v_p(n)$, $z\in \Z(\QQ_p)$ and $k\in \cX_p^{v_p(n)}$, and $\omega_p(z)=\omega_p(a)$ if $z=(\begin{smallmatrix}
                                               a &  \\
                                                & a 
                                             \end{smallmatrix})$.
  \item If $v=q\mid N$, we define $\overline{\bff}_{q}^{n}=\overline{\bff}_q$ to be a function supported on $\Z(\QQ_q)\cI_q^{v_q(N)}$ such that $\overline{\bff}_q^n(zk)=\omega_q(z)^{-1}\omega_q(k)^{-1}\psi_q(N)$, where $z\in \Z(\QQ_q)$ and $k\in\cI_q^{v_q(N)}$, with  $\omega_q(k)=\omega_q(d)$ if $k=(\begin{smallmatrix} a & b \\c & d \end{smallmatrix})$, and $\psi_q(N)=N_{(q)}(1+1/q)$.
  \item If $v=\infty$, we define $\overline{\bff}_\infty^{n}=f_{\infty,m}$ as in the introductory section.
\end{enumerate}

For $\gcd(n,S)=1$, let $T_{m,N,\omega'}(n)$ denote the normalized $n^{\mathrm{th}}$ Hecke operator acting on $S_m(\Gamma_{0}(N),\omega')$, the space of holomorphic cusp forms of weight $m$ and level $N$ with nebentypus $\omega'$. More precisely, for $f\in S_m(\Gamma_{0}(N),\omega')$, we define
\[
T_{m,N,\omega'}(n)(f)(z)=n^{\frac{m-1}{2}}\sum_{\substack{ad=n\\ a>0\\ b\bmod d }}\omega'(a)^{-1}d^{-m}f\legendresymbol{az+b}{d}.
\]
Note that this operator differs from \cite{knightly2006traces} by $n^{-\frac{m-1}{2}}$.

There is also an Arthur-Selberg trace formula for $f\in C_c^\infty(\G(\AA),\omega^{-1})$, which is also of the form \eqref{eq:traceformulagl2}. We refer to \cite[Section 2]{knightly2006traces} for the trace formula in this case.
For $\overline{\bff}=\overline{\bff}^n$ above, we obtain the classical Eichler-Selberg trace formula. Moreover, we have the following theorem proved in \cite[Corollary 13.15]{knightly2006traces}:

\begin{theorem}\label{thm:eichlertraceformula}
We have
\[
I_{\mathrm{cusp}}(\overline{\bff}^n)=J_{\mathrm{spec}}(\overline{\bff}^n)=\Tr(T_{m,N,\omega'}(n)).
\]
\end{theorem}
Note that since we chose a different normalization of the test function and the Hecke operator, the factor $n^{m/2-1}$ in \cite{knightly2006traces} is $1$ in our case.

%By \eqref{eq:traceformulagl2} we obtain
%\begin{equation}\label{eq:eichlertrace}
%\Tr(T_{m,N,\omega'}(n))=I_\id(\overline{\bff}^n)+I_\el(\overline{\bff}^n)+J_\hyp(\overline{\bff}^n) +J_\unip(\overline{\bff}^n).
%\end{equation}

Now we define corresponding functions $\bff=\bff^n$ in this case.
\begin{enumerate}[itemsep=0pt,parsep=0pt,topsep=2pt,leftmargin=0pt,labelsep=3pt,itemindent=9pt,label=\textbullet]
  \item If $v=p\nmid N$, we define $\bff_p^{n}$ to be $p^{-n_p/2}$ times the characteristic function on $\cX_p^{n_p}$.
  \item If $v=q\mid N$, we define $\bff_{q}^{n}=\bff_q$ to be a function supported on $\cI_q^{v_q(N)}$ such that $\bff_q^n(k)=\omega_q(k)^{-1}\psi_q(N)$ for $k\in\cI_q^{v_q(N)}$.
  \item If $v=\infty$, we define $\bff_\infty^{n}=f_{\infty,m}$ as in the introductory section.
\end{enumerate}

\begin{lemma}\label{lem:square}
Let $\gamma\in \G(\QQ)$. Then $\bff^n(\gamma)\neq 0$ only if $\det\gamma=n$. $\overline{\bff}^n(\gamma)\neq 0$ only if $\det\gamma/n$ is a square.
\end{lemma}
\begin{proof}
$\bff^n(\gamma)\neq 0$ only if $\bff_\infty^n(\gamma)\neq 0$ and $\bff_\ell^n(\gamma)\neq 0$ for all prime $\ell$. Since $\bff_\infty^n$ is supported on $\GL_2^{\!+}(\RR)$, we have $\det\gamma>0$. For any $p\nmid N$, $\bff_p^n(\gamma)\neq 0$ implies that $\mathopen{|}\det\gamma\mathclose{|}_p=p^{-n_p}$, and for any $q\mid N$, $\bff_q^n(\gamma)\neq 0$ implies that $\mathopen{|}\det\gamma\mathclose{|}_q=1$. Hence we must have $\det\gamma=n$.

$\overline{\bff}^n(\gamma)\neq 0$ only if $\overline{\bff}_\infty^n(\gamma)\neq 0$ and $\overline{\bff}_\ell^n(\gamma)\neq 0$ for all prime $\ell$. Since $\overline{\bff}_\infty^n$ is supported on $\GL_2^{\!+}(\RR)$, we have $\det\gamma/n>0$. For any $p\nmid N$, $\overline{\bff}_p^n(\gamma)\neq 0$ implies that $\gamma\in \Z(\QQ_p)\cX_p^{n_p}$. Thus
$\det\gamma/p^{-n_p}\in (\QQ_p^\times)^2\ZZ_p$. Hence $v_p(\det\gamma/n)$ is even. Similarly, $v_q(\det\gamma/n)$ is even if $q\mid N$. Thus $\det\gamma/n$ is a square.
\end{proof}

\begin{proposition}\label{prop:identityequal}
We have
\[
I_\id(\bff^n)=2I_\id(\overline{\bff}^n).
\]
\end{proposition}

\begin{proof}
By the Arthur-Selberg formula for $C_c^\infty(\G(\AA),\omega^{-1})$, we have
\[
I_\id(\overline{\bff}^n)=\vol(\overline{\G}(\QQ)\bs \overline{\G}(\AA))\overline{\bff}^n(I),
\]
where $\overline{\G}$ denotes $\G$ modulo the center and $I$ denotes the identity matrix. We have
\[
\G(\QQ)\bs\G(\AA)^1=Z_+\G(\QQ)\bs\G(\AA)=\Z(\AA)\G(\QQ)\bs \G(\AA)\times Z_+\Z(\QQ)\bs \Z(\AA).
\]
The first term in the product is $\overline{\G}(\QQ)\bs \overline{\G}(\AA)$ and the second term has volume $1$. Hence
\[
\vol(\G(\QQ)\bs\G(\AA)^1)=\vol(\overline{\G}(\QQ)\bs \overline{\G}(\AA)).
\]

$\overline{\bff}^n(1)\neq 0$ only if $n$ is a square by \autoref{lem:square}. If it is in this case, then for $p\nmid N$,
\[
I=\begin{pmatrix}
    n^{-1/2} & 0 \\
    0 & n^{-1/2} 
  \end{pmatrix}\begin{pmatrix}
    n^{1/2} & 0 \\
    0 & n^{1/2} 
  \end{pmatrix}\in \Z(\QQ_p)\cX_p^{n_p}.
\]
Since $\gcd(n,N)=1$ and $\omega$ is unramified at $p$ for $p\nmid N$, we obtain $\overline{\bff}_p^n(1)=p^{-n_p/2}$. Since $1\in \cI_q^{v_q(N)}$, we have $\overline{\bff}_q^n(1)=\psi_q(N)$ for $q\mid N$. Thus
\[
\overline{\bff}^n(1)=\prod_{p\nmid N}p^{-n_p/2}\prod_{q\mid N}\psi_q(N)=n^{-1/2}\psi(N),
\]
where 
\[
\psi(N)=N\prod_{q\mid N}\left(1+\frac1q\right).
\]
Thus
\[
I_\id(\overline{\bff}^n)=\begin{cases}
                           \vol(\G(\QQ)\bs\G(\AA)^1)n^{-1/2}\psi(N)f_{\infty,m}(\sqrt{n}I), & \text{if $n$ is a square}, \\
                           0, & \text{otherwise}.
                         \end{cases}
\]

Next we consider $I_\id(\bff^n)$. By definition, it equals
\[
I_\id(\bff^n)=\sum_{z\in \Z(\QQ)}\vol(\G(\QQ)\bs \G(\AA)^1)f^n(z).
\] 
If $n$ is not a square, then $I_\id(\bff^n)=0$  by \autoref{lem:square}. If $n$ is a square, then $z\in \Z(\QQ)$ contributing the sum are $z=\pm \sqrt{n}I$ by \autoref{lem:square} again. Moreover, each $z$ gives the contribution
\[
\vol(\G(\QQ)\bs\G(\AA)^1)n^{-1/2}\psi(N)f_{\infty,m}(\sqrt{n}I)
\]
by a similar computation, noting that $f_{\infty,m}$ is even. Hence we have
\[
I_\id(\bff^n)=2\begin{cases}
                           \vol(\G(\QQ)\bs\G(\AA)^1)n^{-1/2}\psi(N)f_{\infty,m}(\sqrt{n}I), & \text{if $n$ is a square}, \\
                           0, & \text{otherwise}
                         \end{cases}
\]
and thus the conclusion holds.
\end{proof}

\begin{proposition}\label{prop:unipotentequal}
We have
\[
J_\unip(\bff^n)=2J_\unip(\overline{\bff}^n).
\]
\end{proposition}

\begin{proof}
The unipotent term for $\overline{\bff}^n$ equals $\fp_{s=1}Z(s,\triv,F)$, 
where $\fp$ denotes the finite part, $Z$ denotes the zeta integral in Tate's thesis and for $t\in \AA$,
\[
F(t)=\int_{K}\overline{\bff}^n\left(k^{-1}\begin{pmatrix}
                                 1 & t \\
                                 0 & 1 
                               \end{pmatrix}k\right)\rmd k,
\]
where $K$ denotes the standard maximal compact subgroup of $\G(\AA)$.

Since conjugation does not affect the determinant, by \autoref{lem:square} this term vanishes if $n$ is not a square. If $n$ is a square, we have
\[
\overline{\bff}^n\left(\pm\sqrt{n}k^{-1}\begin{pmatrix}
                                 1 & t \\
                                 0 & 1 
                               \end{pmatrix}k\right)=\omega(\pm \sqrt{n})^{-1}\overline{\bff}^n\left(k^{-1}\begin{pmatrix}
                                 1 & t \\
                                 0 & 1 
                               \end{pmatrix}k\right)=\overline{\bff}^n\left(k^{-1}\begin{pmatrix}
                                 1 & t \\
                                 0 & 1 
                               \end{pmatrix}k\right).
\]
We claim that
\[
\overline{\bff}^n\left(\pm\sqrt{n} k^{-1}\begin{pmatrix}
                                 1 & t \\
                                 0 & 1 
                               \end{pmatrix}k\right) =\bff^n\left(\pm\sqrt{n} k^{-1}\begin{pmatrix}
                                 1 & t \\
                                 0 & 1 
                               \end{pmatrix}k\right).
\]
It suffices to show that for each place $v$, the two functions are equal. If $v=\infty$, they are both $f_{\infty,m}$ and hence they are equal.

For $v=p\nmid N$, it suffices to show that if
\[
\pm\sqrt{n} k^{-1}\begin{pmatrix}
                                 1 & t \\
                                 0 & 1 
                               \end{pmatrix}k\in \Z(\QQ_p)\cX_p^{n_p},
\]
then the matrix actually belongs to $\cX_p^{n_p}$. Suppose that $z\in \Z(\QQ_p)$ such that the above matrix belongs to $z\cX_p^{n_p}$, then by taking determinant on each side we obtain $\mathopen{|}\det z\mathclose{|}_p=1$. Hence $z\in \cK_p$. Hence the matrix actually belongs to $\cX_p^{n_p}$.

For $v=q\mid N$, it suffices to show that if
\[
\pm\sqrt{n} k^{-1}\begin{pmatrix}
                                 1 & t \\
                                 0 & 1 
                               \end{pmatrix}k\in \Z(\QQ_q)\cI_q^{v_q(N)},
\]
then the matrix actually belongs to $\cI_q^{v_q(N)}$. This is done by a similar argument.

Hence the claim holds and thus $F=F_{\pm\sqrt{n}I}$, where for $z\in \Z(\QQ)$ we define
\[
F_z(t)=\int_{K}f^n\left(k^{-1}z\begin{pmatrix}
                                 1 & t \\
                                 0 & 1 
                               \end{pmatrix}k\right)\rmd k.
\]
Therefore
\[
J_\unip(\overline{\bff}^n)=\begin{cases}
                           \fp_{s=1}Z(s,\triv,F_{\pm \sqrt{n}}), & \text{if $n$ is a square}, \\
                           0, & \text{otherwise}.
                         \end{cases}
\]

Also, we have
\[
J_\mathrm{unip}(\bff^n)=\sum_{z\in \Z(\QQ)}\fp_{s=1}Z(s,\triv,F_z).
\]
If $n$ is not a square, by \autoref{lem:square} this term vanishes. If $n$ is a square, the only contribution comes from $z=\pm \sqrt{n}I$. Hence it is immediate that $J_\unip(\bff^n)=2J_\unip(\overline{\bff}^n)$.
\end{proof}

\begin{proposition}\label{prop:hyperbolicequal}
We have
\[
J_\hyp(\bff^n)=2J_\hyp(\overline{\bff}^n).
\]
\end{proposition}

\begin{proof}
The hyperbolic part in  Arthur-Selberg formula for $\overline{\bff}^n\in C_c^\infty(\G(\AA),\omega^{-1})$ is
\[
J_\hyp(\overline{\bff}^n)=-\frac{1}{2}\sum_{\gamma\in \overline{\A}(\QQ)_{\reg}}\int_{\overline{\A}(\AA)\bs \overline{\G}(\AA)}\overline{\bff}^n(g^{-1}\gamma g)\alpha(H_\B(wg)+H_\B(g))\rmd g,
\]
where $\A$ is the diagonal torus, $\overline{\A}$ denotes $\A$ modulo the center $\Z$ of $\G$, $w$ is the nontrivial element in the Weyl group of $(\G,\A)$, $\alpha$ denotes the simple root, and $H_\B$ denotes the Harish-Chandra map. 

Clearly we have $\overline{\A}(\AA)\bs \overline{\G}(\AA)=\A(\AA)\bs \G(\AA)$. By \autoref{lem:square}, $\gamma\in \overline{\A}(\QQ)_{\reg}$ contributes the sum above only if $\det\gamma/n$ is a square and we \emph{assume} that $\det\gamma=n$ in the summation \cite[Section 16]{knightly2006traces}.

We claim that the canonical map 
\[
\{\gamma\in \A(\QQ)_{\reg}\,|\,\det\gamma=n\}\to \{\gamma\in \overline{\A}(\QQ)_{\reg}\,|\, \det\gamma=n\}
\]
is surjective and the preimage of $\gamma$ is $\pm \gamma$. Clearly it is surjective. For any $\det\gamma=n$ and $z\in \Z(\QQ)$ with $\det (z\gamma)=n$, we have $\det z=1$. Hence $z=\pm I$. Thus the preimage of $\gamma$ is $\pm \gamma$.

Also, by a same argument in \autoref{prop:unipotentequal}, we have $\overline{\bff}^n(g^{-1}\gamma g)=\bff^n(\pm g^{-1}\gamma g)$. Thus we obtain
\[
J_\hyp(\overline{\bff}^n)=-\frac12\times\frac{1}{2}\sum_{\gamma\in \A(\QQ)_{\reg}}\int_{\A(\AA)\bs \G(\AA)}\bff^n(g^{-1}\gamma g)\alpha(H_\B(wg)+H_\B(g))\rmd g,
\]
which is $J_\hyp(\bff^n)/2$ by definition of $J_\hyp(\bff^n)$.
\end{proof}

\begin{proposition}\label{prop:ellipticequal}
We have
\[
I_\el(\bff^n)=2I_\el(\overline{\bff}^n).
\]
\end{proposition}

\begin{proof}
The elliptic part in Arthur-Selberg formula for $\overline{\bff}^n\in C_c^\infty(\G(\AA),\omega^{-1})$ is
\[
I_\el(\overline{\bff}^n)=\sum_{\gamma\in \overline{\G}(\QQ)_{\el}^\#}\vol(\overline{\G}_\gamma(\QQ)\bs\overline{\G}_\gamma(\AA))\int_{\overline{\G}_\gamma(\AA)\bs \overline{\G}(\AA)}\overline{\bff}^n(g^{-1}\gamma g)\rmd g,
\]
where $\overline{\G}(\QQ)_{\el}^\#$ denotes the set of elliptic conjugacy classes in $\overline{\G}(\QQ)$, and $\overline{\G}_\gamma$ denotes the centralizer of $\gamma$ in $\overline{\G}$. Clearly we have $\overline{\G}_\gamma=\overline{\G_\gamma}$, where $\overline{\G_\gamma}$ denotes $\G_\gamma$ modulo $\Z$. We have
\[
\G_\gamma(\QQ)\bs\G_\gamma(\AA)^1=Z_+\G_\gamma(\QQ)\bs\G_\gamma(\AA)=\Z(\AA)\G_\gamma(\QQ)\bs \G_\gamma(\AA)\times Z_+\Z(\QQ)\bs \Z(\AA).
\]
The first term in the product is $\overline{\G}_\gamma(\QQ)\bs\overline{\G}_\gamma(\AA)$ and the second term has volume $1$. Therefore
\[
\vol(\G_\gamma(\QQ)\bs\G_\gamma(\AA)^1)=\vol(\overline{\G}_\gamma(\QQ)\bs\overline{\G}_\gamma(\AA)).
\]
We have $\overline{\G}_\gamma(\AA)\bs \overline{\G}(\AA)=\G_\gamma(\AA)\bs \G(\AA)$. By \autoref{lem:square}, $\gamma\in \overline{\G}_\gamma(\QQ)_{\reg}$ contributes the sum above only if $\det\gamma/n$ is a square and we assume that $\det\gamma=n$ in the summation. Then we have a surjective map
\[
\{\gamma\in \G(\QQ)_{\el}^\#\,|\,\det\gamma=n\}\to \{\gamma\in \overline{\G}(\QQ)_{\el}^\#\,|\, \det\gamma=n\}
\]
and the preimage of $\gamma$ is $\pm \gamma$.

Also, by a same argument in \autoref{prop:unipotentequal}, we have $\overline{\bff}^n(g^{-1}\gamma g)=\bff^n(\pm g^{-1}\gamma g)$. Thus we obtain
\[
I_\el(\overline{\bff}^n)=\frac12\sum_{\gamma\in \G(\QQ)_{\el}^\#}\vol(\G_\gamma(\QQ)\bs\G_\gamma(\AA)^1)\int_{\G_\gamma(\AA)\bs \G(\AA)}\bff^n(g^{-1}\gamma g)\rmd g,
\]
which is $I_\el(\bff^n)/2$ by definition of $I_\el(\bff^n)$.
\end{proof}

Combining \autoref{prop:identityequal}, \autoref{prop:unipotentequal}, \autoref{prop:hyperbolicequal} and \autoref{prop:ellipticequal} and \eqref{eq:geometric}, we obtain
\[
J_{\mathrm{geom}}(\bff^n)=2J_{\mathrm{geom}}(\overline{\bff}^n).
\]
Hence by \eqref{eq:traceformulagl2}, \autoref{thm:eichlertraceformula} and \autoref{prop:cuspidalimage} we obtain
\begin{theorem}
For all $n$ with $\gcd(n,S)=1$, we have
\[
I_\cusp(\bff^n)=2I_\cusp(\overline{\bff}^n).
\]
\end{theorem}

Finally, by combining the above theorem and \autoref{thm:contributemodular} we obtain \autoref{thm:modularformestimate}.

\bibliography{ref.bib}
\bibliographystyle{amsalpha}

\end{document}